\documentclass[11pt]{article}
\usepackage[margin=0.9in]{geometry}

\RequirePackage[OT1]{fontenc}
\RequirePackage{amsthm,amsmath}
\RequirePackage[numbers]{natbib}
\RequirePackage[colorlinks,citecolor=blue,urlcolor=blue]{hyperref}

\usepackage{times}
\usepackage{bbm}
\usepackage{amssymb}
\usepackage{graphicx}
\usepackage{multirow}

\usepackage{latexsym, epsfig, amssymb, amsmath, amsthm, graphicx, mathrsfs,color}
\usepackage{anysize,natbib}
\usepackage[OT1]{fontenc}
\usepackage{multirow}

\marginsize{1.2in}{1in}{0.55in}{1.55in} %
\makeatletter

\usepackage{moreverb}

\usepackage{natbib}

\usepackage{animate}
\usepackage{multimedia}
\usepackage{bbm}
\usepackage{amsmath,amsthm}
\usepackage{subfigure}
\usepackage{graphicx}

\newcommand\BibTeX{{\rmfamily B\kern-.05em \textsc{i\kern-.025em b}\kern-.08em
T\kern-.1667em\lower.7ex\hbox{E}\kern-.125emX}}

\numberwithin{equation}{section}
\theoremstyle{plain}

\newtheorem*{proposition*}{Proposition}
\newtheorem{theorem}{Theorem}
\newtheorem*{theorem*}{Theorem}
\newtheorem{corollary}[theorem]{Corollary}
\newtheorem{lemma}{Lemma}

\newcommand{\B}{\boldsymbol}
\newcommand{\diag}{\mathrm{diag}}
\newcommand{\M}{\mathbf}
\newcommand{\df}{\emph{df}~}
\DeclareMathOperator*{\argmin}{arg\,min}

\makeatother

\def\boxit#1{\vbox{\hrule\hbox{\vrule\kern6pt\vbox{\kern6pt#1\kern6pt}\kern6pt\vrule}\hrule}}

\begin{document}

\title{Computing the degrees of freedom of rank-regularized estimators and cousins}

\author{
Rahul Mazumder  \vspace{0.2cm} \\
MIT Sloan School of Management \\
Massachusetts Institute of Technology, Cambridge, MA 02142  \vspace{0.4cm} \\
Haolei Weng  \vspace{0.2cm} \\
Department of Statistics and Probability \\
Michigan State University, East Lansing, MI 48824 
}

\date{}
\maketitle

\begin{abstract}
Estimating a low rank matrix from its linear measurements is a problem of central importance in contemporary statistical analysis. The choice of tuning parameters for estimators remains an important challenge from a theoretical and practical perspective. To this end, Stein's Unbiased Risk Estimate (SURE) framework provides a well-grounded statistical framework for degrees of freedom estimation. In this paper, we use the SURE framework to obtain degrees of freedom estimates for a general class of spectral regularized matrix estimators, generalizing beyond the class of estimators that have been studied thus far. To this end, we use a result due to Shapiro (2002) pertaining to the differentiability of symmetric matrix valued functions, developed in the context of semidefinite optimization algorithms. We rigorously verify the applicability of Stein's lemma towards the derivation of degrees of freedom estimates; and also present new techniques based on Gaussian convolution to estimate the degrees of freedom of a class of spectral estimators to which Stein's lemma is not directly applicable. 
\end{abstract}


\section{Introduction}\label{sec1}

Consider the linear model setup with 
$$ \M{y} = \B{\mu} + \B\epsilon, ~~ \text{Cov}(\B{\epsilon}) = \tau^2 \M{I},  ~~ \mathbb E (\B{\epsilon}) = \M{0},$$
where, we observe $\M{y}\in \mathbb{R}^n$, a noisy version of the signal $\B\mu\in \mathbb{R}^n$. Let $\hat{\B\mu}$ be an estimator of $\B\mu$. The accuracy of 
$\hat{\B\mu}$ as an estimator for $\B\mu$ is often quantified via the 
expected mean squared error (MSE) which admits the following decomposition \citep{efron2004least}
\begin{equation}\label{decomp-1-0}
R \triangleq \mathbb{E}\|\hat{\B \mu}-\B \mu \|^2_2=
-\tau^2n+\mathbb{E}\|\hat{\B \mu} -\M{y}\|^2_2+2\sum_{i=1}^n\text{Cov}(\hat{\mu}_i, y_i),
\end{equation}
where $\|\cdot\|_2$ is the usual $\ell_2$ norm, the subscript $i$ indicates the $i$th component of a vector. The covariance term appearing in~\eqref{decomp-1-0} measures the complexity of the estimator $\hat{\B\mu}$ and is related 
to the well known degrees of freedom (\df) of an estimator~\citep{stein1981estimation, efron2004least}:
\begin{eqnarray}
 \emph{df}(\hat{\B \mu})=\sum_{i=1}^n\text{Cov}(\hat{\mu}_i, y_i)/\tau^2. \label{original}
 \end{eqnarray}
The decomposition \eqref{decomp-1-0} suggests an unbiased estimator $\widehat{\df}(\hat{\B \mu})$ for $\emph{df}(\hat{\B \mu})$ that leads to an unbiased estimate for $R$:
 \begin{eqnarray}\label{unbiased:one}
 \widehat{R}=-\tau^2n+\|\hat{\B \mu} -\M{y}\|^2_2+2\tau^2 \cdot \widehat{\df}(\hat{\B \mu}).
 \end{eqnarray}
 We can then use $\widehat{R}$ to choose between different estimators. Hence the degrees of freedom plays an important role in model assessment and selection. Consider the example of multiple linear regression, where $\B\mu = X \B\beta$ with design matrix $X\in \mathbb{R}^{n\times p}$ and regression coefficient $\B \beta\in \mathbb{R}^p$. In the case when $n>p$ and $X$ is of full rank, the \df of the least square estimates equals $p$, i.e., the number of parameters in the model. This fact combined with \eqref{unbiased:one} leads to the well known Mallows's $C_p$ criterion \citep{mallows1973some}. For estimators $\hat{\B\mu}$ that are a linear functional of $\M{y}$ (arising via ridge regression, for example), the \emph{df} can be computed by looking at the trace of the smoother matrix \citep{friedman2001elements}. However, for estimators that are nonlinear functionals of $\M{y}$, the computation of \emph{df} becomes much more challenging. \cite{stein1981estimation,efron2004least} derive an alternate expression of \df for 
the Gaussian sequence model $\M{y}\sim N(\B{\mu}, \tau^2\M{I})$ when $\hat{\B \mu}$ is weakly differentiable with respect to $\M{y}$\footnote{There are additional mild integrability conditions about $\hat{\B \mu}$. Please refer to Appendix \ref{SURE:condition} for details. }. In this case, the degrees of freedom of $\hat{\B \mu}$ is given by the celebrated Stein's Lemma:
\begin{equation}
\text{(Stein's Lemma)}~~~~~~~~~~~~~\emph{df}(\hat{\B \mu})=\mathbb{E}~ \left(\sum_{i=1}^n \partial \hat{\mu}_i/ \partial y_i \right)~~~~~~~~~~~~~~~ \label{alter}
\end{equation}
which suggests an unbiased estimate for $R$, termed Stein's Unbiased Risk Estimate (SURE):


\begin{eqnarray*}
\widehat{R}=-\tau^2n+\|\hat{\B \mu}-\M y\|^2_2+2\tau^2\cdot \sum_{i=1}^n\partial \hat{\mu}_i/\partial y_{i}.
\end{eqnarray*}
The SURE framework has been successfully utilized in different statistical problems. For instance, \cite{donoho1995adapting} derived the \df of soft thresholding in a wavelet shrinkage procedure. \cite{zou2007degrees, tibshirani2012degrees} studied the \df of lasso and generalized lasso fit. \cite{mazumder2011sparsenet, tibshirani2014degrees} obtained the \df of best subset selection under the linear regression model with orthogonal design.

The above framework also applies to matrix estimation --- here, data is of the form 
$y_{ij} = \mu_{ij} + \epsilon_{ij}$ for $i=1, \ldots, m$ and $j = 1, \ldots, n$. 
The general problem of low rank matrix estimation has been widely studied in the statistical community
in the context of multivariate linear regression \citep{anderson1951estimating, izenman1975reduced, yuan2007dimension} and matrix completion \citep{candes2009exact, mazumder2010spectral}, among others. There has been nice recent work on using SURE theory to derive the \df of low rank matrix estimators -- but the problem becomes quite challenging as one needs to deal with the differentiability properties of nonlinear functions of the spectrum and singular vectors of a matrix. 
\citet{CandesEtal2013} obtained the analytic expression of the divergence\footnote{See the formal definition in Section \ref{sec2}.} $\sum_{ij}\partial \hat{\mu}_{ij}/\partial y_{ij}$ for a singular value thresholding estimator -- they also rigorously verified sufficient conditions under which Stein's Lemma holds. 
\cite{mukherjee2015degrees,Yuan2011} derived expressions for the divergence of certain reduced rank and nuclear norm penalized estimators; but they do not formally establish 
if the regularity conditions sufficient for Stein's Lemma to hold, are satisfied.
To sum up, the challenge for deriving the \df of matrix estimators is three-fold. Firstly, it may be challenging to verify the regularity conditions required for \eqref{alter} to hold. A blind use of formula \eqref{alter} may lead to 
inaccurate \df calculation\footnote{For example, in the best subset selection procedure in linear regression, the formula does not hold and the \df estimate is not the number of nonzero regressors.}. Secondly, even when formula \eqref{alter} is available, it might be difficult to derive an analytical expression of $\sum_{ij}\partial \hat{\mu}_{ij}/\partial y_{ij}$, especially for matrix estimators that depend on the singular vectors/values of the observed matrix in a non-linear way. Thirdly, there are estimators for which Stein's Lemma is not readily applicable -- in these cases, new techniques may be necessary to derive \df estimates. Thusly motivated, in this paper, we aim to present a systematic study of two generic low rank matrix estimators, namely spectral regularized and rank constrained estimators---this includes, but is not limited to, all estimators studied in the three aforementioned works. Our contributions are summarized as:

\begin{itemize}
\item[(i)] We propose a framework to derive the analytic formula of $\sum_{ij}\partial \hat{\mu}_{ij}/\partial y_{ij}$ for general matrix estimators, by appealing to some fundamental results pertaining to differentiability of symmetric matrix valued functions due to Shapiro~\citep{Shapiro2002}; derived in the context of semidefinite optimization 
algorithms. The expressions for the \df~of several estimators are thus shown to follow as special cases.
\item[(ii)] For several matrix estimators where Stein's Lemma is not directly applicable, our derivation of the \df relies on using ideas from Gaussian convolution along with 
subtle limiting arguments that utilize the eigenvalue distribution of a real-valued central Wishart matrix. The techniques proposed in this paper may apply to a wider class of estimators, beyond what is studied herein.  
\item[(iii)] Our analysis covers a much wider range of low rank matrix estimators than what has been studied before, and we present a unified framework to address these problems. 
\end{itemize}

The remainder of the paper is organized as follows. We introduce the main theorem for calculating the divergence of matrix estimators in Section \ref{divergence:sec}. Sections \ref{sec3} and \ref{sec4} consist of multiple applications of the main theorem in deriving the degrees of freedom for various low rank matrix estimators; and spectral regularized estimators. Numerical experiments are performed in Section \ref{simulation:sec} to validate the derived \df formulas. We conclude the paper with a conclusion in Section \ref{discussion:sec}. All the proof is relegated to the appendix. 

\subsection{Notations}\label{sec2}
For a vector $\B a=(a_1,\ldots,a_n) \in \mathbb{R}^n$, we use the notation $\diag(\B a)$ to denote the $n\times n$ diagonal matrix with $i$th diagonal entry being $a_i$. For a real matrix $Y \in \mathbb{R}^{m \times n}$ (we assume, without loss of generality, $m\geq n$ throughout the paper), let its transpose be $Y'$ and its reduced singular value decomposition be $Y=U\diag(\B \sigma) V'$, where $U=(\B u_1,\ldots, \B u_n), V=(\B v_1,\ldots,\B v_n), \B \sigma=(\sigma_1,\ldots, \sigma_n)$ and $\sigma_1 \geq \dots \geq \sigma_n\geq 0$. We denote the Frobenius norm of $Y$ by $\|Y\|_F$. Unless otherwise stated, we use $Y=U\diag(\B \sigma) V'$ to represent the reduced singular value decomposition (SVD). $Y$ is called simple if it has no repeated singular values. For a real valued function $f : \mathbb{R}^+ \rightarrow \mathbb{R}$, define the associated matrix valued spectral function $S(\cdot;f): \mathbb{R}^{m\times n} \rightarrow \mathbb{R}^{m\times n} $ as $S(Y;f)=U\diag(f(\B \sigma))V'$ where $f(\B \sigma)=(f(\sigma_1),\ldots, f(\sigma_n))$. A function $f$ is said to be directionally differentiable at $x$ if the directional derivative 
\[
f'(x;h)\triangleq \underset{t \downarrow 0}{\lim}\frac{f(x+th)-f(x)}{t}
\]
exists for any $h$. Denote the divergence of $S(Y;f)$ by 
\[
\nabla \cdot S(Y;f)\triangleq \underset{ij}{\sum}~\partial [S(Y;f)]_{ij}/\partial Y_{ij},
\]
where $[S(Y;f)]_{ij}$ is the $(i,j)$th element of $S(Y;f)$. When we mention regularity conditions, we refer to the integrability and weak differentiability conditions that are required for \eqref{alter} to hold (see, for example, \citet{stein1981estimation,CandesEtal2013} for details).

\section{Computing the Divergence of Matrix Valued Spectral Functions}
\label{divergence:sec}
We present herein a framework to compute the \df~for matrix estimators of the form $S(Y;f)$. Towards this end, we will need to compute the divergence $\nabla \cdot S(Y;f)$, by making use of results due to~\cite{Shapiro2002}. For a symmetric matrix $X \in \mathbb{R}^{N \times N}$, let $\lambda_1(X) > \cdots > \lambda_q(X)$ be the set of its unique eigenvalues, $r_1, \ldots, r_q$ be the associated multiplicities, and $E_1(X) \in \mathbb{R}^{N \times r_1}, \ldots, E_q(X) \in \mathbb{R}^{N \times r_q}$ be the set of matrices whose columns are the corresponding orthonormal eigenvectors. For any given function $f : \mathbb{R} \rightarrow \mathbb{R}$, define the associated matrix valued function $F: \mathbb{R}^{N\times N} \rightarrow \mathbb{R}^{N \times N}$,
\begin{align}
\label{matrix:value:fun:def}
 F(X)=\sum_{k=1}^q f(\lambda_k(X))E_k(X)E_k(X)'. 
 \end{align}
 \cite{Shapiro2002} investigates differentiability properties of the function $F(X)$ in cases where $f(x)$ is directionally differentiable. His study is motivated by the works of \cite{sun2002semismooth, pang2003semismooth} on the semismoothness of $F(X)$ when $f(x)=|x|$ or $\max\{0,x\}$, which play important roles in algorithms for semidefinite programs and complementarity problems. For our purpose, we consider a special case of the directional differentiability property of $F(X)$ from \cite{Shapiro2002}. 
 
 Suppose $f$ is directionally differentiable at every point $\lambda_k(X), k=1, \ldots, q$. Then the directional derivative $f'(\lambda_k(X);h)$ exists for $\forall h\in \mathbb{R}$. Let $\Psi_k : \mathbb{R}^{r_k \times r_k} \rightarrow \mathbb{R}^{r_k \times r_k}$ be the associated matrix valued function defined through $f'(\lambda_k(X);\cdot)$. That is, for a given symmetric matrix $Y \in \mathbb{R}^{r_k\times r_k}$, 
 \[
 \Psi_k(Y)=\sum_{i} f'(\lambda_k(X); \lambda_i(Y))E_i(Y)E_i(Y)',
 \]
 where $\{\lambda_i(Y)\},\{E_i(Y)\}$ are the sets of unique eigenvalues and the corresponding orthonormal eigenvectors of $Y$, respectively. 
 
 
\begin{lemma}\label{shap}
\citep{Shapiro2002} Using the notation above, $F(X)$ is directionally differentiable at $X$ and its directional derivative $F'(X; H)$ is given by:
\begin{eqnarray} 
&&F'(X; H) =  \lim_{t \, \downarrow \, 0} \frac{F(X+tH)-F(X)}{t} \nonumber \\
& = & \frac{1}{2} \sum^q_{\substack{l \neq k \\ l,k=1}} \frac{f(\lambda_l(X))-f(\lambda_k(X))}{\lambda_l(X)-\lambda_k(X)}(E_lE_l' H E_kE_k' + E_kE_k' H E_lE_l')  \nonumber \\
&   & + \sum_{k=1}^q E_k[\Psi_k(E_k'HE_k)]E'_k \,, \label{shapiro}
\end{eqnarray}
where $H \in \mathbb{R}^{N\times N}$ is an arbitrary symmetric matrix, and $E_k$ denotes $E_k(X)$ for $k=1,2,\ldots, q$. 
\end{lemma}

Shapiro's result ensures that matrix valued functions inherit directional differentiability (at a matrix point $X$), from the real valued function $f(\cdot)$ (at all the distinct eigenvalues of $X$). We will present a generalization of Lemma~\ref{shap} to asymmetric matrices---this will be useful to address the differentiability properties of (rectangular) matrix valued spectral functions (see the definition in Section \ref{sec2}). Towards this end, we need the following lemma to connect between symmetric and asymmetric matrices.

\begin{lemma} \label{symmetry}
For any matrix $Y \in \mathbb{R}^{m \times n}$, consider the reduced singular value decomposition $Y=U\Sigma V'$ with $\Sigma \in \mathbb{R}^{n\times n}$. Thus, there exists $\bar{U} \in \mathbb{R}^{m \times (m-n)}$ such that $\bar{U}'\bar{U}=I\in\mathbb{R}^{(m-n)\times (m-n)}$ and $\bar{U}'U=0 \in \mathbb{R}^{(m-n)\times n}$. Define the matrices
\begin{eqnarray}\label{dfone}
Y^{\ast}=
\begin{bmatrix}
0 & Y \\
Y^T & 0
\end{bmatrix}
, \quad P=
\begin{bmatrix}
\frac{1}{\sqrt{2}}U & \frac{1}{\sqrt{2}}U & \bar{U} \\
\frac{1}{\sqrt{2}}V & \frac{-1}{\sqrt{2}}V & 0
\end{bmatrix}
\quad \text{and} \quad \Sigma^{\ast}=
\begin{bmatrix}
\Sigma & 0 & 0\\
 0& -\Sigma & 0 \\
0 & 0 & 0
\end{bmatrix}\,.
\end{eqnarray}
An eigendecomposition of  $Y^{\ast}$ is given by: $Y^{\ast}=P\Sigma^{\ast}P'$.
\end{lemma}

The relation between the singular value decomposition of a matrix $Y$ and the Schur decomposition of its symmetrized version $Y^*$ is a well known result in matrix-theory -- see \cite{golub2012matrix} for example. In our case, Lemma \ref{symmetry} provides a tool to study the directional differentiability of matrix valued spectral functions via Lemma \ref{shap}. In particular, for any given $S(Y;f)$, we can define a real valued function $f^*: \mathbb{R} \rightarrow \mathbb{R}$ as $f^*(x)=f(x)$ for $x\geq 0$ and $f^*(x)=-f(-x)$ otherwise. Let $Y^*$ be the matrix defined in Lemma \ref{symmetry} and $F^*(Y^*)$ be the matrix valued function associated with $f^*(x)$ as described in \eqref{matrix:value:fun:def}. Then Lemma \ref{symmetry} leads to
\begin{eqnarray*}
F^*(Y^*)=
\begin{bmatrix}
0 & S(Y;f) \\
S(Y;f)' &0
\end{bmatrix}.
\end{eqnarray*}
Hence the directional differentiability of $S(Y;f)$ can be analyzed by studying the symmetric matrix valued function $F^*(Y^*)$ through Lemma \ref{shap}. The divergence of $S(Y;f)$ can then be accordingly derived. 
The general formula for the divergence of matrix valued spectral functions is given in Corollary \ref{corcor1}; and the proof is presented in Appendix \ref{pf:divergence}.

\begin{corollary} \label{corcor1}
Given a matrix $Y\in \mathbb{R}^{m\times n}$ with singular values $\sigma_1 \geq \ldots \geq \sigma_n$, let $s_1> s_2 >\ldots > s_K \geq 0$ be the set of distinct singular values, $d_1,\ldots, d_K$ be the associated multiplicities. For any $f : \mathbb{R}^+ \rightarrow \mathbb{R}$ with $f(0)=0$, if it is differentiable at every point $s_i, 1 \leq i \leq K$, then
\begin{eqnarray*}
&&\hspace{-0.7cm} \sum_{i=1}^m\sum_{j=1}^n \frac{\partial [S(Y;f)]_{ij}}{\partial Y_{ij}}= \sum_{s_i >0}\left [  \frac{d_i(d_i+1)}{2}f'(s_i) + \left((m-n)d_i+\frac{d_i(d_i-1)}{2} \right) \frac{f(s_i)}{s_i}\right ] \label{thm1} \\
&&+ d_K(m-n+d_K)  f'(0)\mathbbm{1}(s_K=0)+ \sum_{1\leq i \neq j\leq K} d_id_j\frac{s_if(s_i)-s_jf(s_j)}{s_i^2-s_j^2}. \nonumber \\
\end{eqnarray*}
\end{corollary}


We remark that the differentiability condition on $f$ can be weakened to directional differentiability leading to a more complex divergence formula, as derived in Appendix \ref{pf:divergence}. We choose to present the streamlined version in Corollary \ref{corcor1} to improve the readability. The divergence expression in Corollary \ref{corcor1} originally appears in \cite{CandesEtal2013}. The authors first derive the formula for a matrix $Y$ which is simple and has full rank. Their derivation is based on standard techniques of computing the Jacobian of the SVD \citep{edelman2005matrix, papadopoulo2000estimating}. They then extend the result to general matrices. Here we show that the divergence formula can be derived as a consequence of Lemma~\ref{shap}, and can be generalized to a larger class of functions $f$. We should also mention that the differentiability properties of singular values of a rectangular matrix have been studied in \cite{lewis2005nonsmooth, lewis2005nonsmooth2, drusvyatskiy2015variational}. Those existing results are not applicable, because the current settings are concerned with matrix functions that involve both singular values and singular vectors.

\section{Degrees of Freedom for additive Gaussian models }\label{sec3}

We start by considering the canonical additive Gaussian model :
\begin{equation}
Y=M^{\ast}+\mathcal{E}, \label{model}
\end{equation}
where $Y\in \mathbb{R}^{m\times n}$ is the observed matrix, $M^{\ast} \in \mathbb{R}^{m\times n}$ is the underlying low rank matrix of interest, and $\mathcal{E}=(\epsilon_{ij})_{m\times n}$ is the random noise matrix with $\epsilon_{ij}\overset{\text{iid}}{\sim}N(0, \tau^2)$. 

\subsection{Estimators obtained via spectral regularization} \label{sec:srestimate}

A popular class of low rank matrix estimators are obtained through spectral regularization :
\begin{eqnarray}
S_{\theta}(Y) \in \argmin_{M\in \mathbb{R}^{m \times n}}~ \frac{1}{2}\|Y-M\|^2_F+\sum_{i=1}^n P_{\theta}(\sigma_i), \label{estimator}
\end{eqnarray}
where $\sigma_1\geq \sigma_2 \geq \dots \geq \sigma_n\geq 0$ are the singular values of $M$ and $P_{\theta}: \mathbb{R}^+\rightarrow \mathbb{R}^+$ is a family of sparsity promoting penalty functions indexed by $\theta$. For example, $P_{\theta}(x)=\theta x$ gives the \emph{nuclear norm} penalty. Some non-convex penalty functions include MC+ \citep{zhang2010nearly} and SCAD \citep{fan2001variable}. The optimization problem \eqref{estimator} is closely related to the following problem,
\begin{eqnarray}
s_{\theta}(\B \sigma) \in \argmin_{\B \alpha \in \mathbb{R}^n} \frac{1}{2}\|\B \sigma(Y) - \B \alpha \|^2_2+\sum_{i=1}^n P_{\theta}(\alpha_i),  \label{spectral}
\end{eqnarray}
where $s_{\theta}(\B \sigma)=(s_{\theta}(\sigma_1),\ldots, s_{\theta}(\sigma_n)), \B \alpha=(\alpha_1,\dots, \alpha_n) $, and $\B \sigma(Y)=(\sigma_1(Y),\ldots, \sigma_n(Y)) $ are the singular values of $Y$. Due to the separability in \eqref{spectral}, it is clear that $s_{\theta}(\cdot)$ is the proximal function induced by the penalty $P_{\theta}$:
\begin{align*}
s_{\theta}(u)\in \argmin_{x\in \mathbb{R}}~ \frac{1}{2}(x-u)^2+P_{\theta}(x).
\end{align*}
The problem \eqref{estimator} in fact admits a closed form solution (See Proposition 1 in \cite{mazumder2018matrix}): 
\[
S_{\theta}(Y)=U\diag(s_{\theta}(\B \sigma))V',
\]
where $Y=U\diag(\B \sigma) V'$ is the reduced SVD of $Y$. Since the penalty function $P_{\theta}(\cdot)$ shrinks some singular values to zero, it induces a low rank matrix estimator $S_{\theta}(Y)$. How can one determine the appropriate amount of shrinkage $\theta$? To this end, the following corollary presents 
SURE expressions for a variety of estimators. 

\begin{corollary} \label{cor2}
Consider the spectral regularized estimator $S_{\theta}(Y)$ in \eqref{estimator} under the model \eqref{model}. Assuming $P_{\theta}(\cdot)$ is differentiable on $(0, \infty)$ and $P_{\theta}(0)=0$, we introduce the following quantity ($\phi_P$) that
measures the amount of concavity of $P_{\theta}(\cdot)$:
\begin{equation*}
\phi_P := \; \inf_{\alpha, \alpha' >0} \;\;  \frac{P'_{\theta}(\alpha)  - P'_{\theta}(\alpha')}{\alpha - \alpha'},
\end{equation*}
where $P'_{\theta}(\alpha)$ denotes the derivative of $P_{\theta}(\alpha)$ wrt $\alpha$ on $\alpha >0$. Suppose $\phi_P+1>0$, then
\begin{eqnarray}
df(S_{\theta}(Y))=\mathbb{E} \Bigg [ \sum_{i=1}^n  \Big ( s'_{\theta}(\sigma_i)+(m-n)\frac{s_{\theta}(\sigma_i)}{\sigma_i} \Big)+2\sum_{\substack {i \neq j\\ i, j=1}}^n\frac{\sigma_is_{\theta}(\sigma_i)}{\sigma^2_i-\sigma^2_j} \Bigg ], \label{dfformula}
\end{eqnarray}
where $\sigma_1\geq \ldots \geq \sigma_n\geq 0$ are the singular values of $Y$.
\end{corollary}

The recent work \cite{hansen2018stein} has derived the same \df formula as in \eqref{dfformula}. However, the result in \cite{hansen2018stein} holds for a different class of matrix estimators from the one in Corollary \ref{cor2}. Specifically,  Theorem 1 in \cite{hansen2018stein} requires the function $s_{\theta}(\cdot)$ to be differentiable but allows different $s_{\theta}(\cdot)$ applied to each of the singular value. In contrast, Corollary \ref{cor2} assumes the same $s_{\theta}(\cdot)$  across the singular values, yet allows for non-differentiable $s_{\theta}(\cdot)$. We discuss a few examples below. 

The condition $\phi_P+1>0$ holds for many penalty functions. First of all, any convex function differentiable over $(0,\infty)$, has non-negative $\phi_P$. In particular, for $P_{\theta}(\alpha)=\theta |\alpha|$, it is straightforward to confirm that $s_{\theta}(\sigma)=(\sigma-\theta)_{+}$. This recovers the \df~formula of the singular value thresholding estimator studied in \cite{CandesEtal2013}. Moreover, some families of non-convex penalty functions satisfy $\phi_P+1>0$ as well. Examples include MC+  ($\gamma >1$) and SCAD ($a>2$), where $\gamma, a$ are tuning parameters associated with the two functions, respectively. See Section \ref{simulation:sec} for the explicit expressions. Non-convex penalties 
are well known to attenuate the estimation bias caused by convex sparsity promotion functions~\cite{fan2001variable,mazumder2011sparsenet,mazumder2018matrix}. Note that some popular non-convex penalties like $P_{\theta}(\alpha)=\theta |\alpha|^q~(0\leq q <1)$ do not satisfy the condition $\phi_P + 1>0$.  
 In particular, when $q=0$, $P_{\theta}(\alpha)=\theta \mathbbm{1}(\alpha \neq 0)$ gives the widely known rank regularized estimator
\[
S_{\theta}(Y)=\sum_{i=1}^n\sigma_i \mathbbm{1}(\sigma_i > \sqrt{2\theta})\B u_i\B v_i'. 
\]
Due to the hard thresholding rule on the singular values, $S_{\theta}(Y)$ is not a continuous function of $Y$, hence Stein's Lemma can not be directly applied. 
The following corollary (the proof is presented in Appendix \ref{pf:l0}) derives an expression for the degrees of freedom of the rank regularized estimator.

\begin{corollary} \label{cor3add}
Consider the rank regularized matrix estimator $S_{\theta}(Y)=\sum_{i=1}^n\sigma_i \mathbbm{1}(\sigma_i > \sqrt{2\theta})\B u_i\B v_i'$ under the model \eqref{model}, then
\begin{equation}\label{exp-corr3add}
\begin{aligned}
&&\hspace{-0.8cm}df(S_{\theta}(Y))=\mathbb{E} \sum_{\substack {i \neq j\\ i, j=1}}^n\Big( \frac{\sigma^2_i\mathbbm{1}(\sigma_i> \sqrt{2\theta})}{\sigma_i^2-\sigma^2_j}+ \frac{\sigma^2_j\mathbbm{1}(\sigma_j> \sqrt{2\theta})}{\sigma_j^2-\sigma^2_i}\Big )   \\
&&~~~~~+\sum_{i=1}^n \left [(m-n+1)P(\sigma_i >\sqrt{2\theta})+\sqrt{2\theta}f_{\sigma_i}(\sqrt{2\theta})\right ],
\end{aligned}
\end{equation}
where $f_{\sigma_i}(\cdot)$ is the marginal density function of $Y$'s $i$th singular value $\sigma_i$.
\end{corollary}

If we ignore the regularity conditions and use Equation \eqref{alter} directly, we will get an incorrect estimate of the \emph{df} --- specifically, the expression we obtain (by applying Corollary \ref{corcor1}) will not include the term $\sum_{i=1}^n\sqrt{2\theta}f_{\sigma_i}(\sqrt{2\theta})$ above.  
To arrive at~\eqref{exp-corr3add} we construct a sequence of matrix valued spectral functions (induced by MC+ penalty) which satisfy the conditions of Corollary \ref{cor2} and whose \emph{df} converges to the \df~of the rank regularized matrix estimator. We then combine the formula in Corollary \ref{cor2} with a careful limiting argument that hinges on the eigenvalue distribution of a central Wishart matrix to derive the \df~of the rank regularized estimator.

Note that when $P_{\theta}(\alpha)=\theta |\alpha|^q$ with $0<q<1$, problem \eqref{spectral} does not admit an explicit solution. Introducing the notation
\begin{align}
\label{etaq:def}
\eta_q(\sigma;\theta)=\argmin_{x \in \mathbb{R}}\frac{1}{2}|\sigma-x|^2+\theta |x|^q, 
\end{align}
we have $s_{\theta}(\sigma)=\eta_q(\sigma;\theta)$. According to Lemmas 5 and 6 in \cite{zheng2017lp}, the function $\eta_q(\sigma;\theta)$ has a 
jump discontinuity:
\begin{align}
&\eta_q(\sigma;\theta)=0, \mbox{~for~~} 0\leq \sigma < c_q \theta^{1/(2-q)}, \nonumber \\
& \eta_q(c_q \theta^{1/(2-q)};\theta)=[2(1-q)\theta]^{1/(2-q)}, \nonumber \\
&c_q=[2(1-q)]^{1/(2-q)}+q[2(1-q)]^{(q-1)/(2-q)}. \label{cq:constant}
\end{align}
Hence $s_{\theta}(\sigma)$ is not continuous. Similar to the 
rank regularized estimator, Stein's Lemma is not applicable to the case $P_{\theta}(\alpha)=\theta |\alpha|^q (0<q<1)$. We adapt the approach used in the proof of Corollary \ref{cor3add} to derive the \df for the case $0<q<1$ -- the result is presented in the following corollary, the proof of which is in Appendix \ref{pf:lq}.

\begin{corollary} \label{cor3done}
Consider the matrix estimator $S_{\theta}(Y)$ in \eqref{estimator} with $P_{\theta}(\alpha)=\theta |\alpha|^q (0<q<1)$ under the model \eqref{model}, then
\begin{eqnarray*}
&&\hspace{-0.8cm}df(S_{\theta}(Y))=\mathbb{E} \sum_{\substack {i \neq j\\ i, j=1}}^n \frac{\sigma_i\eta_q(\sigma_i;\theta)-\sigma_j\eta_q(\sigma_j;\theta)}{\sigma_i^2-\sigma^2_j}  \\
&&+\mathbb{E} \sum_{i=1}^n \left [(m-n)\frac{\eta_q(\sigma_i;\theta)}{\sigma_i}+\eta'_q(\sigma_i;\theta)+[2(1-q)\theta]^{1/(2-q)} f_{\sigma_i}(c_q\theta^{1/(2-q)})\right ],
\end{eqnarray*}
where $f_{\sigma_i}(\cdot)$ is the marginal density function of $Y$'s $i$th singular value $\sigma_i$; $\eta'_q(\sigma_i;\theta)$ is the partial derivative of $\eta_q(\sigma_i;\theta)$ with respect to $\sigma_i$. 
\end{corollary}
By a quick inspection, we can find that setting $q=1, 0$ in the \df formula of Corollary \ref{cor3done} recovers the \df formula for the case $q=1$ and $q=0$ which are already derived in Corollaries \ref{cor2} and \ref{cor3add} respectively. Hence Corollary \ref{cor3done} presents a unified \df~formula for the family $0\leq q \leq 1$. Moreover, it can be easily verified that the term $\sum_{i=1}^n[2(1-q)\theta]^{1/(2-q)} f_{\sigma_i}(c_q\theta^{1/(2-q)})$ in the above expression will be missed if we apply Stein's Lemma and Corollary \ref{corcor1} directly to derive the \df. This is further confirmed from the simulation results presented in Figure \ref{fig:one}. 
\begin{figure}[!htb]
\centering
\begin{tabular}{cc}
\includegraphics[width=4.5cm, height=4.cm,trim = 5mm 11mm  10mm 20mm, clip]{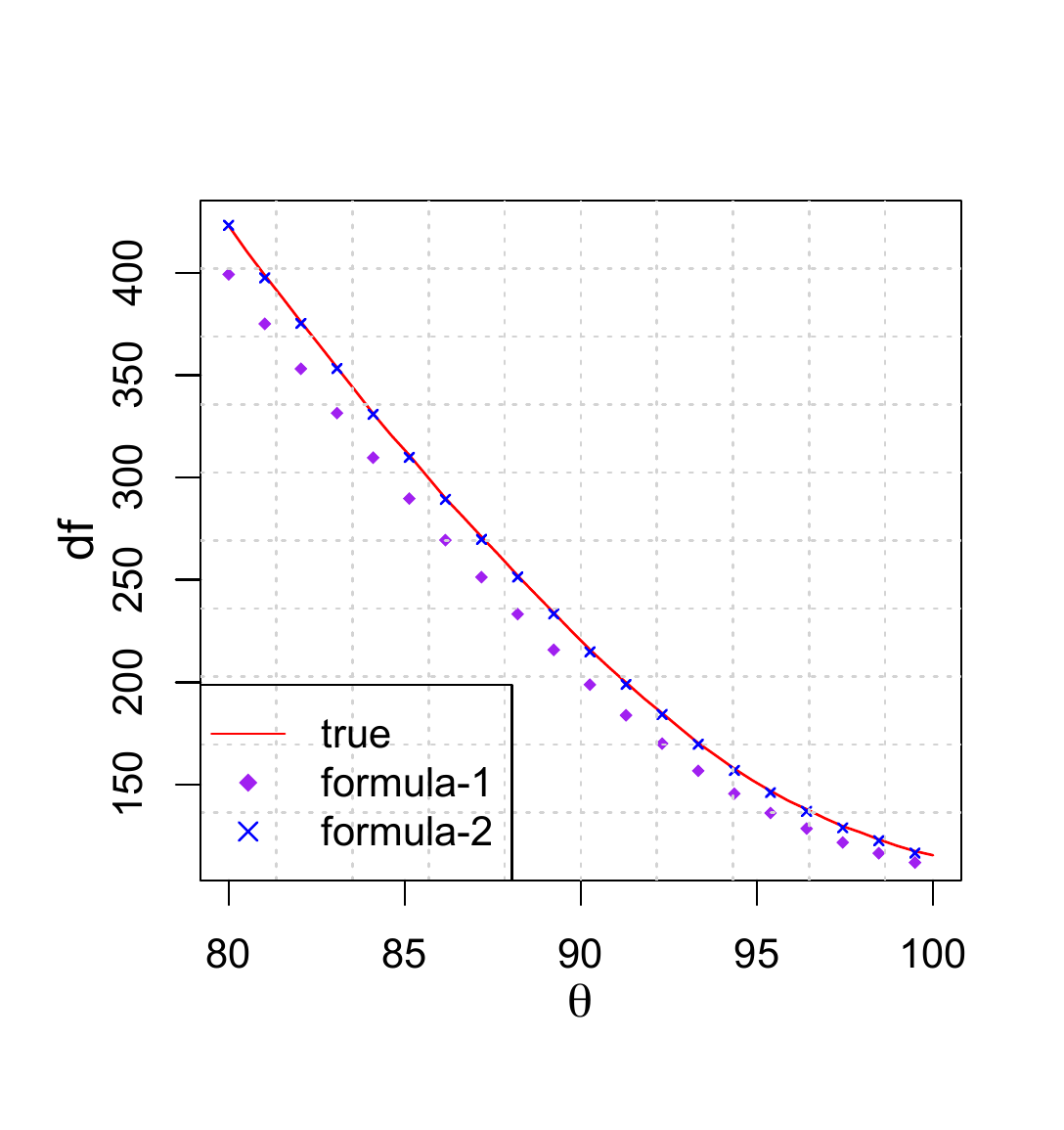} &
\includegraphics[width=4.5cm, height=4.cm, trim = 10mm 11mm  10mm 20mm, clip]{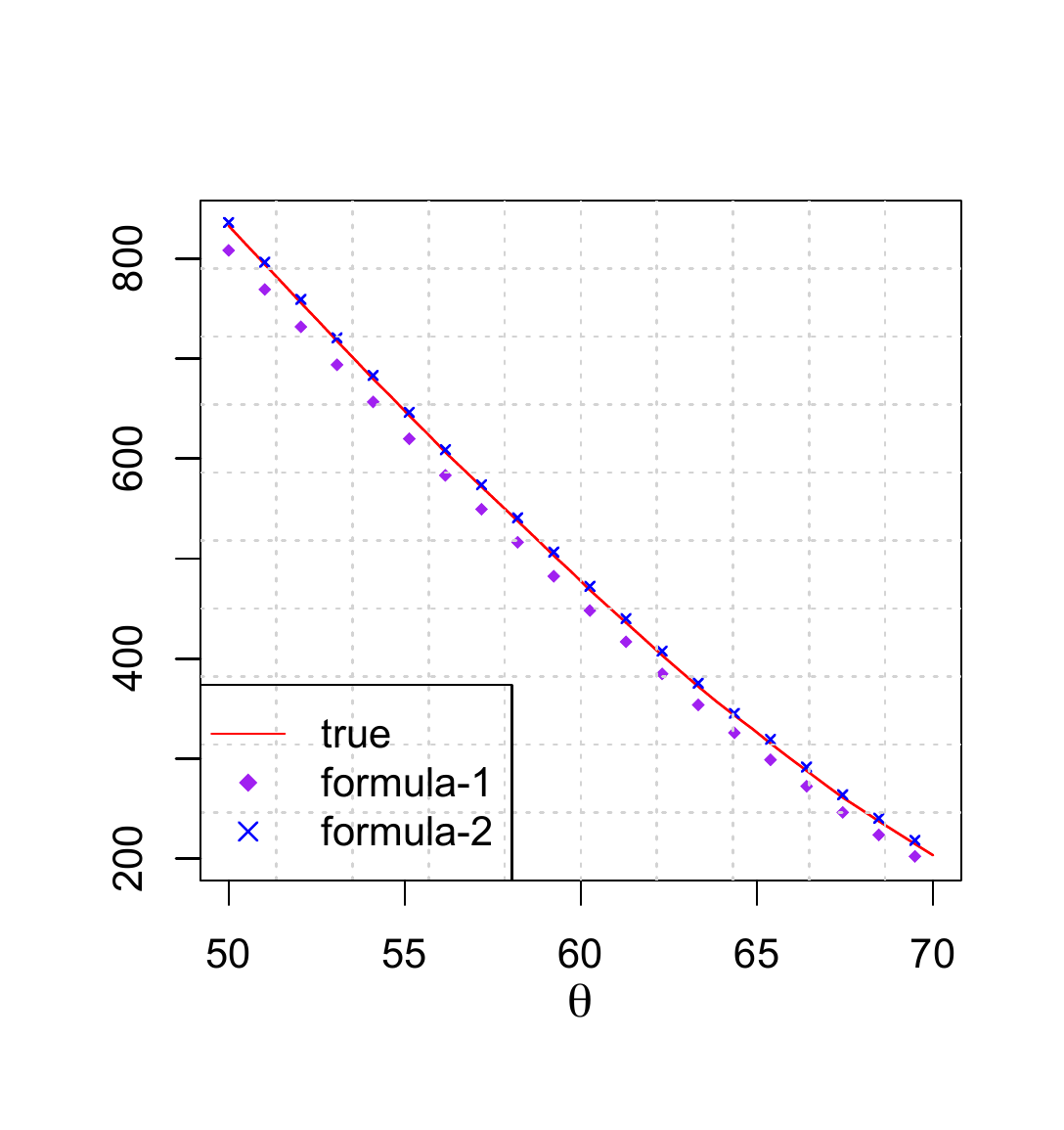}  
\end{tabular}
\vspace{-1em}
\caption{{\small{The df computation of $S_{\theta}(Y)$ with $P_{\theta}(\alpha)=\theta|\alpha|^q$ for $q=0$ (left), and $q=0.1$ (right). The true df (red curve) is computed according to the definition of df in \eqref{original}; the formula-1 (purple diamond) denotes the df obtained by using \eqref{alter} directly; the formula-2 (blue cross) represents the df derived from Corollary \ref{cor3done}. In this simulation, we set $n=m=50, M^*=5\M{1}\M{1}', \tau=1$. The df is calculated by monte carlo simulation over 10000 repetitions.}} } \label{fig:one}
\end{figure}
\vspace{-1em}
\subsection{Reduced rank estimators}
We now consider rank constrained estimators of the form:
\begin{eqnarray}
C_K(Y) \in \argmin_{\text{rank}(M)\leq K} \|Y-M\|^2_F, \label{rank}
\end{eqnarray}
for some positive integer $K\leq n$. The Eckart-Young Theorem \citep{eckart1936approximation} shows that 
\[
C_K(Y)=\sum_{i=1}^K\sigma_i \B u_i \B v_i',
\] 
where $\sigma_1\geq\cdots \geq \sigma_K$ are the largest $K$ singular values of $Y$, and $\{\B u_i, \B v_i\}_{i=1}^K$ are the corresponding singular vectors. Here, $K$ controls the amount of regularization. The choice of $K$ can be guided by an expression for the \emph{df} of $C_K(Y)$, as presented below.

\begin{corollary} \label{cor3}
For the reduced rank estimator $C_K(Y)$ in \eqref{rank} under the model \eqref{model}, 
\begin{eqnarray}
df(C_K(Y))=
\begin{cases}
\mathbb{E}\Big[     (m+n-K)K+2\sum_{i=1}^K\sum_{j=K+1}^n\frac{\sigma_j^2}{\sigma_i^2-\sigma_j^2} \Big ],  &\text{if }K<n \label{rform}\\
mn & \text{if } K=n
\end{cases}
\end{eqnarray}
where $\sigma_1\geq \ldots \geq \sigma_n\geq 0$ are the singular values of $Y$.
\end{corollary}

The proof of Corollary \ref{cor3} can be found in Appendix \ref{pf:rank}. The term $(m+n-K)K$ appearing in the formula of \emph{df} equals the number of free parameters in the specification of a $m\times n$ matrix with rank $K$. Corollary \ref{cor3} demonstrates that the degrees of freedom of $C_K(Y)$ is typically larger than the number of free parameters (when $K<n$). 

The expression inside the expectation in~\eqref{rform} has been proved equal to the divergence $\nabla \cdot C_K(Y)$ in \citet{Yuan2011}. It was obtained by fairly involved tools in calculus and tedious algebraic derivations. As will be shown in Appendix \ref{pf:rank}, we obtain this expression via a simple application of Corollary~\ref{cor2}. More importantly, Corollary \ref{cor3} establishes that $\nabla \cdot C_K(Y)$ is unbiased for $df(C_K(Y))$ -- that is, formula \eqref{alter} holds for the matrix estimator $C_K(Y)$. This verification step was not 
presented in~\citet{Yuan2011}; wherein, the validity of~\eqref{alter} was assumed. As we explain in the next section, verifying the regularity conditions is rather nontrivial.

\subsubsection{Verifying the regularity conditions}
\label{verifying:regularity:con}

We have showed in Section \ref{sec:srestimate} that Stein's lemma \eqref{alter} is inapplicable to the discontinuous rank regularized estimator $S_{\theta}(Y)$. In light of such a result, it is important to investigate if the regularity conditions sufficient for the identity $\emph{df}(C_K(Y))=\mathbb{E}[\nabla \cdot C_K(Y)]$ in \eqref{alter} to hold true, are satisfied for the reduced rank estimator $C_K(Y)$. In fact, checking the weak differentiability of $C_K(Y)$ is not straightforward. We provide some evidence below.


Firstly, $C_K(Y)$ might not be continuous at $Y$ when $\sigma_K(Y)=\sigma_{K+1}(Y)$. This can be seen by a simple example. Suppose $m=n=3, K=2$ and $\{\B e_1, \B e_2,\B e_3\}$ is a set of orthonormal bases in $\mathbb{R}^3$. For $Y=2\B e_1 \B e'_1+\B e_2 \B e'_2+\B e_3 \B e'_3$, consider a sequence 
$$Y_\ell=2\B e_1 \B e'_1+(1+1/\ell)\B e_2 \B e'_2+(1-1/\ell)\B e_3 \B e'_3 \rightarrow Y,  \mbox{~~as~~}\ell \rightarrow \infty.$$
It is direct to verify that as $\ell \rightarrow \infty$,
$$C_K(Y_\ell)=2\B e_1 \B e'_1+(1+1/\ell)\B e_2 \B e'_2 \rightarrow 2\B e_1 \B e'_1+\B e_2 \B e'_2.$$
Now for another sequence
 $$\tilde{Y}_\ell=2\B e_1 \B e'_1+(1-1/\ell)\B e_2 \B e'_2+(1+1/\ell)\B e_3 \B e'_3 \rightarrow Y,  \mbox{~~as~~}\ell \rightarrow \infty,$$
it is clear that as $\ell \rightarrow \infty$,
\[
C_K(\tilde{Y}_\ell)=2\B e_1 \B e'_1+(1+1/\ell)\B e_3 \B e'_3 \rightarrow 2\B e_1 \B e'_1+\B e_3 \B e'_3.
\]

Moreover, $C_K(Y)$ might not be Lipschitz continuous over the open ball outside the set  $\{Y: \sigma_K(Y)=\sigma_{K+1}(Y)\}$. To illustrate this, we take a simple example as follows. Let $m=n=2K$ for a positive integer $K$, 
and set
$$Y_1=U\diag(\B \sigma)V', Y_2=U\diag( \tilde{\B \sigma})V',$$
$$\sigma_i=a, \tilde{\sigma}_i=b, \sigma_j=b, \tilde{\sigma}_j=a, 1\leq i \leq K, K+1 \leq j \leq n,$$
 where $U, V\in \mathbb{R}^{n\times n}$ are orthogonal matrices and $a, b$ are two constants satisfying $0<b<a$. We can then compute that $\|Y_1-Y_2\|^2_F=2K(a-b)^2$, and $\|C_K(Y_1)-C_K(Y_2)\|^2_F=2Ka^2$. Hence, by choosing $b=a-1$, we can conclude 
 \[
 \sup_a \frac{\|C_K(Y_1)-C_K(Y_2)\|_F}{\|Y_1-Y_2\|_F}=\infty.
 \] 

\subsubsection{Estimating df via smoothing with convolution operators} 

The discussions in Section \ref{verifying:regularity:con}, suggest the difficulty of legitimately invoking Stein's Lemma to obtain an expression for \df. We thus pursue a different approach, which to our knowledge, is novel. To this end, we first compute the \df for a smoothed version of $C_K(Y)$, obtained by the following convolution operation:
\begin{eqnarray*}
g_h(Y)=\mathbb{E}_Z[C_K(Y+hZ)],
\end{eqnarray*}
where the elements of $Z \in \mathbb{R}^{m \times n}$ are i.i.d from $N(0,1)$, independent of $Y$; the expectation $\mathbb{E}_Z(\cdot)$ is taken with respect to $Z$; and $h>0$ is a constant.  Because $g_h(Y)$ satisfies the regularity conditions, we can derive $\emph{df}(g_h(Y))$ by computing the divergence of $g_h(Y)$. Since it can be shown that $\emph{df}(g_h(Y)) \rightarrow \emph{df}(C_K(Y))$, as $h \rightarrow 0+$, we are able to obtain $\emph{df}(C_K(Y))$ by letting $h\rightarrow 0+$. However, the detailed analysis is quite involved, we thus postpone the complete proof to Appendix \ref{pf:rank}.

As we were preparing the paper, we became aware of the recent work \cite{hansen2018stein} that also provides a rigorous derivation of the \df for reduced rank estimators. However, the proof technique in \cite{hansen2018stein} is significantly different from ours. The author in \cite{hansen2018stein} verifies directly the weak differentiability of the estimator and proceeds with divergence calculation, while our approach is rather indirect and constructive as explained in the preceding paragraph. Moreover, the approximation strategy via convolution with Gaussian kernel discussed above can in fact work beyond matrix estimation settings. For example, under the linear regression model, the best subset selection in constrained form is:
\begin{align*}
\hat{\B \beta} \in \argmin_{\B \beta \in \mathbb{R}^p} \|\B y - X\B \beta\|_2^2 \mbox{~~subject to~}\|\B \beta\|_0\leq k.
\end{align*}
Under the orthogonal design, $\hat{\B \beta}_i=\B x_i'\B y \cdot \mathbbm{1}(|\B x_i'\B y| \geq |\B x' \B y|_{(k)})$, where $\B x_i$ is the $i$th column of $X$ and $|\B x'\B y|_{(k)}$ is the $k$th largest value among $\{|\B x_i'\B y|\}_i$. The $\emph{df}$ of $\hat{\B \beta}$ in this case with null underlying signal has been derived in \cite{ye1998measuring} by making use of the projection property of least square estimates. Alternatively, we can follow the approximation arguments and study the sequence 
\[
\hat{\B \beta}_i^h=\mathbb{E}_{\B z}[(\B x_i'\B y+z_i) \cdot \mathbbm{1}(|\B x_i'\B y+z_i| \geq |\B x' \B y+\B z|_{(k)})], \quad \B z\sim N(\B{0}, h\cdot \M{I}_p)
\]
to obtain the $\emph{df}$ formula in an automatic way. Since the calculation is standard, we skip it here. 

\section{Degrees of freedom in (low rank) multivariate linear regression} \label{sec4}
Low rank matrix estimation problems also arise in the multivariate linear regression setting, where one is interested in modeling several response measurements simultaneously. In particular, the multivariate linear regression model is given by:
\begin{eqnarray}
Y=XM^{\ast}+\mathcal{E},   \label{multimodel}
\end{eqnarray}
where, $Y=(\B y_1,\ldots, \B y_m)' \in \mathbb{R}^{m \times n}$ is the response matrix, $X=(\B x_1, \ldots, \B x_m)' \in \mathbb{R}^{ m \times p}$ is the design matrix, $M^{\ast} \in \mathbb{R}^{p \times n}$ is the underlying coefficient matrix, and $\mathcal{E}=(\epsilon_{ij})_{m \times n}$ with $\epsilon_{ij} \overset{iid}{\sim}N(0, \tau^2)$ is the random noise matrix. 

\subsection{Reduced rank regression estimators}

In many applications, it is reasonable to assume that the dependency of $Y$ on $X$ is only through $K<\min(p,n)$ linear combinations, namely, $M^{\ast}$ is of low rank. In such cases, we can consider the following reduced rank regression estimator \citep{anderson1951estimating,velu2013multivariate},
\begin{eqnarray}
M_K(Y) \in \argmin_{\text{rank}(M)\leq K }\|Y-XM\|^2_F.  \label{multi}
\end{eqnarray}
Let the compact singular value decomposition of $X$ be $X_{m \times p}=U_{m \times r}\Sigma_{r\times r} V_{p\times r}'$, with $r$ being the rank of $X$. Then the least squares fit is given by 
\begin{align}
\label{multiols:fit}
\hat{Y}=X(X'X)^+X'Y=UU'Y,
\end{align}
 where $(X'X)^+$ is the Moore-Penrose pseudo inverse of $X'X$. By applying Eckart-Young Theorem, an explicit solution of \eqref{multi} is given as follows \citep{Yuan2011, mukherjee2015degrees}:
\begin{eqnarray*}
 M_K(Y)= (X'X)^{-1}X'C_K(\hat{Y}) & \text{if~} r=p<m.
 \end{eqnarray*}
$M_K(Y)$ might not be unique when $p>m$, but the fitted value $XM_K(Y)$ is unique with $XM_K(Y)=C_K(UU'Y)$. The reduced rank problem \eqref{rank} can be thought of as a special case of \eqref{multi} where $X$ equals the identity matrix $ I \in \mathbb{R}^{m\times m}$. We will use the \emph{df} result for the reduced rank estimator in Corollary \ref{cor3} to derive the \emph{df} formula for the estimator defined in \eqref{multi}. It is important to note that, in the current regression setting, the interest under SURE framework lies on the prediction error $\mathbb{E}(\|XM^*-XM_K(Y)\|^2_F)$ rather than the estimation error $\mathbb{E}(\|M^*-M_K(Y)\|_F^2)$. Therefore, the degrees of freedom for $M_K(Y)$ is defined as 
 \[
 df(M_K(Y))=\sum_{ij}\text{Cov}((XM_K(Y))_{ij}, Y_{ij}) / \tau^2,
 \]
 where $(XM_K(Y))_{ij}$ is the $(i,j)$th entry of the matrix $XM_K(Y)$.

\begin{corollary} \label{cor4}
Consider the reduced rank regression estimator $M_K(Y)$ in \eqref{multi} under the model \eqref{multimodel}. We have the following df formula for $M_K(Y)$:
\begin{align} \label{multidf}
&df(M_K(Y))  \nonumber  \\
=&
\begin{cases}
\mathbb{E}\Big [(r+n-K)K+2\sum_{i=1}^K\sum_{j=K+1}^{\min(r,n)}\frac{\sigma_j^2}{\sigma_i^2-\sigma_j^2} \Big ], & \text{if } K < \min(r, n) \\
rn, & \text{if } K \geq \min(r,n)
\end{cases}
\end{align}
where $\sigma_1 \geq \ldots \geq \sigma_{\min(r,n)}\geq 0$  are the singular values of the least square fitted value $\hat{Y}$ in \eqref{multiols:fit}.
\end{corollary}

We are aware that the analytic expression inside the expectation in \eqref{multidf} has been shown to be equal to $\nabla \cdot (XM_K(Y))$ in \citet{Yuan2011, mukherjee2015degrees}. Both papers use the chain rule to relate the divergence of $XM_K(Y)$ to the divergence of a related reduced rank estimator. 
Our approach differs as we compute the \df from basic principles and then appeal to Corollary \ref{cor3}. 
We emphasize that the unbiasedness of divergence for the \df, i.e., Equation \eqref{alter}, does not necessarily hold --- the regularity conditions sufficient for the identity to hold, need to be rigorously verified. As has been demonstrated in Section \ref{verifying:regularity:con}, the weak differentiability of the reduced rank estimator may not be easily confirmed. Based on the result from Corollary \ref{cor3}, we are able to provide a complete justification for the expression derived in~\eqref{multidf} bypassing such a difficulty. 


\subsection{Spectral regularized regression estimators}

In addition to the constrained estimator in \eqref{multi}, we may also consider the penalized problem
\begin{equation}
 \argmin_{M\in \mathbb{R}^{p\times n}} ~~~ \frac{1}{2}\|Y-XM\|^2_F + \sum_{i=1}^{\min(p,n)}P_{\theta}(\sigma_i) \label{rmestimate}.
\end{equation}
However, unlike the spectral regularized problem \eqref{estimator}, except for few penalty functions like $P_{\theta}(\sigma)=\theta \mathbbm{1}(\sigma\neq 0)$ \citep{bunea2011optimal}, there is no closed form solution for \eqref{rmestimate}. Simple expressions for the degrees of freedom for such fitting procedures seem to be unknown. We note however,
that some nice work is available on the \emph{df} of regularized estimators in the multiple linear regression---see for e.g.~\citet{zou2007degrees, tibshirani2012degrees}.

We follow the approach of~\citet{mukherjee2015degrees}.
 Motivated by the solution form of \eqref{multi}, we explicitly construct an estimator for $M^*$ given by
\begin{eqnarray}
RM_{\theta}(Y)=(X'X)^{-1}X'S_{\theta}(\hat{Y}),~~~~~~~~ X\cdot RM_{\theta}(Y)=S_{\theta}(UU'Y),  \label{rmreal}
\end{eqnarray}
where $\hat{Y}$ is the least square fitted value, $U$ is the left singular vector matrix of $X$, and $S_{\theta}(\cdot)$ is defined in \eqref{estimator}. 
The following two corollaries provide an expression of the \emph{df} for a variety of such estimators.

\begin{corollary}
\label{final:last0}
For the penalized multivariate regression estimator $RM_{\theta}(Y)$ in \eqref{rmreal} under the model \eqref{multimodel}, if the same conditions for $P_{\theta}(\cdot)$ as in Corollary \ref{cor2} hold, then
\begin{eqnarray*} \label{last}
df(RM_{\theta}(Y))=\mathbb{E} \Bigg [ \sum_{i=1}^{\min(r,n)}  \Big ( s'_{\theta}(\sigma_i)+|r-n|\frac{s_{\theta}(\sigma_i)}{\sigma_i} \Big)+2\sum_{\substack {i \neq j\\ i, j=1}}^{\min(r,n)}\frac{\sigma_is_{\theta}(\sigma_i)}{\sigma^2_i-\sigma^2_j} \Bigg ],
\end{eqnarray*}
where $\sigma_1 \geq \ldots \geq \sigma_{\min(r,n)}\geq 0$ are the singular values of the least square fitted value $\hat{Y}$ in \eqref{multiols:fit}.
\end{corollary}

\begin{table}[htb!]
	\small
	\begin{center}
		\begin{tabular}{|c|c|c|}
			\hline
			 Penalty name & Penalty function & Applicability of Stein's lemma    \\
			\hline
		         Lasso \cite{tibshirani1996regression} & \parbox{5.8cm}{ \vspace{-0.1cm} \center $P_{\theta}(\sigma)=\theta \sigma$ \vspace{0.2cm}} & Yes \\
			\hline
			SCAD \cite{fan2001variable} & \parbox{5.8cm}{
			 \vspace{0.2cm}
			 $
			 P_{\theta}(\sigma)=
			 \begin{cases}
			 \theta \sigma & 0\leq \sigma \leq \theta \\
			 \frac{-\sigma^2+2a\theta\sigma-\theta^2}{2(a-1)} &  \theta < \sigma \leq a \theta \\
			 \frac{(a+1)\theta^2}{2} & \sigma >a \theta
			 \end{cases}
			  \vspace{0.2cm}
			 $
			 }
			  & Yes, when $a>2$
			  \\
			\hline
			MC+ \cite{zhang2010nearly} & \parbox{5.8cm}{
			 \vspace{0.2cm}
			$
			P_{\theta}(\sigma)=
			\begin{cases}
			\theta(\sigma-\frac{\sigma^2}{2\theta\gamma}) & 0\leq \sigma \leq \gamma \theta \\
			\frac{\gamma \theta^2}{2} & \sigma >\gamma \theta
			\end{cases}
			 \vspace{0.2cm}
			$ 
			}
			& Yes, when $\gamma>1$ \\
			\hline
			Bridge \cite{frank1993statistical}  &  \parbox{5.8cm}{ \vspace{-0.1cm} \center $P_{\theta}(\sigma)=\theta \sigma^q, \quad q\in [0,1)$ \vspace{0.2cm}}  & No \\
			\hline 
			Log \cite{mazumder2011sparsenet} &  \parbox{5.8cm}{ \vspace{0.2cm} \center $P_{\theta}(\sigma)=\frac{\theta\log(\gamma \sigma+1)}{\log(1+\gamma)}$ \vspace{0.2cm}}  & Yes, when $\log(1+\gamma)>\theta \gamma^2$ \\
			\hline
			Firm \cite{gao1997waveshrink} & \parbox{5.8cm}{
			 \vspace{0.2cm}
			$
			P_{\theta}(\sigma)=
			\begin{cases}
			\theta(\sigma-\frac{\sigma^2}{2\gamma}) & 0\leq \sigma \leq \gamma  \\
			\frac{\gamma \theta }{2} & \sigma >\gamma 
			\end{cases}
			 \vspace{0.2cm}
			$ 
			}
			& Yes, when $\gamma>\theta$ \\
			\hline
		\end{tabular}
		\vspace{0.3cm}
		\caption{Examples of commonly used penalty functions. The applicability of Stein's lemma is in terms of the estimator $S_{\theta}(Y)$ in \eqref{estimator} under model \eqref{model} and the estimator $RM_{\theta}(Y)$ in \eqref{rmreal} under model \eqref{multimodel}.}
	\end{center}	
\end{table}

\begin{corollary}
\label{final:last}
For the penalized multivariate regression estimator $RM_{\theta}(Y)$ in \eqref{rmreal} with $P_{\theta}(\alpha)=\theta|\alpha|^q (0\leq q <1)$ under the model \eqref{multimodel}, then
\begin{align*}
&df(RM_{\theta}(Y))=\mathbb{E} \sum_{\substack {i \neq j\\ i, j=1}}^{\min(r,n)} \frac{\sigma_i\eta_q(\sigma_i;\theta)-\sigma_j\eta_q(\sigma_j;\theta)}{\sigma_i^2-\sigma^2_j}+  \\
&\mathbb{E} \sum_{i=1}^{\min(r,n)} \left [|r-n|\frac{\eta_q(\sigma_i;\theta)}{\sigma_i}+\eta'_q(\sigma_i;\theta)+[2(1-q)\theta]^{1/(2-q)} f_{\sigma_i}(c_q\theta^{1/(2-q)})\right ],
\end{align*}
where $f_{\sigma_i}(\cdot)$ is the marginal density function of $Y$'s $i$th singular value $\sigma_i$; $\eta'_q(\sigma_i;\theta)$ is the partial derivative of $\eta_q(\sigma_i;\theta)$ with respect to $\sigma_i$; $\eta(\cdot;\theta), c_q$ are defined in \eqref{etaq:def} and \eqref{cq:constant} respectively; $\sigma_1 \geq \ldots \geq \sigma_{\min(r,n)}\geq 0$ are the singular values of the least square fitted value $\hat{Y}$ in \eqref{multiols:fit}.
\end{corollary}

The result in the above two corollaries for $RM_{\theta}(Y)$ notably differs from that in~\cite{mukherjee2015degrees}. \citet{mukherjee2015degrees} calculates the divergence $\nabla \cdot (XRM_{\theta}(Y))$, while we provide a theoretical justification for the unbiasedness of the divergence for $\emph{df}(RM_{\theta}(Y))$ under a wide class of non-convex penalties $P_{\theta}(\cdot)$ in Corollary \ref{final:last0}, and further obtain the \df formula for a family of penalty functions to which Stein's lemma is not applicable in Corollary \ref{final:last}.

\section{Simulations}
\label{simulation:sec}
In this section, we perform simulation studies to lend further support to the \df formulas that we have derived in Sections \ref{sec3} and \ref{sec4}.

\subsection{Additive Gaussian model}
\label{agm:form}

We generate the data $Y$ according to the canonical additive Gaussian model \eqref{model}:
\[
Y=M^{\ast}+\mathcal{E}, 
\]
where $Y\in \mathbb{R}^{m\times n}, M^{\ast} \in \mathbb{R}^{m\times n}$, and $\mathcal{E}=(\epsilon_{ij})_{m\times n}$ with $\epsilon_{ij}\overset{\text{iid}}{\sim}N(0, \tau^2)$. We set $m=n=100, \tau=0.1, M^{\ast}=\sum_{k=1}^5 k \B u_k \B u'_k,$ where all entries of the $\B u_k$'s are independently sampled from $N(0,1/\sqrt{n})$. We consider the spectral regularized estimator $S_{\theta}(Y)$ in \eqref{estimator} with the following non-convex penalty functions:

\begin{itemize}
\item[(1)] The SCAD penalty \cite{fan2001variable}
\begin{align*}
\hspace{-0.3cm} P_{\theta}(\sigma)=\theta \sigma \mathbbm{1}(\sigma \leq \theta)+\frac{-\sigma^2+2a\theta \sigma-\theta^2}{2(a-1)}\mathbbm{1}(\theta< \sigma \leq a \theta)+\frac{(a+1)\theta^2}{2}\mathbbm{1}(\sigma> a\theta), 
\end{align*}
where $a>2$ is a fixed parameter. We choose $a=3.7$ as used in \citet{fan2001variable}.
\vspace{0.1cm}
\item[(2)] The MC+ penalty \cite{zhang2010nearly}
\begin{align*}
P_{\theta}(\sigma)=\theta\Big(\sigma-\frac{\sigma^2}{2\theta \gamma}\Big)\mathbbm{1}(0\leq \sigma \leq \gamma \theta)+\frac{\gamma \theta^2}{2}\mathbbm{1}(\sigma> \gamma \theta ),
\end{align*}
where $\gamma >0$ is a fixed constant. We set $\gamma=2$. 
\vspace{0.1cm}
\item[(3)] The log-penalty \cite{mazumder2011sparsenet}
\begin{align*}
P_{\theta}(\sigma)=\frac{\theta}{\log(1+\gamma)}\log(1+\gamma \sigma), \quad \gamma>0.
\end{align*}
We choose $\gamma=0.01$.
\vspace{0.1cm}
\item[(4)] The bridge-penalty \cite{frank1993statistical}
\begin{align*}
P_{\theta}(\sigma)=\theta \sigma^q, \quad  q\in[0,1).
\end{align*}
We consider $q=0.1, 0.5, 0.9$.
\end{itemize}

\begin{figure}[t!]
    \begin{center}
        \begin{tabular}{ccc}
            \hspace{-.5cm}  SCAD & \hspace{-.8cm} MC+& \hspace{-.8cm} Log   \vspace{-0.6cm} \\
              \hspace{-.5cm}
            \includegraphics[scale=0.43]{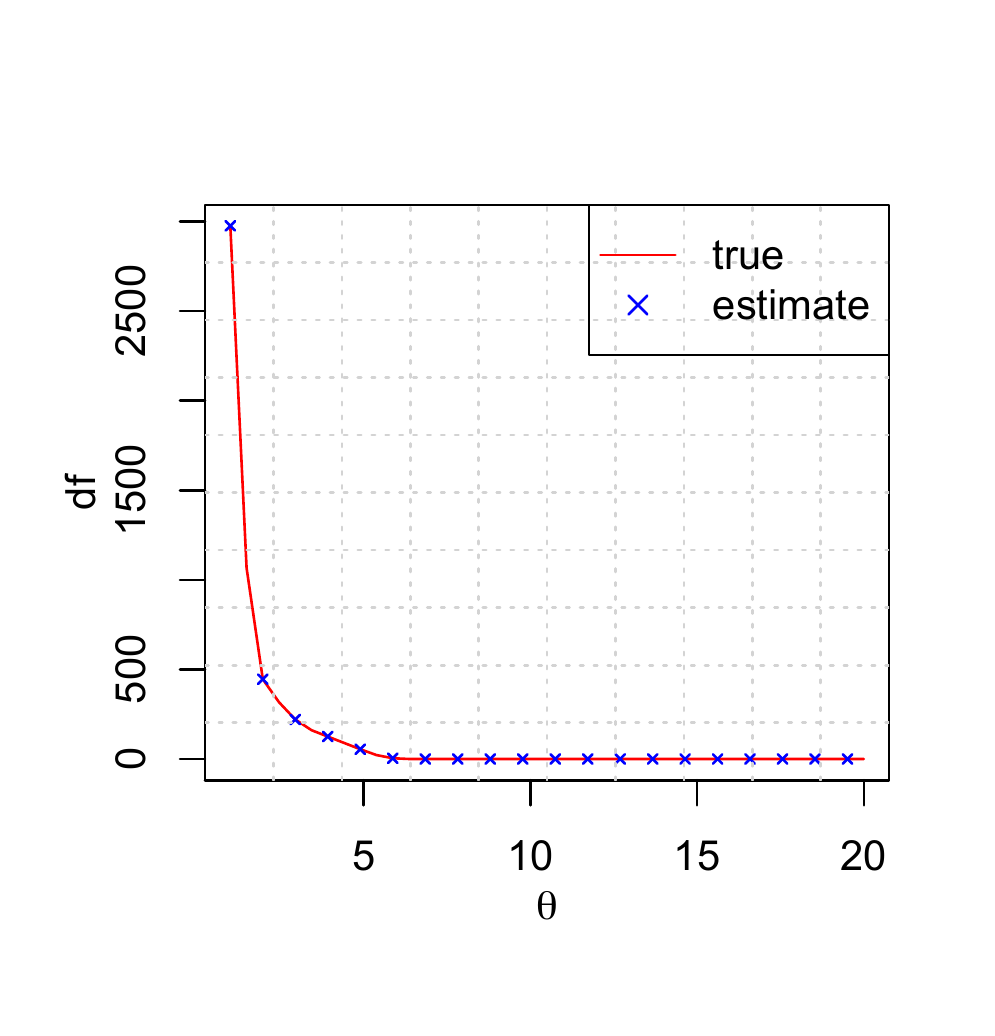} &
             \hspace{-.8cm}
            \includegraphics[scale=0.43]{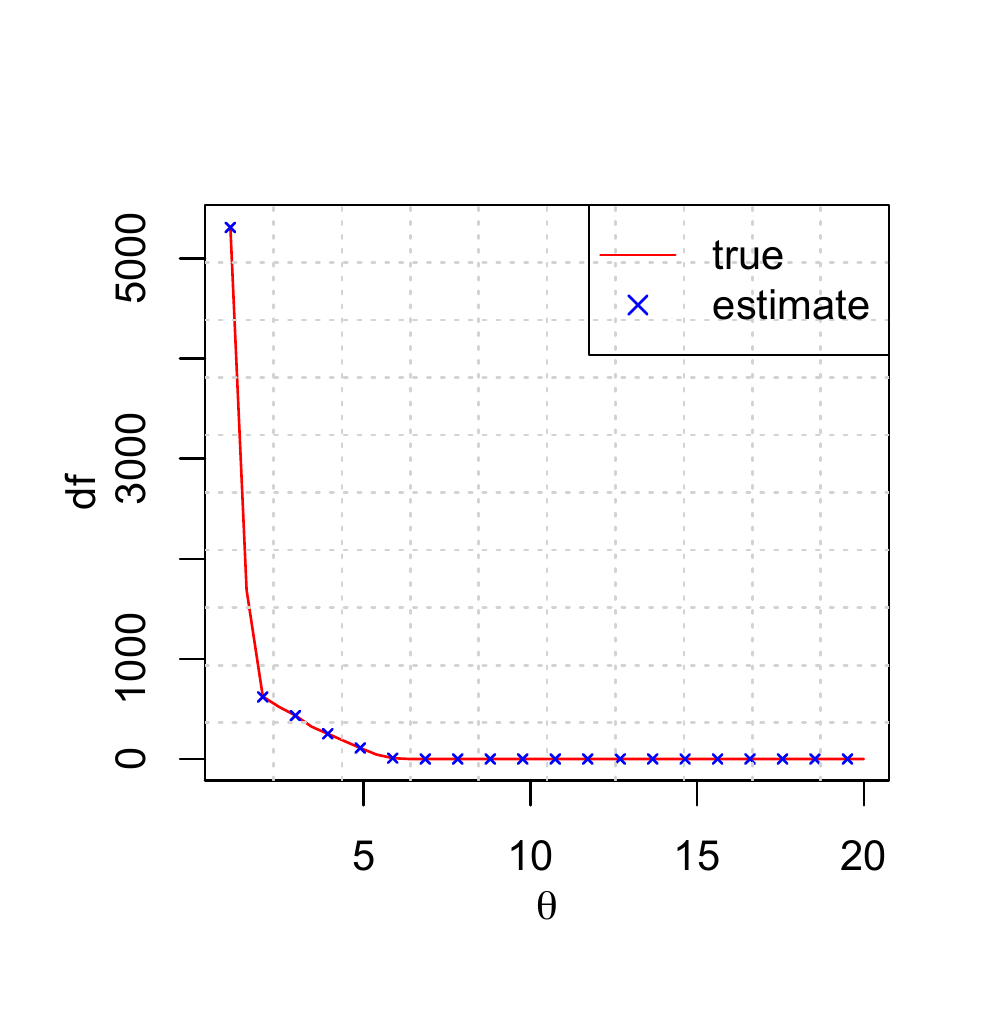} &
             \hspace{-.8cm}
            \includegraphics[scale=0.43]{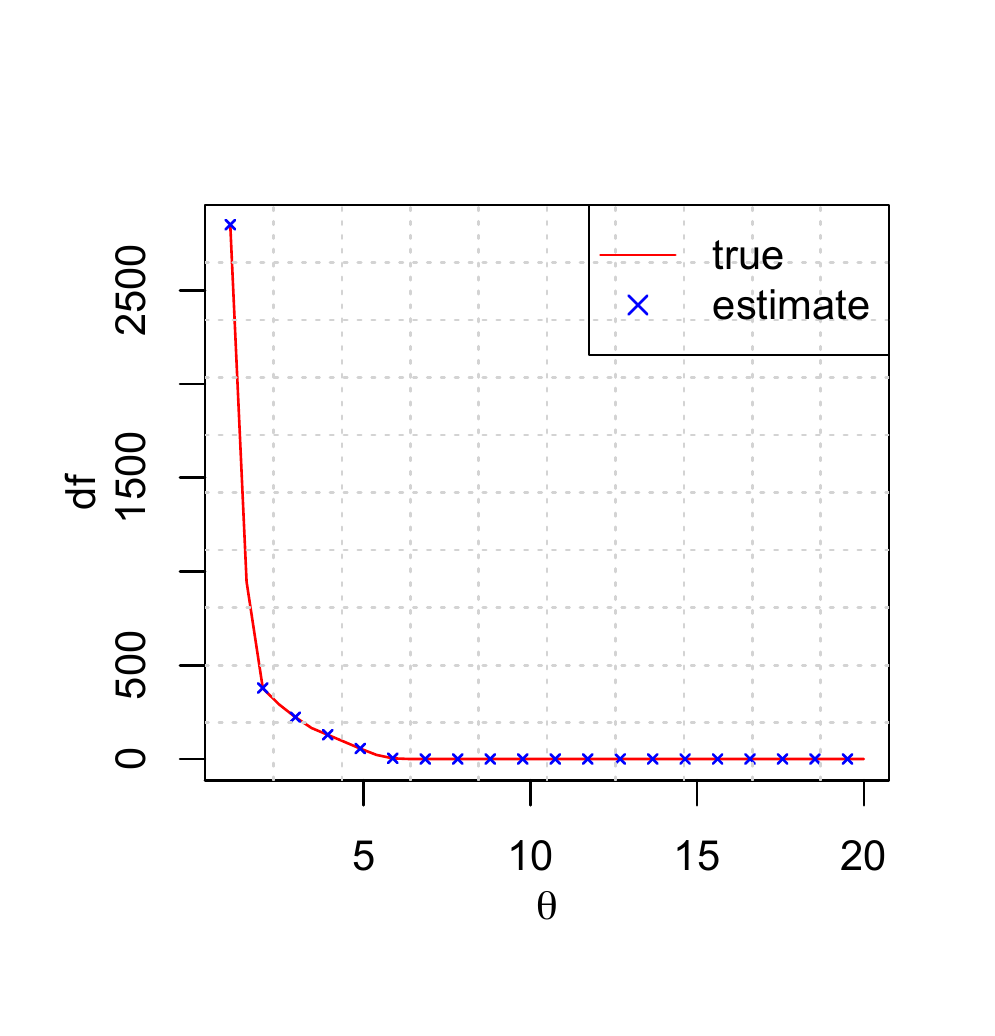}  \vspace{-0.2cm} \\
           \hspace{-.5cm} Bridge ($q=0.1$) &  \hspace{-.8cm} Bridge ($q=0.5$)&   \hspace{-.8cm} Bridge ($q=0.9$)  \vspace{-0.6cm} \\
             \hspace{-.5cm}
            \includegraphics[scale=0.43]{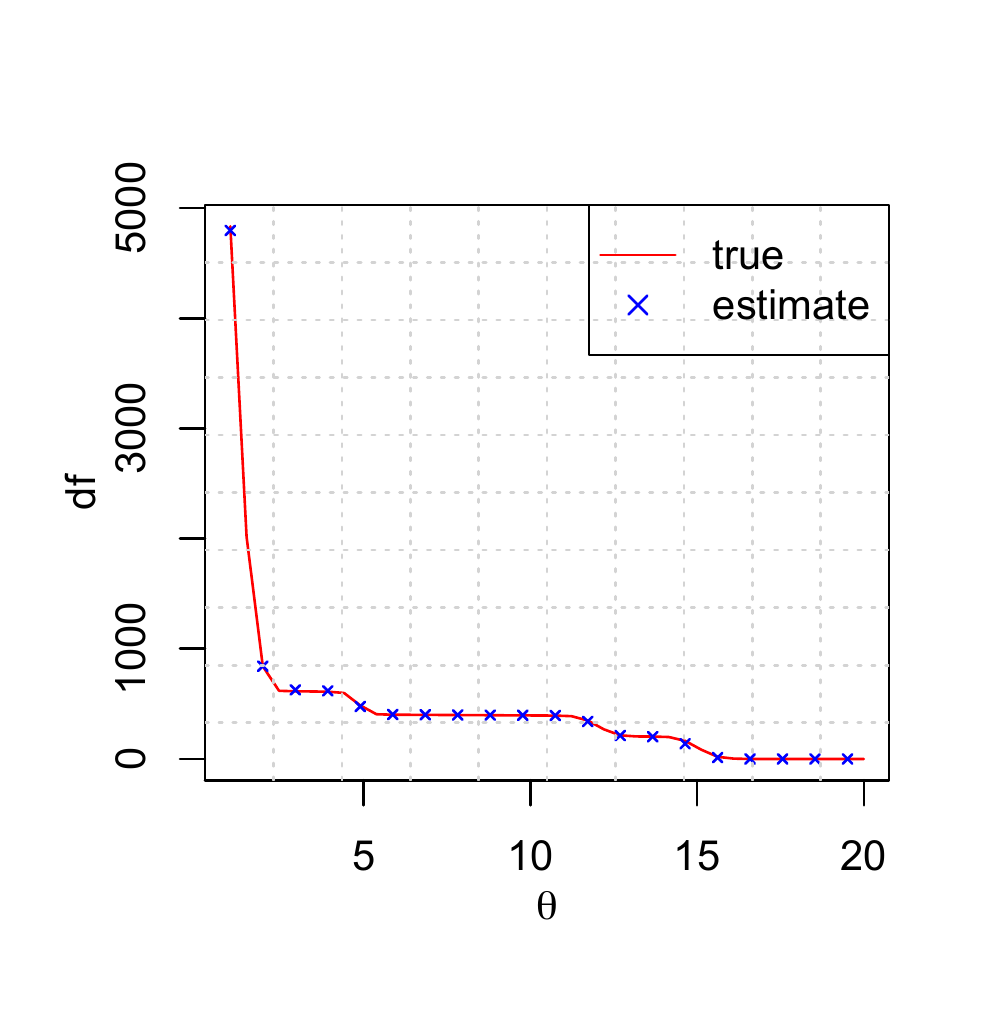} &
             \hspace{-.8cm}
            \includegraphics[scale=0.43]{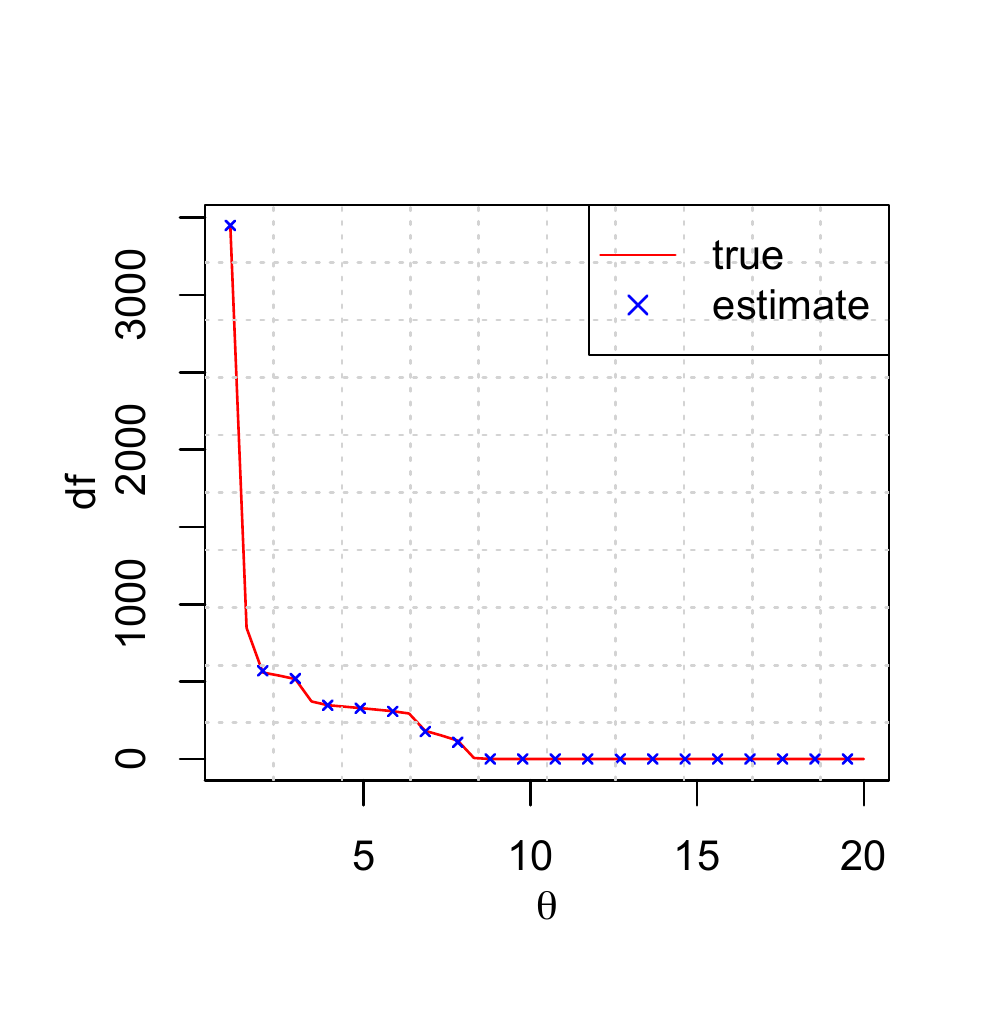} &
             \hspace{-.8cm}
            \includegraphics[scale=0.43]{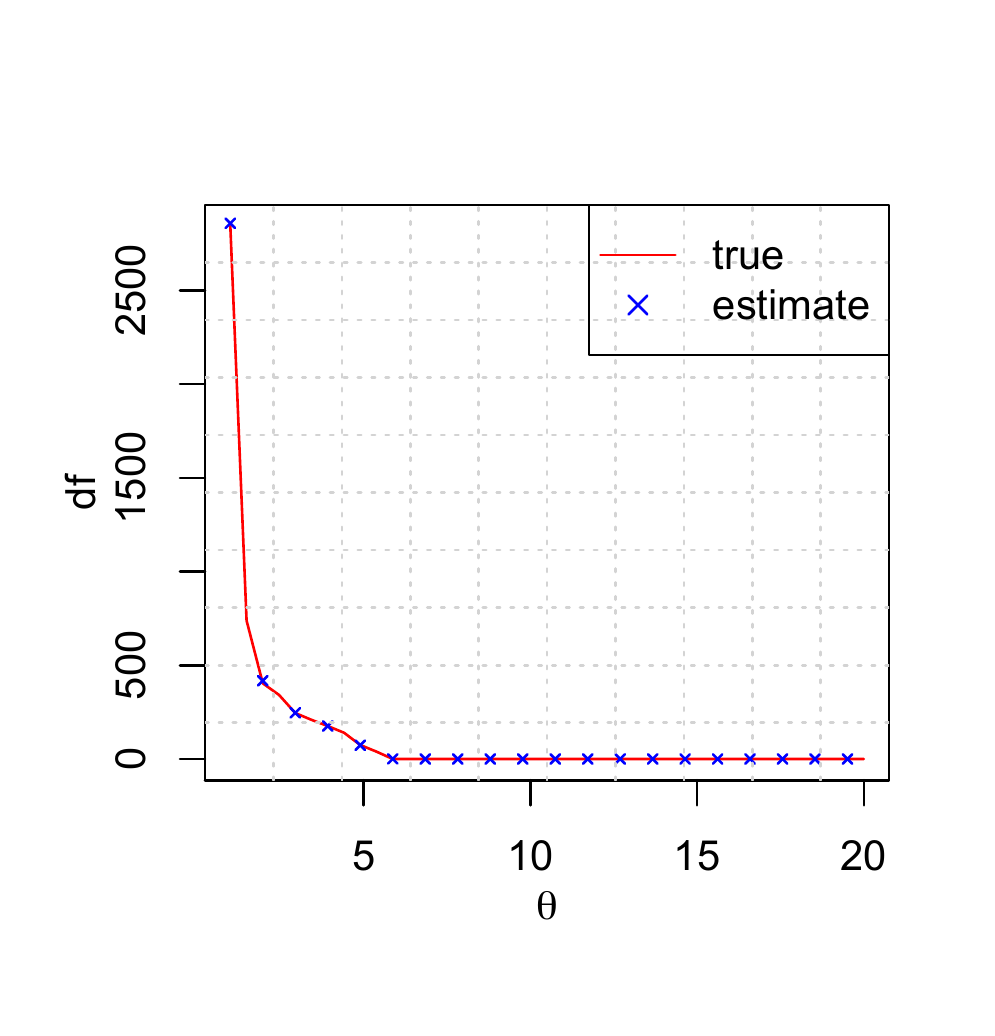}  \vspace{-0.5cm}
        \end{tabular}
        \caption{Degrees of freedom under the additive Gaussian model. The true df (red curve) is computed from \eqref{original}. The estimate (blue cross) is the average of the unbiased estimator over 100 repetitions. } \label{df:plot:case1}
    \end{center}
\end{figure}

\begin{figure}[tb!]
    \begin{center}
        \begin{tabular}{ccc}
            \hspace{-.5cm}  SCAD & \hspace{-.8cm} MC+& \hspace{-.8cm} Log   \vspace{-0.6cm} \\
              \hspace{-.5cm}
            \includegraphics[scale=0.43]{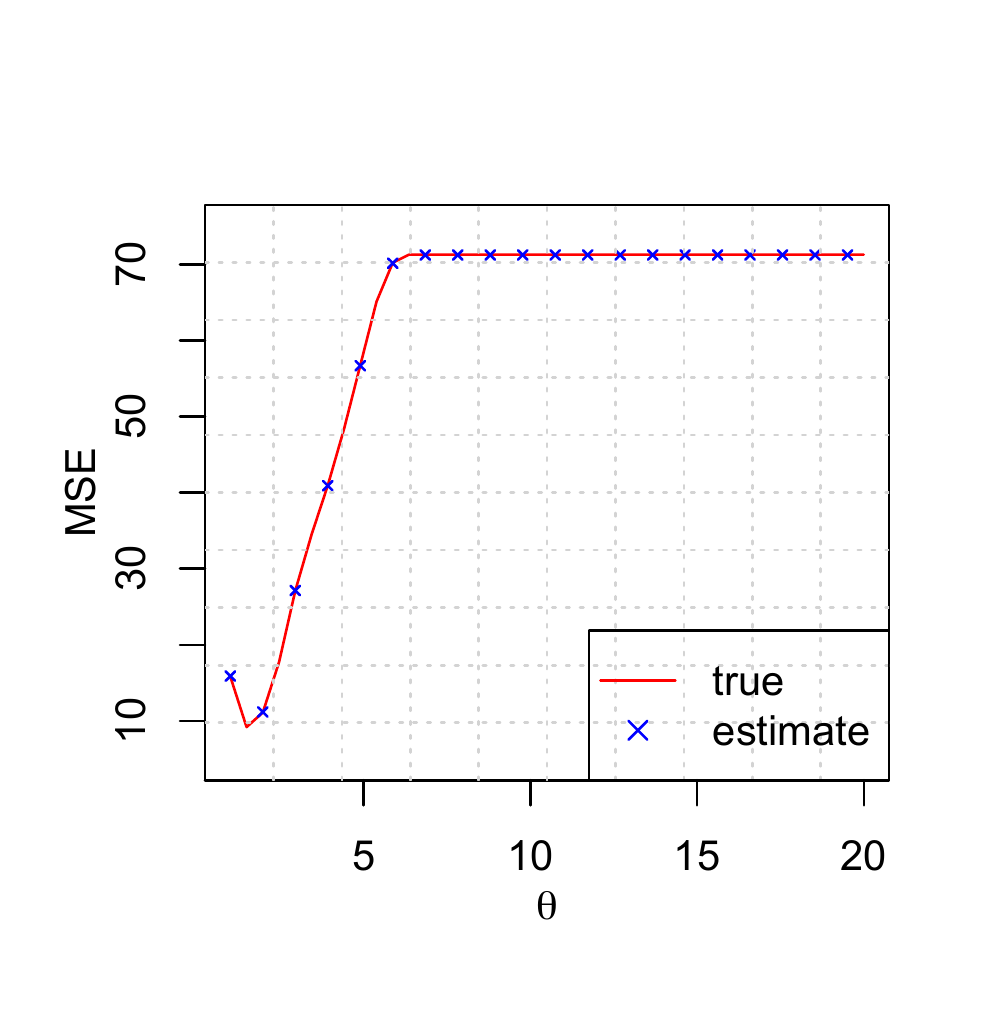} &
             \hspace{-.8cm}
            \includegraphics[scale=0.43]{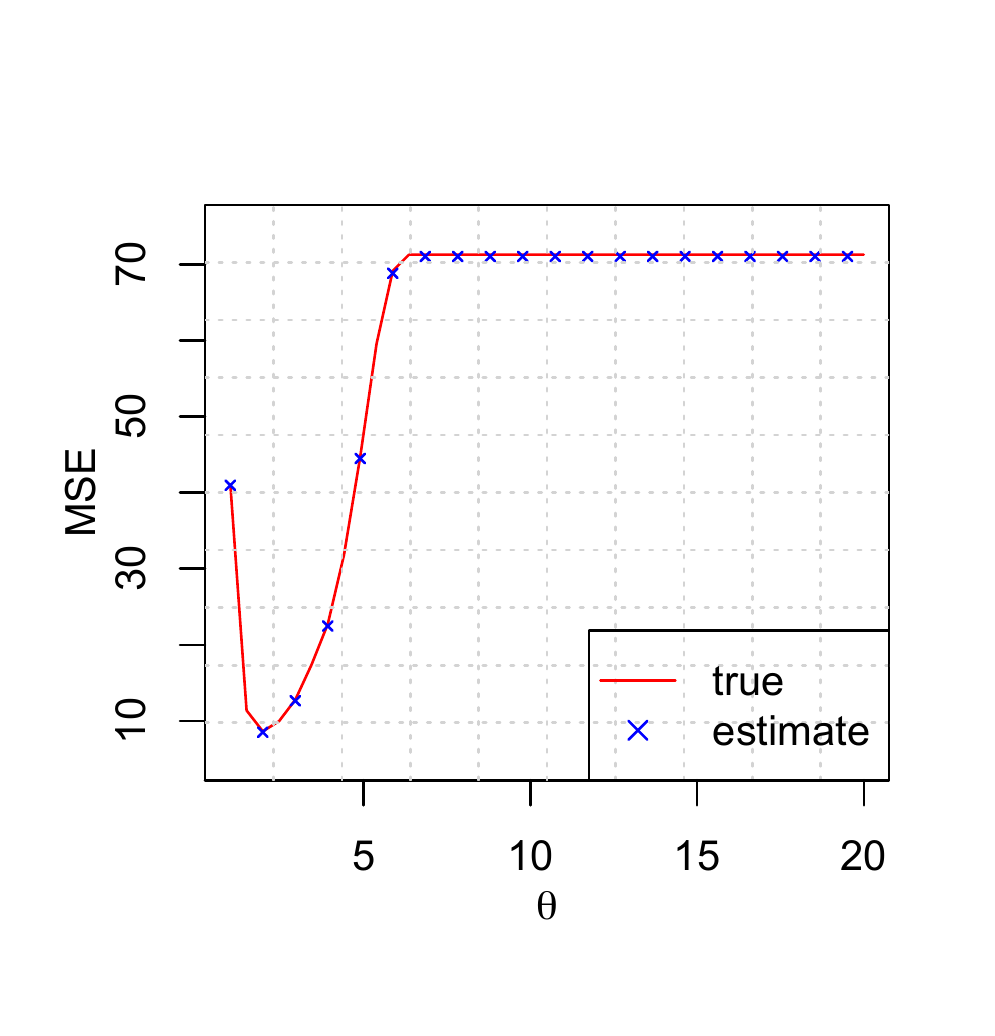} &
             \hspace{-.8cm}
            \includegraphics[scale=0.43]{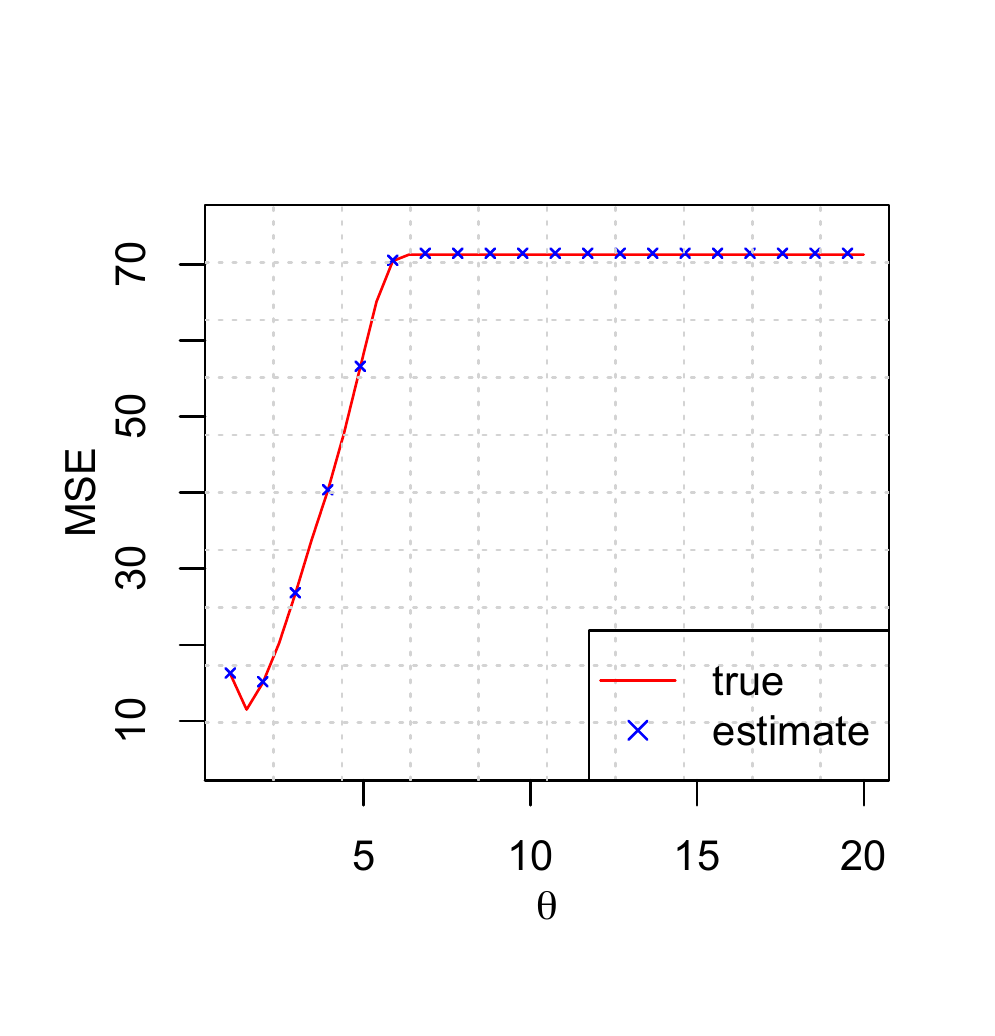}  \vspace{-0.2cm} \\
           \hspace{-.5cm} Bridge ($q=0.1$) &  \hspace{-.8cm} Bridge ($q=0.5$)&   \hspace{-.8cm} Bridge ($q=0.9$)  \vspace{-0.6cm} \\
             \hspace{-.5cm}
            \includegraphics[scale=0.43]{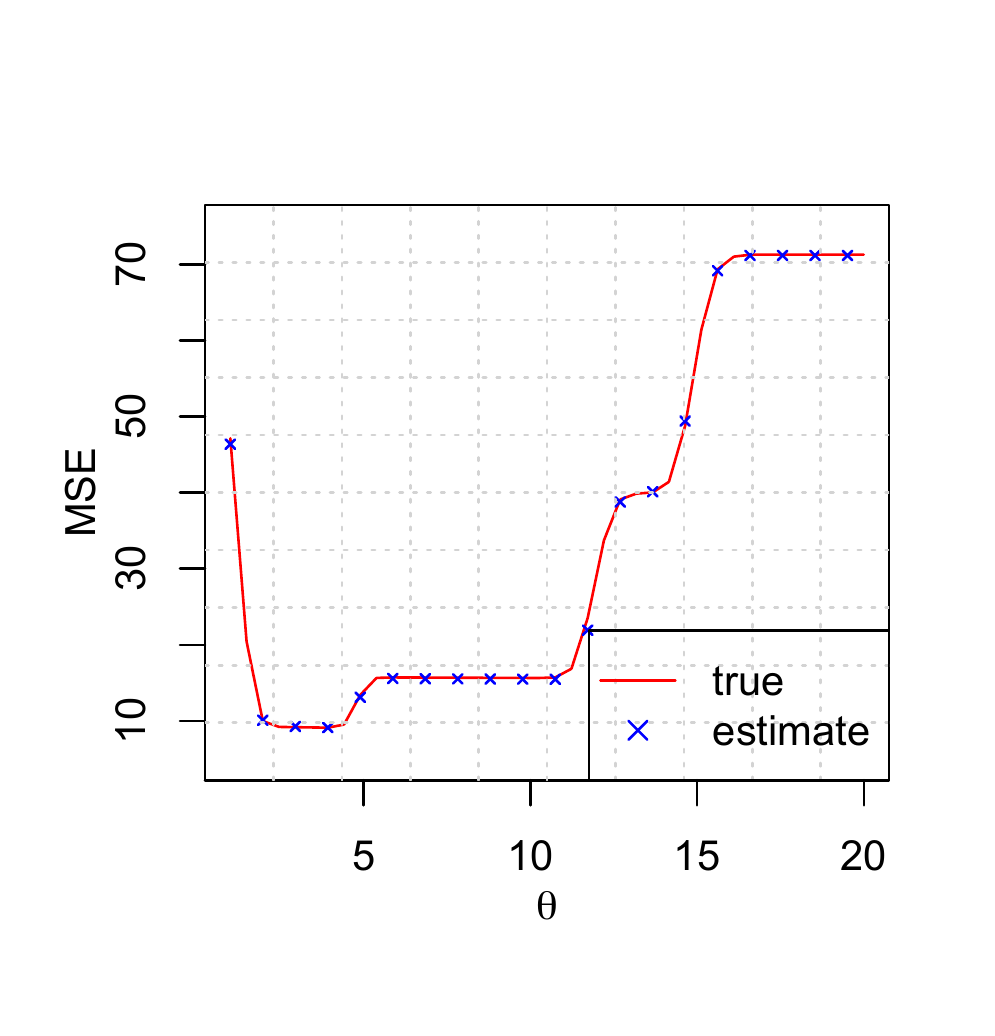} &
             \hspace{-.8cm}
            \includegraphics[scale=0.43]{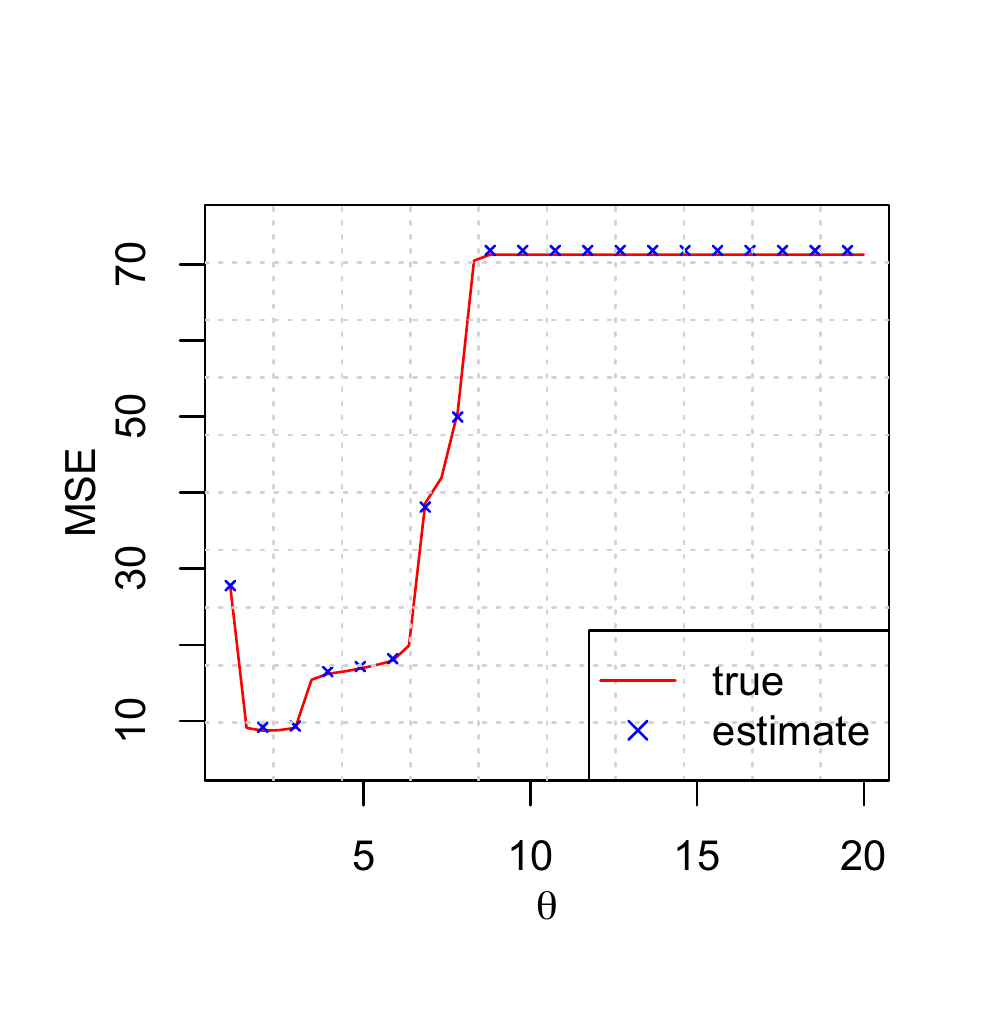} &
             \hspace{-.8cm}
            \includegraphics[scale=0.43]{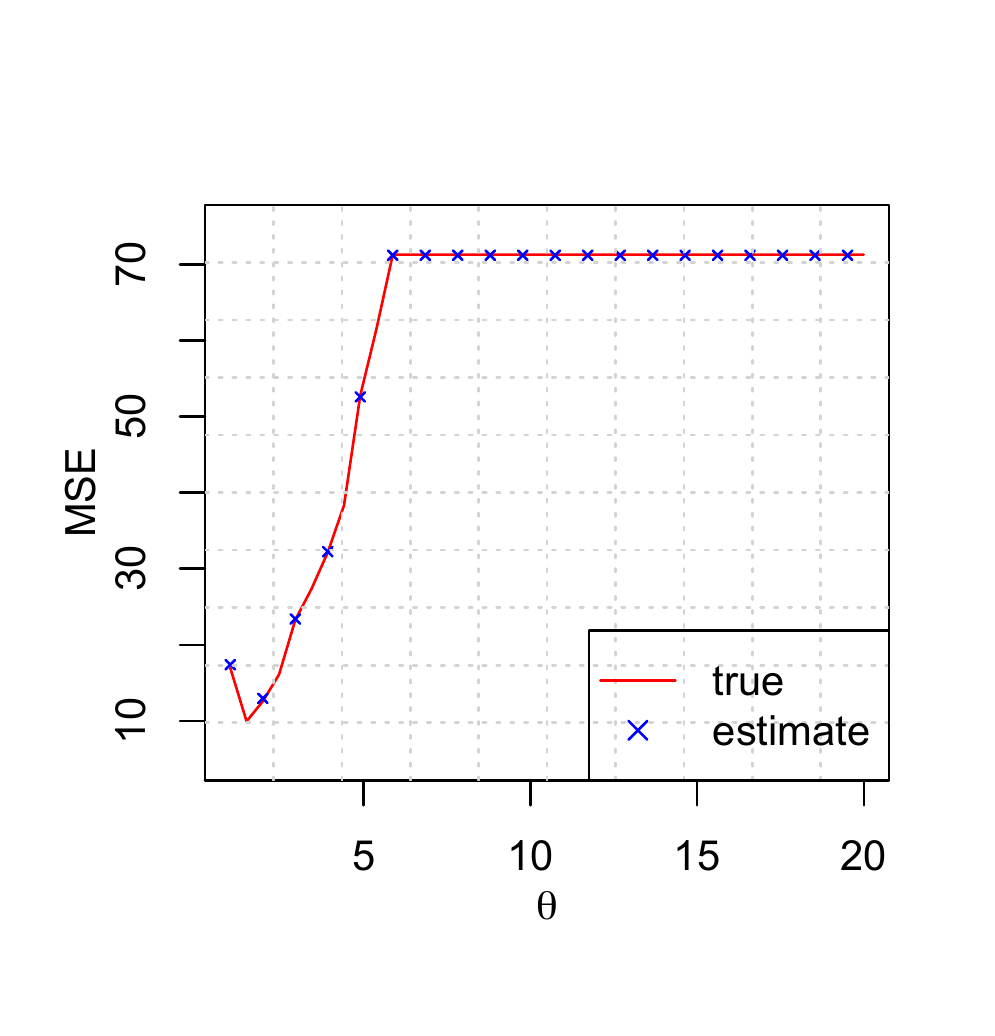}  \vspace{-0.5cm}
        \end{tabular}
        \caption{Expected MSE under the additive Gaussian model. The truth (red curve) is computed via monte carlo simulation. The estimate (blue cross) is the average of the unbiased estimator over 100 repetitions. } \label{mse:plot:case1}
    \end{center}
\end{figure}

It is straightforward to verify that the first three penalty functions above satisfy the conditions in Corollary \ref{cor2} for $\theta \in (0,20]$. Hence we can use the formula \eqref{dfformula} in Corollary \ref{cor2} to construct an unbiased estimator for the \df of the corresponding estimator $S_{\theta}(Y)$ when $\theta \in (0,20]$. For the bridge-penalty function, we use the result in Corollary \ref{cor3done} to obtain the estimator for the \df. Moreover, for each matrix estimator $S_{\theta}(Y)$, we compute its \df (the ground truth) according to the definition \eqref{original}.

Figure \ref{df:plot:case1} depicts the true \df and its unbiased estimate for the aforementioned non-convex penalties with $\theta$ varying over $[0,20]$. It is clear that  the ground truth and the (averaged) estimates are well matched for all the penalties and values of $\theta$ under consideration, thus offering empirical support for the correctness of the derived \df expressions.

In addition to \df, we further evaluate the estimation of the expected MSE $\mathbb{E}\|S_{\theta}(Y)-M^{\ast}\|_2^2$. Recall that for a given $S_{\theta}(Y)$, once an unbiased estimator of the \df is available, an unbiased estimate for the expected MSE can be constructed based on \eqref{unbiased:one}. In the present case, we will use the \df estimates to obtain the estimates for $\mathbb{E}\|S_{\theta}(Y)-M^{\ast}\|_2^2$ according to \eqref{unbiased:one}. Figure \ref{mse:plot:case1} shows the expected MSE and its estimates for the four types of non-convex penalties with $\theta \in [0,20]$. We observe that the (averaged) estimates are well aligned with the truth.

\subsection{Multivariate linear regression}

\begin{figure}[tb!]
    \begin{center}
        \begin{tabular}{ccc}
            \hspace{-.5cm}  SCAD & \hspace{-.8cm} MC+& \hspace{-.8cm} Log   \vspace{-0.6cm} \\
              \hspace{-.5cm}
            \includegraphics[scale=0.43]{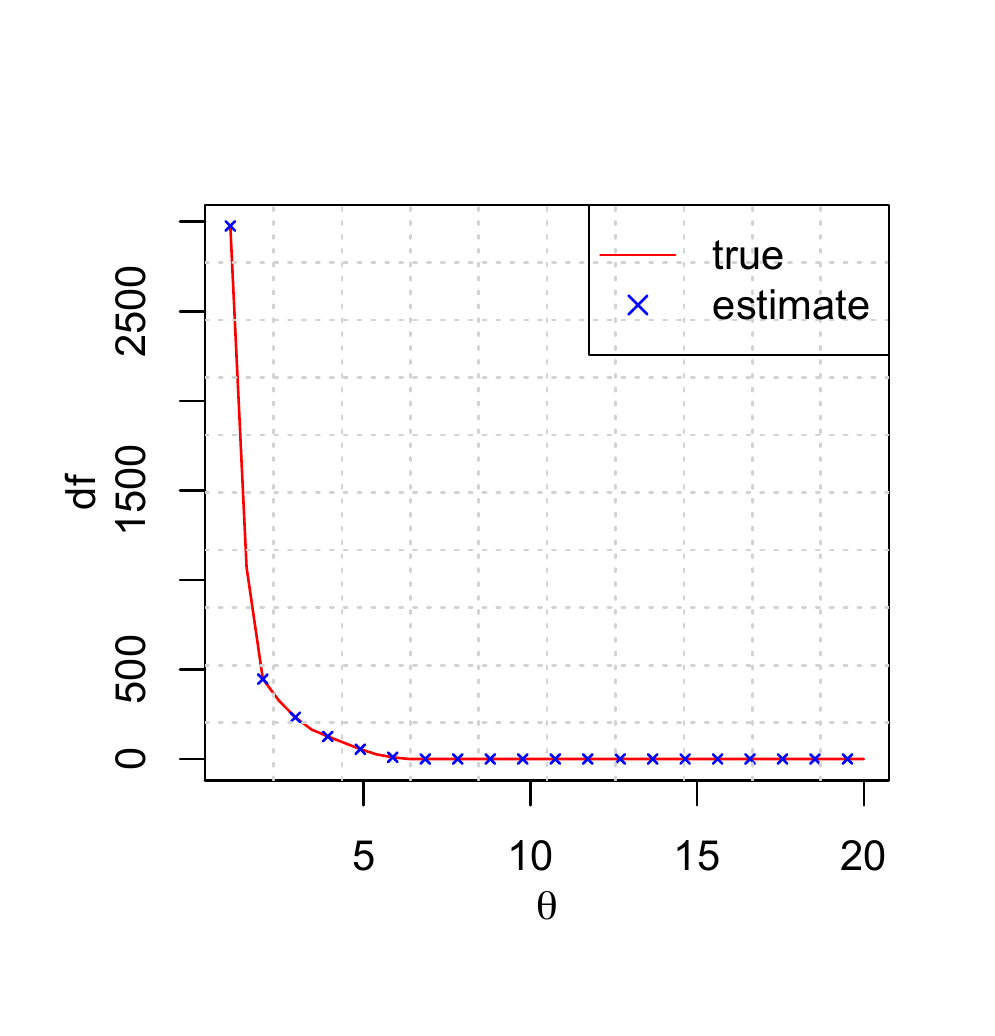} &
             \hspace{-.8cm}
            \includegraphics[scale=0.43]{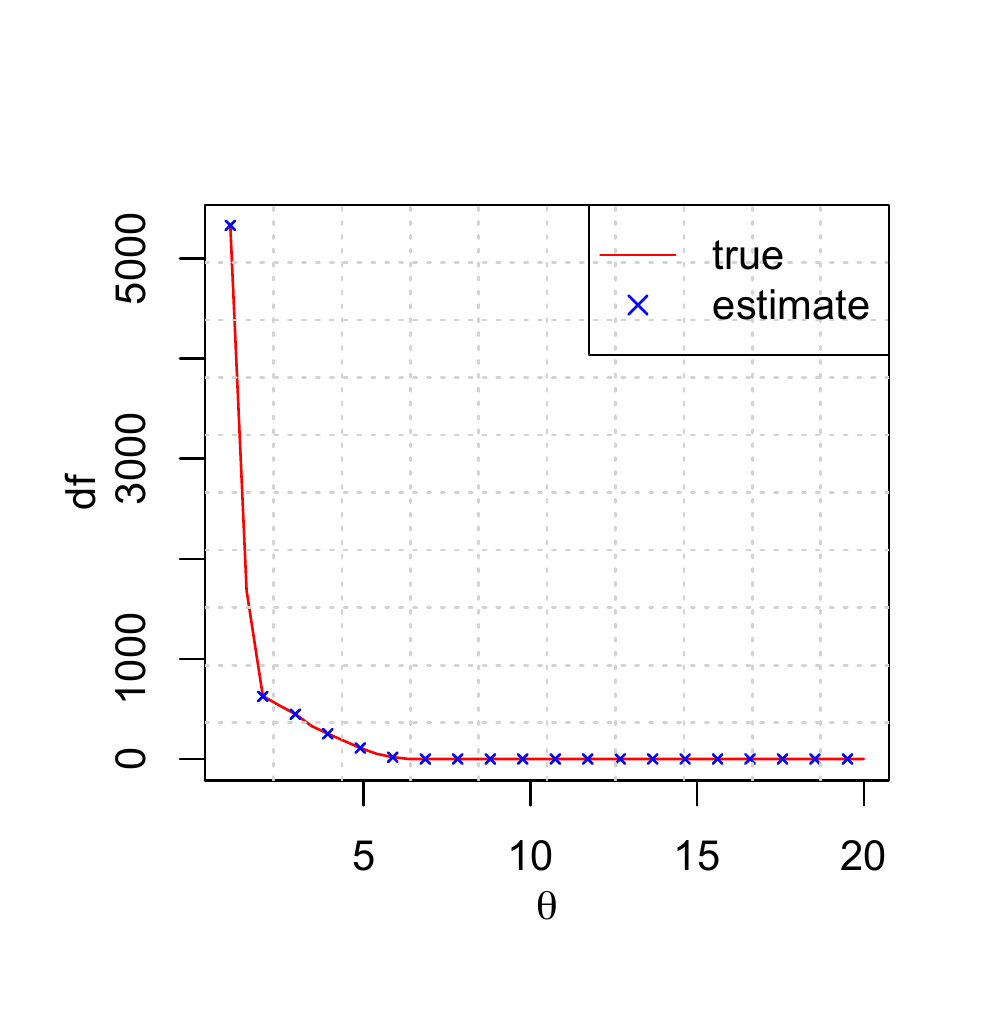} &
             \hspace{-.8cm}
            \includegraphics[scale=0.43]{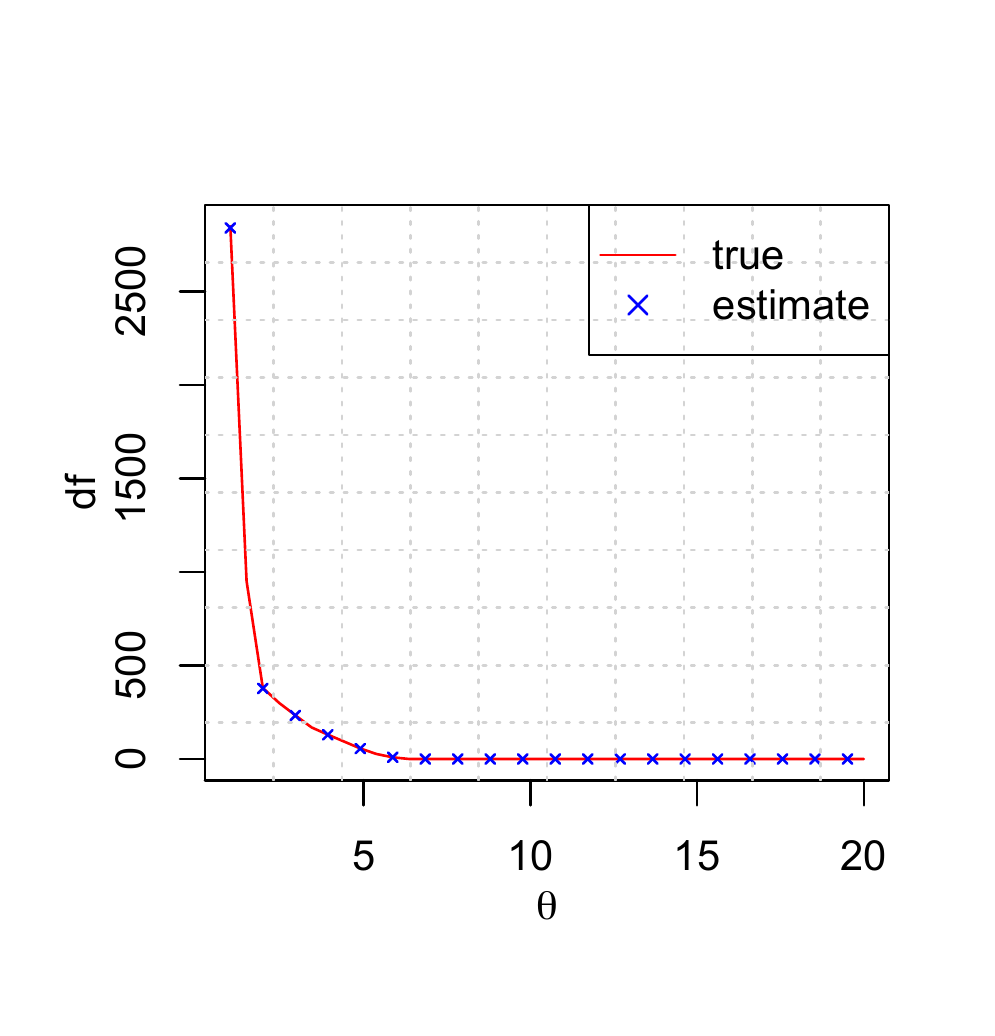}  \vspace{-0.2cm} \\
           \hspace{-.5cm} Bridge ($q=0.1$) &  \hspace{-.8cm} Bridge ($q=0.5$)&   \hspace{-.8cm} Bridge ($q=0.9$)  \vspace{-0.6cm} \\
             \hspace{-.5cm}
            \includegraphics[scale=0.43]{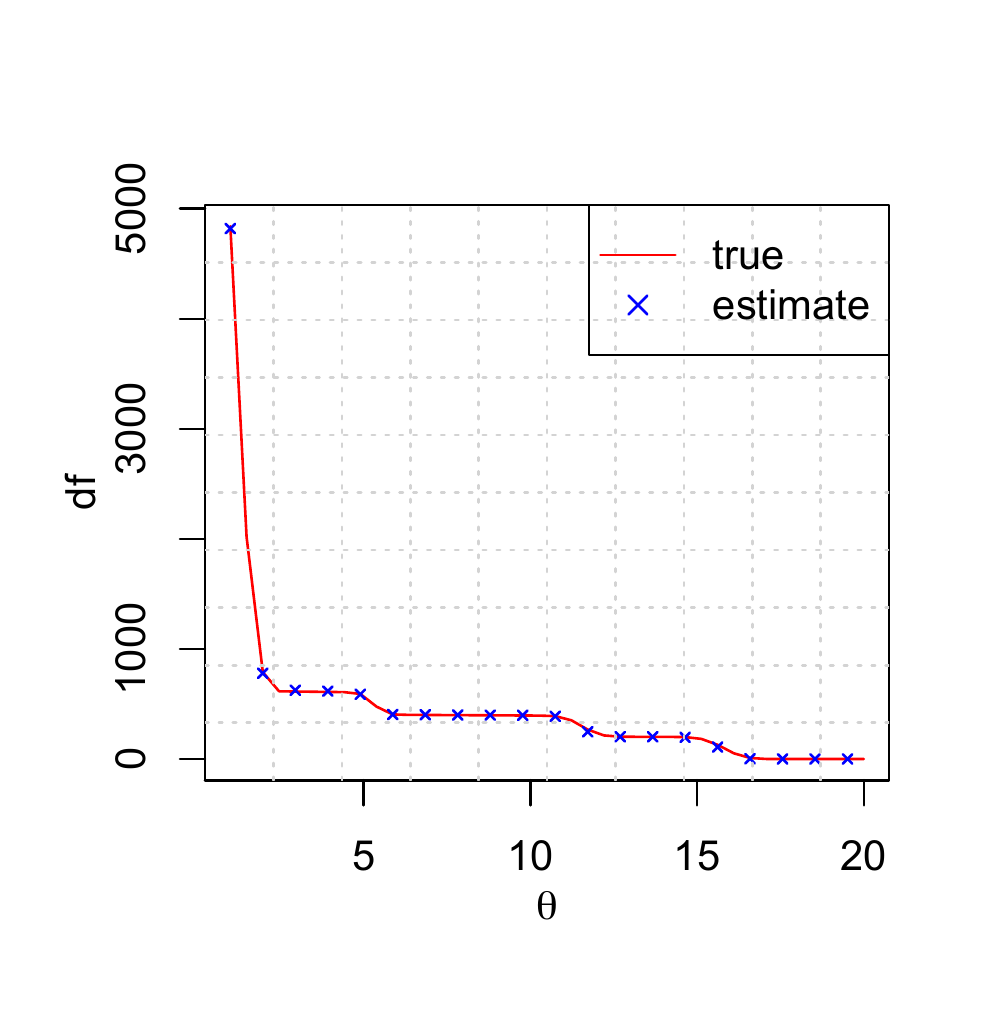} &
             \hspace{-.8cm}
            \includegraphics[scale=0.43]{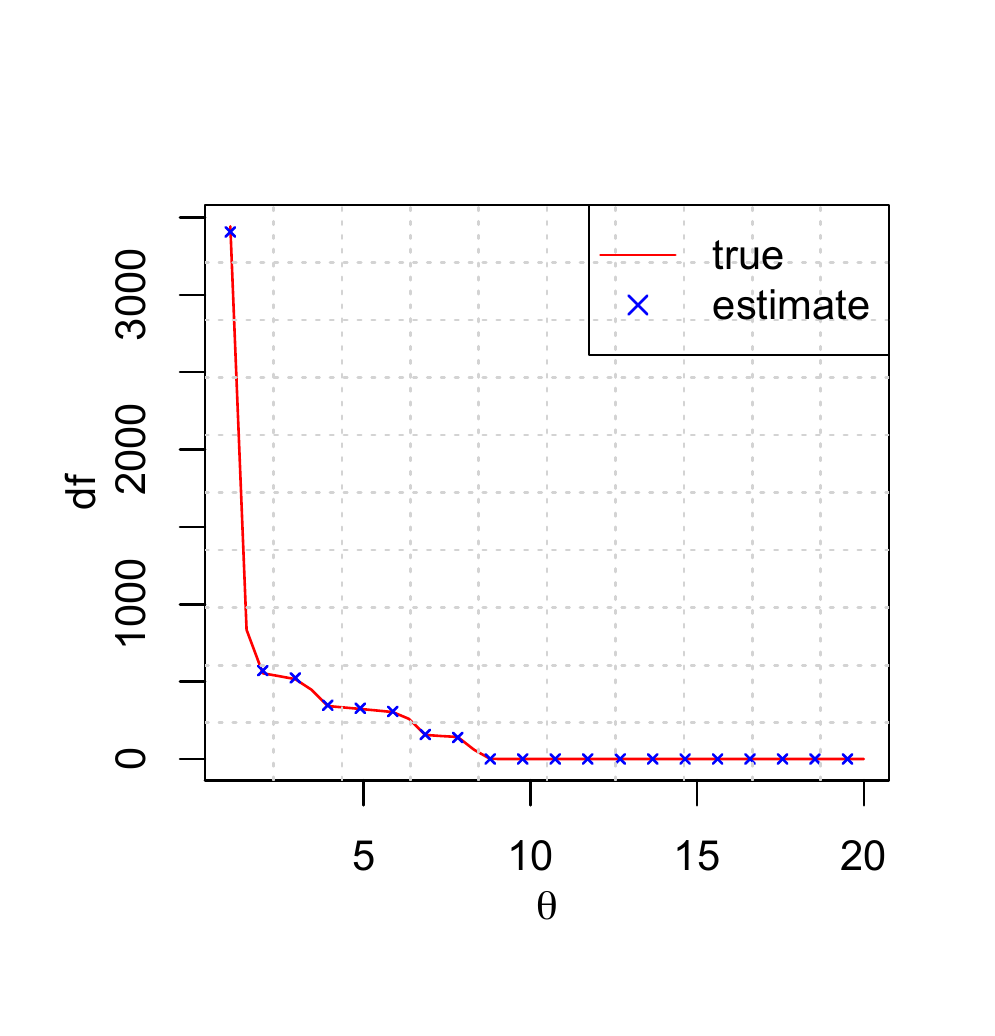} &
             \hspace{-.8cm}
            \includegraphics[scale=0.43]{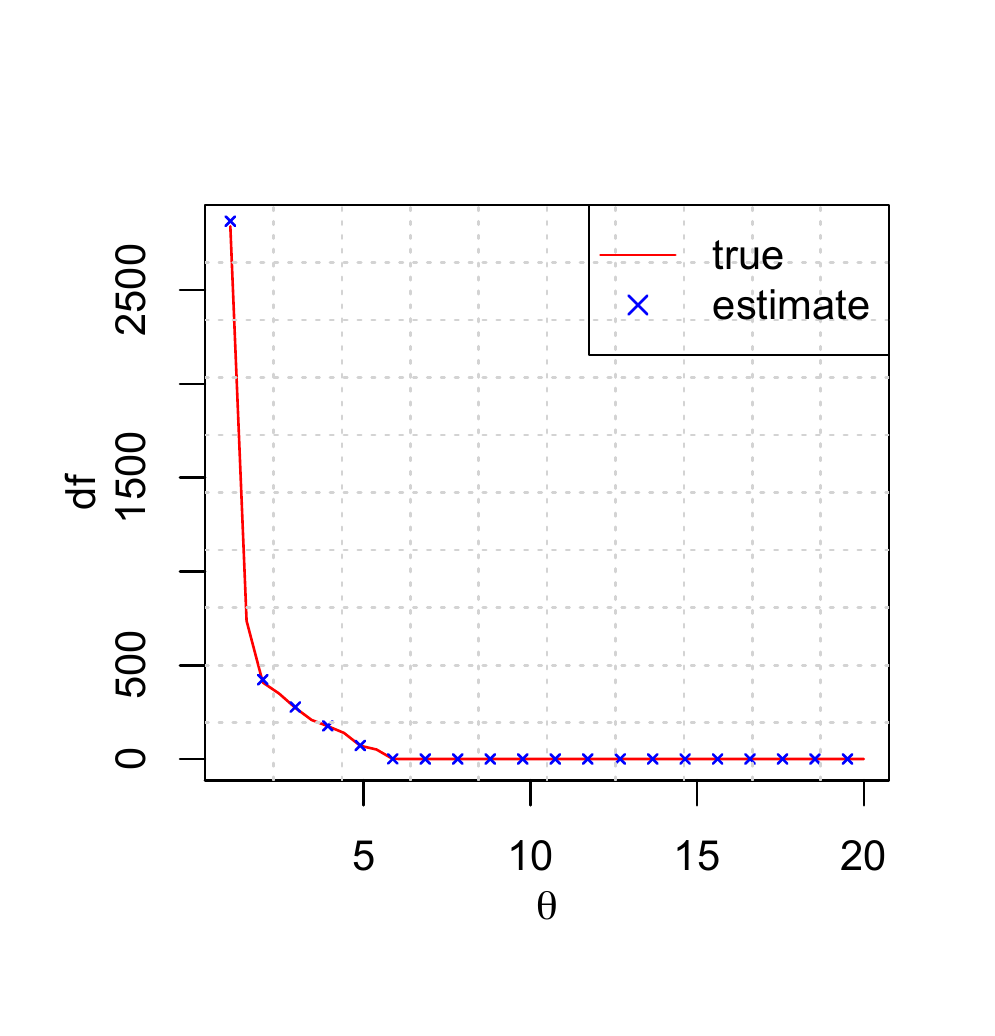}  \vspace{-0.5cm}
        \end{tabular}
        \caption{Degrees of freedom under multivariate linear regression. The true df (red curve) is computed from \eqref{original}. The estimate (blue cross) is the average of the unbiased estimator over 100 repetitions. } \label{df:plot:case2}
    \end{center}
\end{figure}

\begin{figure}[htb!]
    \begin{center}
        \begin{tabular}{ccc}
            \hspace{-.5cm}  SCAD & \hspace{-.8cm} MC+& \hspace{-.8cm} Log   \vspace{-0.6cm} \\
              \hspace{-.5cm}
            \includegraphics[scale=0.43]{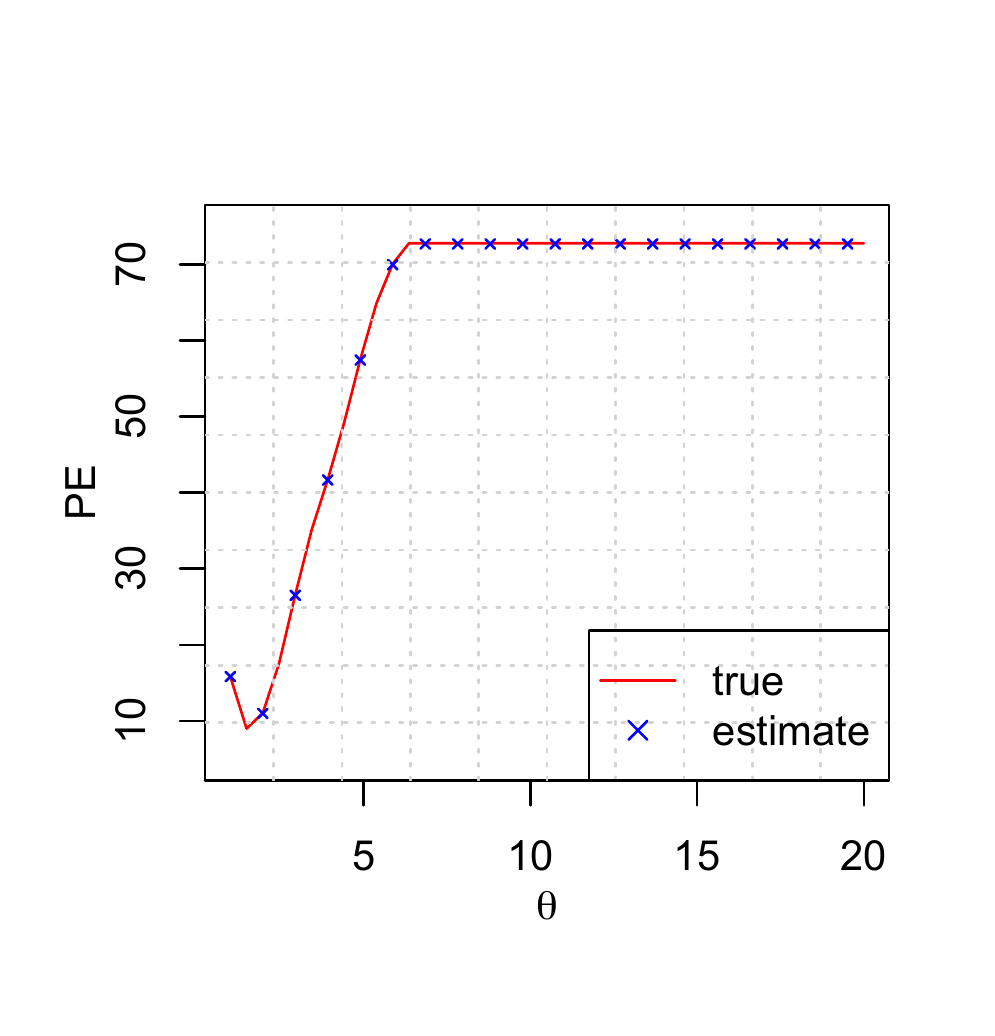} &
             \hspace{-.8cm}
            \includegraphics[scale=0.43]{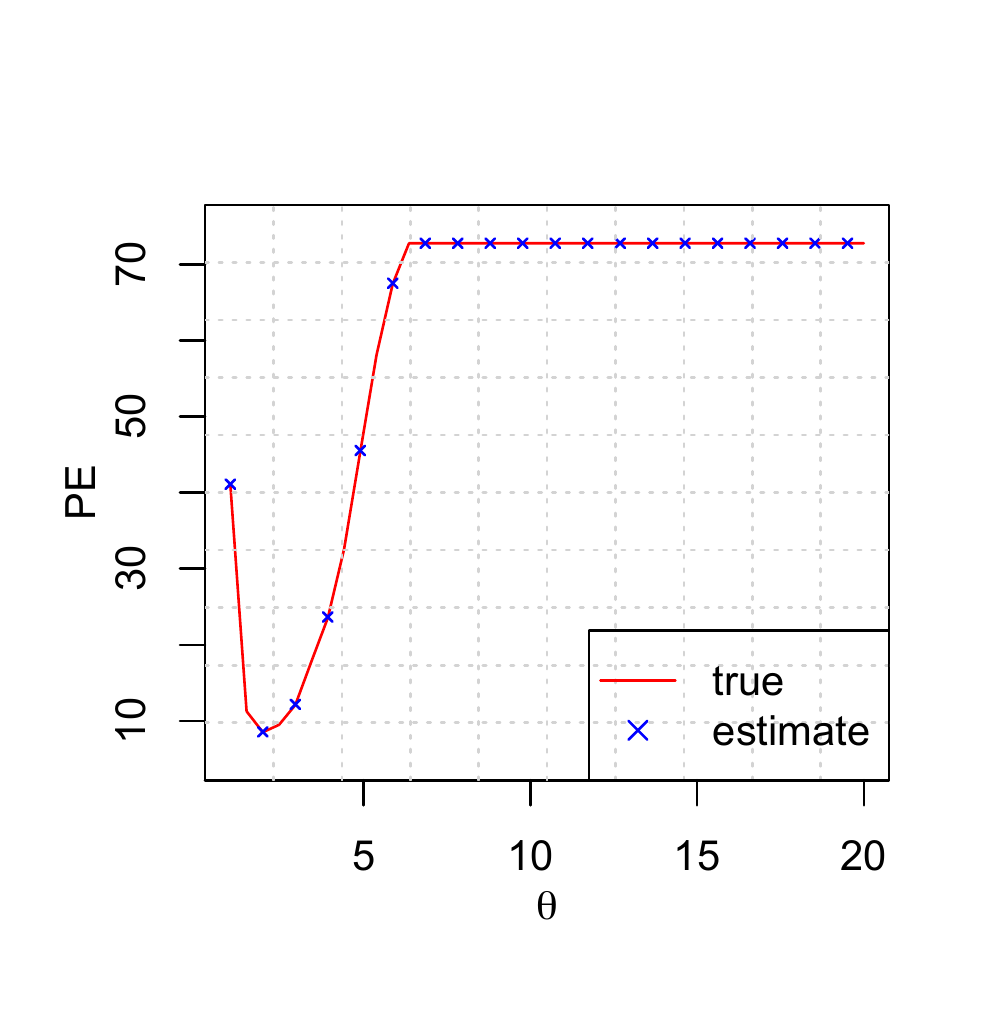} &
             \hspace{-.8cm}
            \includegraphics[scale=0.43]{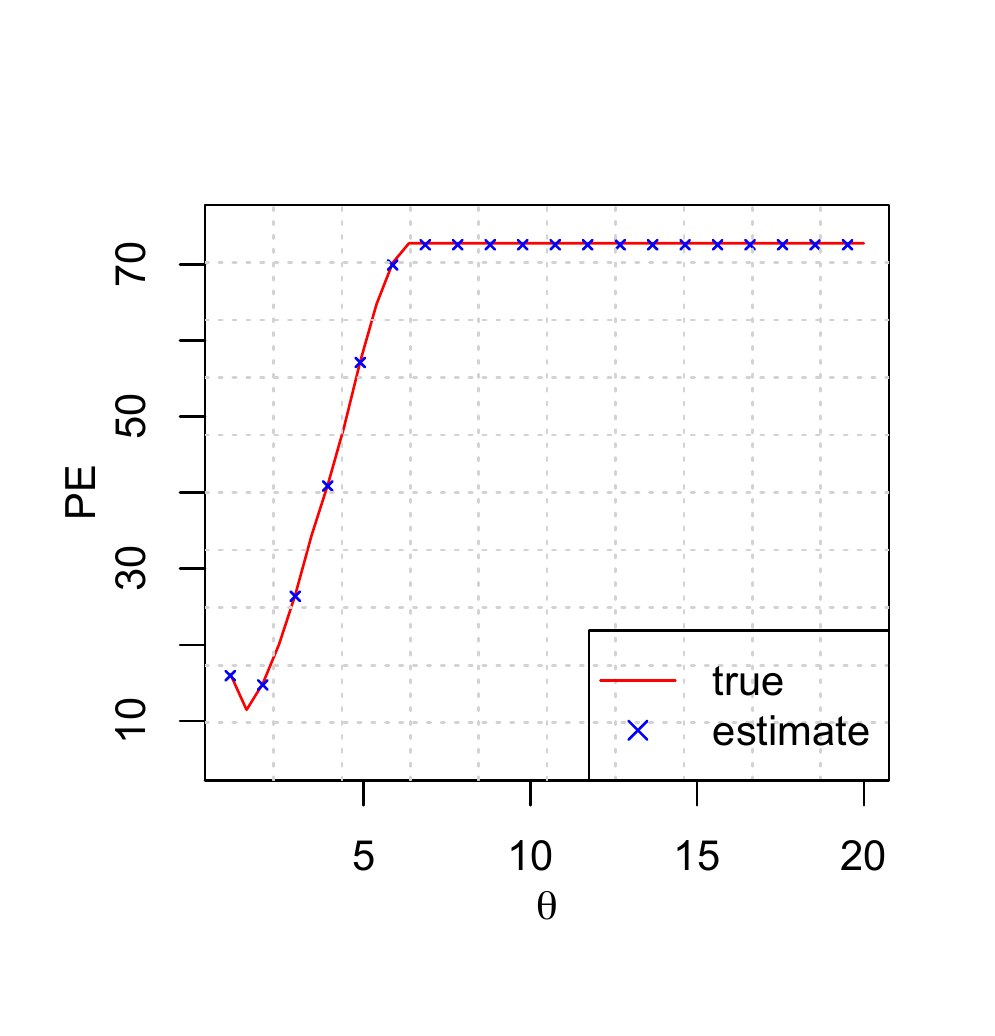}  \vspace{-0.2cm} \\
           \hspace{-.5cm} Bridge ($q=0.1$) &  \hspace{-.8cm} Bridge ($q=0.5$)&   \hspace{-.8cm} Bridge ($q=0.9$)  \vspace{-0.6cm} \\
             \hspace{-.5cm}
            \includegraphics[scale=0.43]{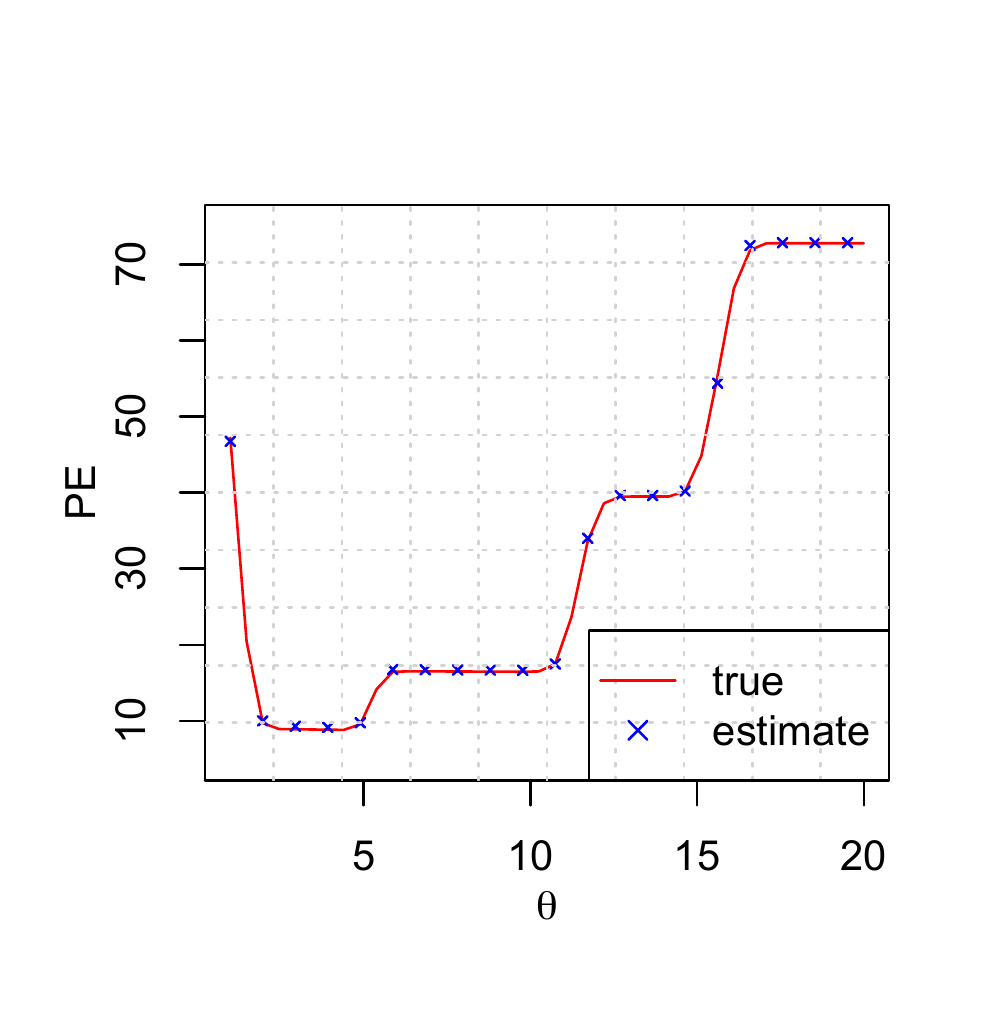} &
             \hspace{-.8cm}
            \includegraphics[scale=0.43]{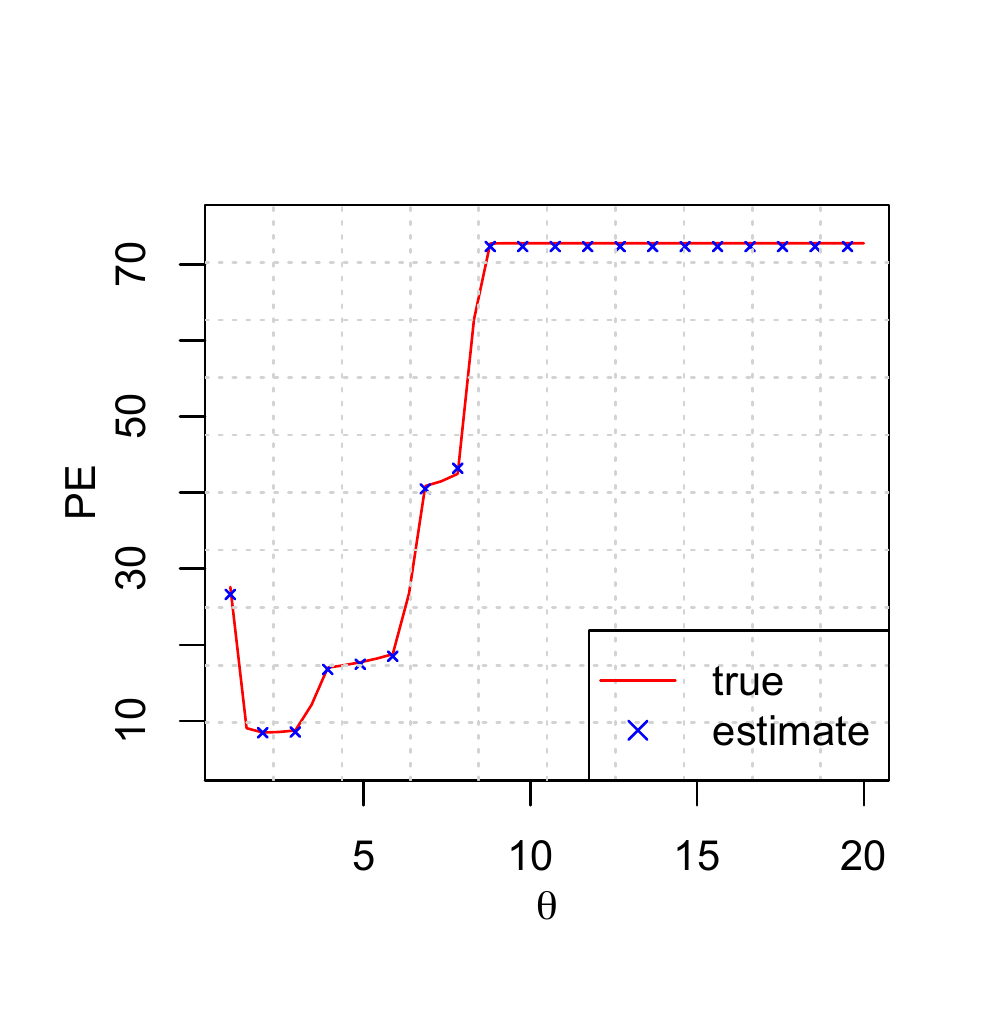} &
             \hspace{-.8cm}
            \includegraphics[scale=0.43]{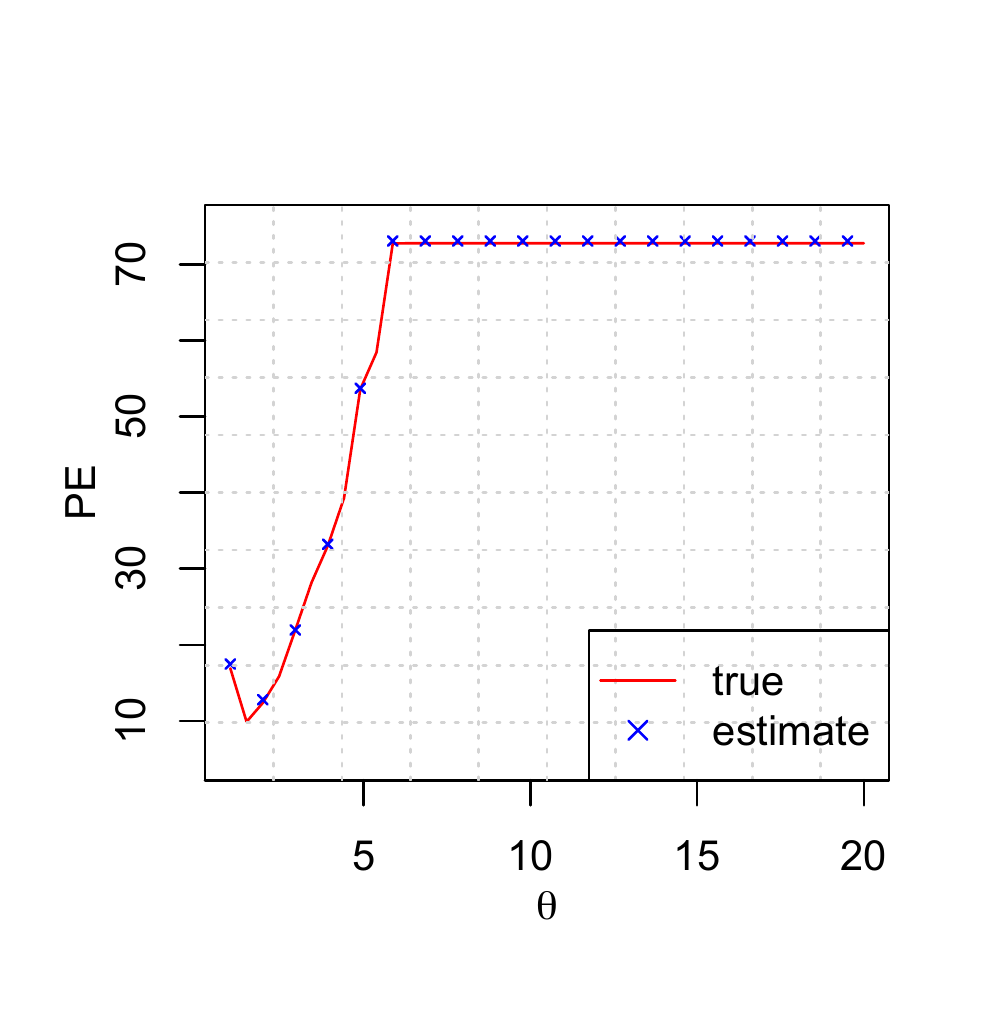}  \vspace{-0.5cm}
        \end{tabular}
        \caption{Prediction error (PE) under multivariate linear regression. The truth (red curve) is computed via monte carlo simulation. The estimate (blue cross) is the average of the unbiased estimator over 100 repetitions. } \label{mse:plot:case2}
    \end{center}
\end{figure}

We generate the data $Y$ according to the multivariate linear regression model \eqref{multimodel}:
\begin{align*}
Y=XM^{\ast}+\mathcal{E},   
\end{align*}
where $Y\in \mathbb{R}^{m\times n}, X\in \mathbb{R}^{m\times p}, M^{\ast}\in \mathbb{R}^{p\times n}$, and $\mathcal{E}=(\epsilon_{ij})_{m\times n}$ with $\epsilon_{ij}\overset{\text{iid}}{\sim}N(0, \tau^2)$. We set $m=300, n=p=100, \tau=0.1$, and use the $M^{\ast}$ from Section \ref{agm:form}. Each row of the design matrix $X$ is independently sampled from $N(\B 0, \B \Sigma)$, where $\B \Sigma$ is a Toeplitz matrix with the $(i,j)$th entry equal to $\frac{1}{2^{|i-j|}m}$ for $1\leq i, j \leq n$. We consider the regularized estimator $RM_{\theta}(Y)$ in \eqref{rmreal} with the same non-convex penalty functions studied in Section \ref{agm:form}. In the current regression setting, the \df of $RM_{\theta}(Y)$ is aligned with the in-sample prediction error $\mathbb{E}\|XM^{\ast}-XRM_{\theta}(Y)\|_2^2$ and defined as
\[
df(RM_{\theta}(Y))=\sum_{ij}\mbox{Cov}((XRM_{\theta}(Y))_{ij},Y_{ij})/\tau^2.
\]
According to Corollaries \ref{final:last0} and \ref{final:last}, we can obtain the estimates for the \df. As in Section \ref{agm:form}, we can also construct the estimates for the prediction error (PE) according to \eqref{unbiased:one}. Figures \ref{df:plot:case2} and \ref{mse:plot:case2} depict the comparison between the estimates and the truth for the \df and PE, respectively. As is clear from the plots, the (averaged) estimates are well matched with the truth. These results empirically validate the \df expressions showed in Section \ref{sec4}.

\section{Conclusion}
\label{discussion:sec}

In this paper, we have presented a systematic study of computing the degrees of freedom for a wide range of low rank matrix estimators, under the SURE framework. As a building block for the computation, the divergence formula for general spectral functions is derived by appealing to a fundamental result on differentiability of matrix functions due to ~\cite{Shapiro2002}. We have put a particular emphasis on the validity of Stein's Lemma. For a class of estimators, we rigorously verify the regularity conditions, invoke the divergence formula, and obtain \df estimates. For other estimators to which Stein's Lemma is not readily applicable, we propose a new Gaussian convolution method and successfully derive their \df expressions. The estimators covered in this paper include the ones studied in the recent literature. For these estimators, our treatment either provides a simpler derivation or complements the existing analysis by a rigorous verification for the use of Stein's Lemma.



\section{Appendix}

This appendix contains the proof of all the main results. The organization is listed below:
\begin{itemize}
\item[1.] Appendix \ref{pf:divergence} proves Corollary \ref{corcor1}.
\item[2.] Appendix \ref{useful:lemma} proves a lemma that is useful in multiple places. 
\item[3.] Appendix \ref{proof:cor2} proves Corollary \ref{cor2}.
\item[4.] Appendix \ref{pf:l0} proves Corollary \ref{cor3add}.
\item[5.] Appendix \ref{pf:lq} proves Corollary \ref{cor3done}.
\item[6.] Appendix \ref{pf:rank} proves Corollary \ref{cor3}.
\item[7.] Appendix \ref{proof:cor6} proves Corollary \ref{cor4}.
\item[8.] Appendix \ref{proof;cor78} proves Corollaries \ref{final:last0} and \ref{final:last}.
\item[9.] Appendix \ref{SURE:condition} reviews the regularity conditions for SURE formulas.
\end{itemize}

\subsection{Proof of Corollary \ref{corcor1}} \label{pf:divergence}

We present a more general result than what appears in Corollary \ref{corcor1}, and prove the general result by making use of Lemmas \ref{shap} and \ref{symmetry}. The proof of Corollary \ref{corcor1} follows as a special case. 

\begin{theorem*} \label{main}
Given a matrix $Y\in \mathbb{R}^{m\times n}$ with singular values $\sigma_1 \geq \ldots \geq \sigma_n$; let $s_1> s_2 >\ldots > s_K \geq 0$ be the set of distinct singular values, and $d_1,\ldots, d_K$ be the associated multiplicities. Consider a function $f : \mathbb{R}^+ \rightarrow \mathbb{R}$ with $f(0)=0$ that is differentiable at every point $s_i$ with $d_i>1$ and directionally differentiable at every point $s_i$ with $d_i=1$. Let $\mathcal{D}$ denote
the set of points where $f$ is directionally differentiable but not differentiable. Then
\begin{eqnarray*}
&&\sum_{i=1}^m\sum_{j=1}^n \frac{\partial [S(Y;f)]_{ij}}{\partial Y_{ij}} \\
= &&\sum_{s_i >0}\left [  \frac{d_i(d_i+1)}{2}f'(s_i)\mathbbm{1}(s_i\notin \mathcal{D} )+ \left((m-n)d_i+\frac{d_i(d_i-1)}{2} \right) \frac{f(s_i)}{s_i}\right ] + \label{thm1} \\
&& d_K(m-n+d_K)  f'(0)\mathbbm{1}(s_K=0)+ \sum_{1\leq i \neq j\leq K} d_id_j\frac{s_if(s_i)-s_jf(s_j)}{s_i^2-s_j^2}+  \nonumber \\
&&\hspace{-0.9cm} \sum_{\substack{ s_k>0 \\ s_k \in \mathcal{D}}} \sum_{i=1}^m\sum_{j=1}^n\left [   u_{ik}^2v_{jk}^2f'(s_k;1)\mathbbm{1}(u_{ik}v_{jk}>0) -u_{ik}^2v_{jk}^2f'(s_k;-1)\mathbbm{1}(u_{ik}v_{jk} < 0) \right ], \nonumber 
\end{eqnarray*}
where $u_{ik}(v_{ik})$ is the $(i,k)th$ entry of the left (right) singular vector matrix $U(V)$ of $Y$.
\end{theorem*}

According to the above theorem, if $f$ is differentiable at some singular value $s_j$, the corresponding singular vectors do not appear in the divergence formula of $S(Y;f)$. Under the conditions of Corollary \ref{corcor1}, $\mathcal{D}=\emptyset$. This directly leads to the formula appearing in Corollary \ref{corcor1}. From the proof to be presented below, we can show a more general result: the directional differentiability of $f$ at singular values of $Y$ is sufficient to guarantee the existence of $\nabla \cdot S(Y;f)$. But since the explicit formula is complicated, we decide to skip it for simplicity.  \\

\begin{proof}
We focus on the more complicated setting when $s_K=0$. The case in which $Y$ is of full rank can be analyzed in the same way. We first assume $f$ is differentiable at every point $s_j, 1\leq j \leq K$. Consider the symmetric matrix in Lemma \ref{symmetry}: it follows that $Y^{\ast}$ has distinct eigenvalues $\pm s_1,\ldots, \pm s_{K-1}, 0$ with multiplicities $d_1, \ldots, d_{K-1}, 2d_K+m-n$. Define a real function $f^{\ast} : \mathbb{R}\rightarrow \mathbb{R}$ as $f^{\ast}(x)=f(x)$ for $x\geq 0$ and $f^{\ast}(x)=-f(-x)$ for $x<0$. Let $F^{\ast}(Y^{\ast})$ be the corresponding matrix valued function stated in Lemma \ref{shap}. The eigenvalue decomposition in Lemma \ref{symmetry} implies a key connection between $F^{\ast}(Y^{\ast})$ and $S(Y;f)$,
\begin{eqnarray}
F^{\ast}(Y^{\ast})=
\begin{bmatrix}
0 & S(Y;f) \\
S(Y;f)' & 0
\end{bmatrix}.   \label{key}
\end{eqnarray}
Let $e_{ij} \in \mathbb{R}^{m\times n}$ be the canonical basis matrix in Euclidean space, i.e., the matrix with all entries equal to $0$ but the $(i,j)$th equal to $1$, and denote
$$
h_{ij}=
\begin{bmatrix}
0&e_{ij}\\
e_{ij}' &0
\end{bmatrix}\,,
$$
\eqref{key} leads to
\begin{eqnarray}
\lim_{t \, \downarrow \, 0}\frac{F^{\ast}(Y^{\ast} + t h_{ij})-F^{\ast}(Y^{\ast})}{t} =
\begin{bmatrix}
0 & \frac{\partial S(Y;f)}{\partial Y_{ij}} \\
 \left (\frac{\partial S(Y;f)}{\partial Y_{ij}} \right )^T & 0
\end{bmatrix}.   \label{key2}
\end{eqnarray}
By the differentiability of $f$ at $s_j, 1 \leq j \leq K$; $f^{\ast}$ is differentiable at all the distinct eigenvalues of $Y^{\ast}$. We can thus apply Lemma \ref{shap} to $F^{\ast}(Y^{\ast})$ with $H=h_{ij}$. After a few algebraic manipulations, it is not hard to obtain\footnote{Note that the second term on the right hand side of Equation \eqref{shapiro} in Lemma \ref{shap} is $\sum_{k=1}^q f'(\mu_k(X))E_kE_k' H E_kE_k'$}
\begin{align}
&\hspace{-0.cm} \sum_{i,j}\frac{\partial [S(Y;f)]_{ij}}{\partial Y_{ij}} = \sum_{i,j}\mbox{tr}\left \{\frac{S(Y;f)}{\partial Y_{ij}} e_{ij}'\right \} \nonumber \\
&=\frac{1}{2} \sum_{i,j} \sum_{l\neq k, l, k=1}^q g_{lk}\{2\mbox{tr}[E_l(1) E_l(2)' e_{ij}' E_k(1) E_k(2)' e_{ij}']+ \nonumber \\
& \mbox{tr}[E_l(1) E_l(1)' e_{ij} E_k(2) E_k(2)' e_{ij}'] +\mbox{tr}[E_k(1) E_k(1)' e_{ij} E_l(2) E_l(2)' e_{ij}']\} + \nonumber \\
& \hspace{-.cm} \sum_{i,j}\sum_{k=1}^q (f^{\ast}(\mu_k))' \{ \mbox{tr} [E_k(1) E_k(2)' e_{ij}' E_k(1) E_k(2)' e_{ij}' ] + \mbox{tr}[E_k(1) E_k(1)' e_{ij} E_k(2) E_k(2)' e_{ij}'] \}  \nonumber   \\
&\triangleq  I + II,   \label{key0}
\end{align}
where $g_{lk}=\frac{f^{\ast}(\mu_l)-f^{\ast}(\mu_k)}{\mu_l-\mu_k}\,$, $q=2K-1$ is the number of unique eigenvalues and  $\,E_k(1),E_k(2)$ are the first $m$ rows and last $n$ rows of the eigenvector matrix $E_k$, respectively. We have used $I, II$ to represent the summations $\frac{1}{2}\sum_{i,j} \sum_{l\neq k, l, k=1}^q(\cdot), \sum_{i,j}\sum_{k=1}^q(\cdot)$, respectively. Let the multiplicity of $\mu_k$ be $r_k, E_k(1)=(\B w^k_1,\ldots, \B w^k_{r_k}), E_k(2)=(\B z^k_1, \ldots, \B z^k_{r_k})$ and $w_1^k(i)$ be the $i$th element of $\B w_1^k$. We then have
\begin{align}
&\hspace{-0.cm} T(\mu_l,\mu_k)\triangleq\sum_{ij} \mbox{tr}[E_l(1)E_l(2)'e'_{ij}E_k(1)E_k(2)'e'_{ij}] = \sum_{ij} \sum_{a=1}^{r_l}\sum_{b=1}^{r_k} \mbox{tr}[\B w^l_a(\B z^l_a)'e'_{ij}\B w^k_b(\B z^k_b)'e'_{ij}] \nonumber \\
&=   \sum_{a=1}^{r_l}\sum_{b=1}^{r_k}  \sum_{ij} z^l_a(j)w^k_b(i)z^k_b(j)w^l_a(i) = \sum_{a=1}^{r_l}\sum_{b=1}^{r_k} [(\B w^k_b)'\B w^l_a] \cdot  [(\B z^k_b)'\B z^l_a] \nonumber \\
&\overset{(a)}{=}
\begin{cases}
0 &    \mbox{~if~} |\mu_k| \neq |\mu_l| \mbox{~or~} |\mu_k\mu_l|=0 \\
 \mbox{sign}(\frac{\mu_k}{\mu_l})\frac{r_k}{4} & \mbox{otherwise},
\end{cases}   \label{form1}
\end{align}
where $(a)$ follows by noting that $\B w_a^k, \B z_a^k$ is one the columns of $(\frac{1}{\sqrt{2}}U,\frac{1}{\sqrt{2}}U, \bar{U})$ and $(\frac{1}{\sqrt{2}}V, \frac{-1}{\sqrt{2}}V, 0)$, respectively, indexed by which one of $\pm s_1,\ldots, \pm s_{K-1}, 0$ that $\mu_k$ achieves. Similarly, we also get
\begin{eqnarray}
G(\mu_l,\mu_k)&\triangleq&\sum_{ij} \mbox{tr}[E_l(1)E_l(1)'e_{ij}E_k(2)E_k(2)'e'_{ij}]  =  \sum_{a=1}^{r_l}\sum_{b=1}^{r_k} ||\B w_a^l ||^2 \cdot || \B z_b^k ||^2\nonumber \\
&=&
\begin{cases}
\frac{r_k r_l}{4} &  \mbox{~if~} \mu_l, \mu_k \neq 0 \\
\frac{r_k(m-n+d_K)}{2} & \mbox{~if~} \mu_l=0, \mu_k \neq 0\\
\frac{r_l d_K}{2} &   \mbox{~if~} \mu_l\neq 0, \mu_k =0 \\
d_K(m-n+d_K) & \mbox{~if~} \mu_l=\mu_k=0.
\end{cases}   \label{form2}
\end{eqnarray}
We now use the results \eqref{form1} and \eqref{form2} to calculate $I$ and $II$ in \eqref{key0}. Recall that $\{\mu_k\}_{k=1}^q=\{\pm s_1,\ldots, \pm s_{K-1}, 0\}$, $\{r_k\}_{k=1}^q=\{d_1,\ldots, d_K\}$ and $f^{\ast}$ is an odd function. It is then not hard to see  $\sum_{l \neq k, l,k=1}^q g_{lk} T(\mu_l,\mu_k)=\frac{-1}{2}\sum_{s_k>0}d_k\frac{f(s_k)}{s_k}$ and
\begin{eqnarray*}
&& \hspace{-0.cm} \frac{1}{2}\sum_{l \neq k, l, k=1}^q g_{lk}(G(\mu_l,\mu_k)+G(\mu_k,\mu_l)) \\
&&=\sum_{\substack{ \mu_l\neq 0, \mu_k\neq 0 \\ l \neq k}}g_{lk}\frac{r_kr_l}{4} + \sum_{\mu_l=0,\mu_k \neq 0}g_{lk}\frac{r_k(m-n+2d_K)}{2}\\
&&= \sum_{s_k>0}\frac{f(s_k)}{s_k}\left(\frac{d_k^2}{2}+d_k(m-n) \right )+\sum_{l\neq k, l, k=1}^{K}d_ld_k\frac{f(s_l)s_l-f(s_k)s_k}{s_l^2-s_k^2}
\end{eqnarray*}
Therefore,
\begin{align}
&I =\frac{1}{2}\sum_{l\neq k, l, k=1}^q g_{lk}(2T(\mu_l,\mu_k)+G(\mu_l,\mu_k)+G(\mu_k,\mu_l))=  \nonumber \\
&\hspace{-0.5cm}  \sum_{s_k>0}\frac{f(s_k)}{s_k}\left(\frac{d_k(d_k-1)}{2}+d_k(m-n) \right )+\sum_{l\neq k, l, k=1}^{K}d_ld_k\frac{f(s_l)s_l-f(s_k)s_k}{s_l^2-s_k^2}. \label{I}
\end{align}
Regarding $II$, it is straightforward to do the computation and obtain,
\begin{eqnarray}
II &=&\sum_{k=1}^q (f^{\ast}(\mu_k))'(T(\mu_k,\mu_k)+G(\mu_k,\mu_k)) \nonumber \\
&=&\sum_{s_k>0}\frac{d_k}{2}f'(s_k)+d_K(m-n+d_K)f'(0)+\sum_{s_k>0}\frac{d_k^2}{2}f'(s_k). \label{II}
\end{eqnarray}
Combining \eqref{key0}, \eqref{I} and \eqref{II} gives the divergence formula. When $f$ is only directionally differentiable at some point $s_k$, the first part $I$ remains the same. For the second part $II$, since the multiplicities of $\pm s_k$ are both $1$, we can simplify the related terms in $II$ as (denote $\mu_a=s_k, \mu_b=-s_k$):
\begin{eqnarray*}
&&\sum_{ij}[E_a'(2)e'_{ij}E_a(1)]^2f^{\ast}(s_k;\mbox{sign}(E_a'(2)e'_{ij}E_a(1)))\cdot \mbox{sign}(E_a'(2)e'_{ij}E_a(1)) + \\
&&\sum_{ij}[E_b'(2)e'_{ij}E_b(1)]^2f^{\ast}(-s_k;\mbox{sign}(E_b'(2)e'_{ij}E_b(1)))\cdot \mbox{sign}(E_b'(2)e'_{ij}E_b(1)) = \\
&&\sum_{ij}u_{ik}^2v_{jk}^2f'(s_k;\mbox{sign}(u_{ik}2v_{jk}))\cdot \mbox{sign}(u_{ik}v_{jk}),
\end{eqnarray*}
where we have used $f^{\ast}(s_k;h)=-f^{\ast}(-s_k;-h), E_a(1)=E_b(1)=\frac{\B u_k}{\sqrt{2}}, E_a(2)=-E_b(2)=\frac{\B v_k}{\sqrt{2}}$. This completes the proof.
\end{proof}

\subsection{A Useful Lemma}
\label{useful:lemma}

We derive a lemma below that will be used multiple times in later proofs.

\begin{lemma}\label{key_key}
Under the canonical additive Gaussian model $Y=M^*+\mathcal{E}$, let the singular values of $Y\in \mathbb{R}^{m \times n}$ be $\sigma_1\geq \ldots \geq \sigma_n\geq 0$, then we have $(1)~ \mathbb{E} \left(\frac{\sigma_i}{\sigma_i-\sigma_j} \right)<\infty, (2)~ \mathbb{E}\left(\frac{1}{\sigma_i-\sigma_j}\right)<\infty, (3)~\mathbb{E}\left(\frac{\sigma_i^2}{\sigma_i^2-\sigma_j^2}\right)<\infty$, where $1\leq i<j\leq n$.
\end{lemma}
\begin{proof}
Firstly, we show the results hold when $M^{\ast}=\B 0$. Let $\lambda_i=\sigma_i^2~(1 \leq i \leq n)$, then $\lambda_1\geq \ldots \geq \lambda_n$ are the eigenvalues of $Y'Y$. The joint distribution $f(\lambda_1,\ldots, \lambda_n)$ of $(\lambda_1, \ldots, \lambda_n)$, i.e., the eigenvalues of a real-valued central Wishart matrix, is known to be \citep{muirhead2009aspects}:
\begin{eqnarray*}
f(\lambda_1,\ldots, \lambda_n)\propto \prod_{a=1}^n\mbox{exp}\left(-\frac{\lambda_a}{2\tau}\right)\cdot \prod_{a=1}^n\lambda_a^{(m-n-1)/2}\cdot \prod_{a<b}(\lambda_a-\lambda_b)
\end{eqnarray*}
Hence, we have
\begin{eqnarray*}
&&\hspace{-0.9cm} \mathbb{E}\frac{\sigma_i}{\sigma_i-\sigma_j} \propto  \idotsint \displaylimits_{\lambda_1\geq \ldots \geq \lambda_n}\frac{\sqrt{\lambda_i}}{\sqrt{\lambda_i}-\sqrt{\lambda_j}}\cdot \prod_{a=1}^n\mbox{exp}\left(-\frac{\lambda_a}{2\tau}\right)\cdot \prod_{a=1}^n\lambda_a^{(m-n-1)/2}\cdot \prod_{a<b}(\lambda_a-\lambda_b)d\B \lambda \\
&\leq& \idotsint \displaylimits_{\lambda_1\geq \ldots \geq \lambda_n}2\lambda_i \prod_{a=1}^n\mbox{exp}\left(-\frac{\lambda_a}{2\tau}\right)\cdot \prod_{a=1}^n\lambda_a^{(m-n-1)/2}\cdot \prod_{\substack{a<b \\ (a,b)\neq (i,j)}}(\lambda_a-\lambda_b)d\B \lambda \\
&\leq& 2\prod_{a=1}^n \int_{0}^{+\infty} \mbox{exp}\left(-\frac{\lambda_a}{2\tau}\right)\cdot \lambda_a^{(m-n-1)/2 + n-a}d\lambda_a <\infty
\end{eqnarray*}
Similarly, we can show $\mathbb{E}1/(\sigma_i-\sigma_j)<\infty$. Moreover, $\mathbb{E}\sigma^2_i/(\sigma_i^2-\sigma_j^2)\leq \mathbb{E}\sigma_i/(\sigma_i-\sigma_j)<\infty$. When $M^{\ast}\neq \B 0$, we express $\sigma_i/(\sigma_i-\sigma_j)$ as a function of $M^{\ast}+\mathcal{E}$, denoted by $h(M^{\ast}+\mathcal{E})$. Then,
\begin{align*}
\mathbb{E}\frac{\sigma_i}{\sigma_i-\sigma_j}&=\frac{1}{(2\pi)^{mn/2}\tau^{mn}}\int h(M^{\ast}+\mathcal{E})\mbox{exp}\left(-\frac{1}{2\tau^2}\|\mathcal{E}\|^2_F \right)d\mathcal{E} \\
&= \frac{1}{(2\pi)^{mn/2}\tau^{mn}}\int h(\mathcal{E})\mbox{exp}\left(-\frac{1}{2\tau^2}\|\mathcal{E}-M^{\ast}\|^2_F \right)d\mathcal{E}  \\
&\hspace{-0.6cm} \leq  \frac{1}{(2\pi)^{mn/2}\tau^{mn}}\cdot \mbox{exp}(\|M^{\ast}\|^2_F/(2\tau^2)) \int h(\mathcal{E})\mbox{exp}\left(-\frac{1}{4\tau^2}\|\mathcal{E}\|^2_F \right)d\mathcal{E} \overset{(a)}{<}\infty, 
\end{align*}
where $(a)$ is implied by the results when $M^{\ast}=0$. Similar arguments work for the other two expectations. 
\end{proof}

\subsection{Proof of Corollary \ref{cor2}}
\label{proof:cor2}

According to Proposition 3 in \citet{mazumder2018matrix}, $S_{\theta}(Y)$ is Lipschitz continuous, which is sufficient for the regularity conditions to hold (see Lemma 3.2 in \citet{CandesEtal2013}). Since $s_{\theta}(\cdot)$ is Lipschitz, it is differentiable almost everywhere. Under the model \eqref{model}, the singular values of $Y$ have a multiplicity of one and are non-zero with probability one. It means that we only need to compute $\nabla \cdot S_{\theta}(Y)$ for the matrix $Y$ of full rank with singular values $\sigma_1> \ldots > \sigma_n >0$ at which $s_{\theta}(\cdot)$ is differentiable. A direct application of Corollary \ref{corcor1} gives the formula in \eqref{dfformula}.

\subsection{Proof of Corollary \ref{cor3add}} \label{pf:l0}

Denote the spectral regularized estimator in expression \eqref{estimator} with $P_{\theta}(\cdot)$ being MC+ penalty functions by $S_{\sqrt{2\theta}, \gamma}(Y)$. Specifically, $S_{\sqrt{2\theta}, \gamma}(Y)=\sum_{i=1}^ng_{\sqrt{2\theta},\gamma}(\sigma_i)\B u_i \B v_i'$, where $g_{\sqrt{2\theta},\gamma}(\cdot)$ is a piecewise linear function defined on $[0, +\infty)$ :
\begin{eqnarray*}
g_{\sqrt{2\theta},\gamma}(\sigma)=
\begin{cases}
0 & \text{if } \sigma \leq \sqrt{2\theta} \\
\frac{\gamma(\sigma -\sqrt{2\theta})}{\gamma-1} &\text{if } \sqrt{2\theta}< \sigma \leq \sqrt{2\theta}\gamma \\
\sigma & \text{if } \sigma >  \sqrt{2\theta}\gamma
\end{cases}
\end{eqnarray*}
Then it is easy to see that $S_{\sqrt{2\theta}, \gamma}(Y) \rightarrow S_{\theta}(Y), $ as $\gamma \downarrow 1$. Hence we have,
\begin{eqnarray*}
|df(S_{\theta}(Y))-df(S_{\sqrt{2\theta},\gamma}(Y))|&=&\frac{1}{\tau^2}\Big |\sum_{ij}\mathbb{E}((S_{\theta}(Y))_{ij}-(S_{\sqrt{2\theta},\gamma}(Y))_{ij})\epsilon_{ij}\Big | \\
&\leq & \frac{1}{\tau^2} \mathbb{E} \|S_{\theta}(Y)-S_{\sqrt{2\theta},\gamma}(Y)\|_F\cdot \|\mathcal{E}\|_F \\
&\leq& \frac{1}{\tau^2} \mathbb{E}^{1/2} \|S_{\theta}(Y)-S_{\sqrt{2\theta},\gamma}(Y)\|^2_F \cdot \mathbb{E}^{1/2} \|\mathcal{E}\|^2_F  \\
&\rightarrow & 0 \quad  \mbox{~as~} \gamma \downarrow 1,
\end{eqnarray*}
where the last line holds by using Dominated Convergence Theorem (DCT). We can apply DCT here because $\|S_{\theta}(Y)-S_{\sqrt{2\theta},\gamma}(Y)\|^2_F \leq 4\|S_{\theta}(Y)\|^2_F$. Therefore, we can calculate $df(S_{\theta}(Y))$ via the following limiting argument,
\begin{eqnarray*}
df(S_{\theta}(Y))=\lim_{\gamma \downarrow 1}df(S_{\sqrt{2\theta},\gamma}(Y)).
\end{eqnarray*}
When $\gamma >1$, $S_{\sqrt{2\theta},\gamma}(Y)$ satisfies the conditions in Corollary \ref{cor2}. Hence, we can get
\begin{eqnarray*}
df(S_{\sqrt{2\theta},\gamma}(Y))&=&  \sum_{i=1}^n \left ( \frac{\gamma}{\gamma-1}P(\sqrt{2\theta}<\sigma_i\leq \sqrt{2\theta}\gamma)+P(\sigma_i >  \sqrt{2\theta} \gamma) \right )   \\
&& + \mathbb{E}\Bigg[ \sum_{\substack {i \neq j\\ i, j=1}}^n\frac{\sigma_i g_{\sqrt{2\theta},\gamma}(\sigma_i)-\sigma_jg_{\sqrt{2\theta},\gamma}(\sigma_j)}{\sigma^2_i-\sigma^2_j}\Bigg ]
\end{eqnarray*}
Now we calculate the limit of each term in the above equation. Let $F_{\sigma_i}(\cdot), f_{\sigma_i}(\cdot)$ be the cdf, pdf of $\sigma_i$ respectively, and $f_{\sigma_i,\sigma_j}(\cdot,\cdot)$ the joint pdf of $(\sigma_i,\sigma_j)$. It is straightforward to see $\lim_{\gamma \downarrow 1}P(\sigma_i >  \sqrt{2\theta} \gamma) =P(\sigma_i>\sqrt{2\theta})$, and 
\begin{align*}
\lim_{\gamma \downarrow 1} \frac{\gamma}{\gamma-1}P(\sqrt{2\theta}<\sigma_i\leq \sqrt{2\theta}\gamma)&=\lim_{\gamma \downarrow 1}\sqrt{2\theta}\gamma \cdot \lim_{\gamma \downarrow 1} \frac{F_{\sigma_i}(\sqrt{2\theta}\gamma)-F_{\sigma_i}(\sqrt{2\theta})}{\sqrt{2\theta}(\gamma-1)} \\
&=\sqrt{2\theta}f_{\sigma_i}(\sqrt{2\theta}).
\end{align*}
Finally, we break $\mathbb{E}\left (\frac{\sigma_i g_{\sqrt{2\theta},\gamma}(\sigma_i)-\sigma_jg_{\sqrt{2\theta},\gamma}(\sigma_j)}{\sigma^2_i-\sigma^2_j}\right )$ into 8 terms,
\begin{eqnarray*}
I_1&=&\mathbb{E} \mathbbm{1}(\sigma_i\leq \sqrt{2\theta}, \sqrt{2\theta} < \sigma_j \leq \sqrt{2\theta}\gamma)\cdot \frac{\gamma \sigma_j(\sigma_j-\sqrt{2\theta})}{(\gamma-1)(\sigma_j^2-\sigma_i^2)}  \\
I_2&=&\mathbb{E} \mathbbm{1}(\sigma_j\leq \sqrt{2\theta}, \sqrt{2\theta} < \sigma_i \leq \sqrt{2\theta}\gamma)\cdot \frac{\gamma \sigma_i(\sigma_i-\sqrt{2\theta})}{(\gamma-1)(\sigma_i^2-\sigma_j^2)} \\
I_3&=&\mathbb{E} \mathbbm{1}(\sigma_i\leq \sqrt{2\theta},  \sigma_j > \sqrt{2\theta}\gamma)\cdot \frac{\sigma_j^2}{\sigma_j^2-\sigma_i^2} \\
I_4&=& \mathbb{E} \mathbbm{1}(\sigma_j\leq \sqrt{2\theta},  \sigma_i > \sqrt{2\theta}\gamma)\cdot \frac{\sigma_i^2}{\sigma_i^2-\sigma_j^2} \\
I_5&=&\mathbb{E} \mathbbm{1}(\sqrt{2\theta} < \sigma_i \leq \sqrt{2\theta}\gamma, \sqrt{2\theta} < \sigma_j \leq \sqrt{2\theta}\gamma)\cdot \frac{\gamma}{\gamma-1} \cdot \Big(1-\frac{\sqrt{2\theta}}{\sigma_i+\sigma_j} \Big) \\
I_6 &=& \mathbb{E} \mathbbm{1}(\sqrt{2\theta} < \sigma_i \leq \sqrt{2\theta}\gamma,  \sigma_j > \sqrt{2\theta}\gamma)\cdot  \Big(1+ \frac{1}{\gamma -1} \cdot \frac{\sigma^2_i-\sqrt{2\theta}\gamma \sigma_i}{\sigma_i^2-\sigma_j^2} \Big) \\
I_7&=& \mathbb{E} \mathbbm{1}(\sqrt{2\theta} < \sigma_j \leq \sqrt{2\theta}\gamma,  \sigma_i > \sqrt{2\theta}\gamma)\cdot  \Big(1+ \frac{1}{\gamma -1} \cdot \frac{\sigma^2_j-\sqrt{2\theta}\gamma \sigma_j}{\sigma_j^2-\sigma_i^2} \Big) \\
I_8&=& \mathbb{E} \mathbbm{1}(\sigma_i> \sqrt{2\theta}\gamma,  \sigma_j > \sqrt{2\theta}\gamma)
\end{eqnarray*}
We then analyze them term by term. First, since $\mathbb{E}1/|\sigma_j-\sigma_i| <\infty$ by Lemma \ref{key_key}, 
\begin{eqnarray*}
I_1 \leq \mathbb{E}\mathbbm{1}(\sqrt{2\theta}<\sigma_j \leq \sqrt{2\theta}\gamma)\cdot \frac{\gamma\sqrt{2\theta}}{|\sigma_j-\sigma_i|} \rightarrow 0, \mbox{~as~}\gamma \downarrow 1.
\end{eqnarray*}
Similarly, we have $\lim_{\gamma \downarrow 1}I_2=0$. Because $\mathbb{E}\sigma^2_j/|\sigma^2_j-\sigma^2_i|<\infty$ according to Lemma \ref{key_key}, we have $\lim_{\gamma \downarrow 1}I_3 =\mathbb{E}\mathbbm{1}(\sigma_i <\sqrt{2\theta}, \sigma_j >\sqrt{2\theta}) \cdot \sigma^2_j /(\sigma_j^2-\sigma_i^2), \lim_{\gamma \downarrow 1}I_4 =\mathbb{E}\mathbbm{1}(\sigma_j <\sqrt{2\theta}, \sigma_i >\sqrt{2\theta}) \cdot \sigma^2_i /(\sigma_i^2-\sigma_j^2)$. Moreover, 
\begin{eqnarray*}
I_5 \leq \frac{\gamma}{\gamma-1}P(\sqrt{2\theta} < \sigma_i \leq \sqrt{2\theta}\gamma, \sqrt{2\theta} < \sigma_j \leq \sqrt{2\theta}\gamma) \rightarrow 0,
\end{eqnarray*}
since $P(\sqrt{2\theta} < \sigma_i \leq \sqrt{2\theta}\gamma, \sqrt{2\theta} < \sigma_j \leq \sqrt{2\theta}\gamma) \sim (\gamma-1)^{2}$. Also,
\begin{eqnarray*}
I_6 \leq P(\sqrt{2\theta}< \sigma_i \leq \sqrt{2\theta} \gamma) + \mathbb{E}\mathbbm{1}(\sqrt{2\theta}< \sigma_i \leq \sqrt{2\theta} \gamma) \cdot \frac{\sqrt{2\theta}}{|\sigma_j-\sigma_i|} \rightarrow 0.
\end{eqnarray*}
Similarly, $\lim_{\gamma \downarrow 0}I_7=0$. Clearly, $\lim_{\gamma \downarrow 1} I_8=P(\sigma_i>\sqrt{2\theta},\sigma_j>\sqrt{2\theta})$. Collecting all the terms we analyzed so far leads to the $\emph{df}$ expression of rank regularized estimator.

\subsection{Proof of Corollary \ref{cor3done}} \label{pf:lq}

According to Lemmas 5--7 in \cite{zheng2017lp}, we can decompose the function $\eta_q(\sigma;\theta)$ over $[0, \infty)$,
\begin{eqnarray*}
\eta_q(\sigma;\theta)=\zeta_q(\sigma;\theta)+\xi_q(\sigma;\theta),
\end{eqnarray*}
where $\zeta_q(\sigma;\theta)=[2(1-q)\theta]^{1/(2-q)}\cdot \mathbbm{1}(\sigma>c_q\theta^{1/(2-q)})$, and $\xi_q(\sigma;\theta)$ is a Lipschitz continuous function. If we define 
\[
\widetilde{S}_{\theta}(Y)=\sum_{i=1}^n\zeta_q(\sigma_i;\theta)\B{u}_i\B{v}'_i, \bar{S}_{\theta}(Y)=\sum_{i=1}^n\xi_q(\sigma_i;\theta)\B{u}_i\B{v}'_i, 
\]
by the definition of \df in \eqref{original} we have
$$df(S_{\theta}(Y))=df(\widetilde{S}_{\theta}(Y))+df(\bar{S}_{\theta}(Y)).$$
Due to the Lipschitz continuity of $\xi_q(\sigma;\theta)$, we can use the same arguments as presented in the proof of Lemma \ref{loc:lip} to conclude that $\bar{S}_{\theta}(Y)$ is Lipschitz continuous. Hence the formula \eqref{alter} is applicable to $\bar{S}_{\theta}(Y)$. Its \df can be computed by the divergence formula in Corollary \ref{corcor1}. Regarding the \df of $\widetilde{S}_{\theta}(Y)$, similar to what we did in the proof of Corollary \ref{cor3add}, we construct a sequence of approximations: $\widetilde{S}_{\theta, h}(Y)=\sum_{i=1}^n g_{\theta,h}(\sigma_i)\B{u}_i\B{v}'_i$, where $g_{\theta,h}$ is a piecewise linear function,
\begin{eqnarray*}
g_{\theta,h}(\sigma)=
\begin{cases}
0 & \text{if } 0\leq \sigma < c_q\theta^{1/(2-q)} \\
\frac{[2(1-q)\theta]^{1/(2-q)}}{h} (\sigma-c_q\theta^{1/(2-q)} ) & \text{if } c_q\theta^{1/(2-q)}\leq \sigma \leq c_q\theta^{1/(2-q)}+h \\
[2(1-q)\theta]^{1/(2-q)} & \text{if } \sigma > c_q\theta^{1/(2-q)}+h
\end{cases}
\end{eqnarray*}
Because $\widetilde{S}_{\theta,h}(Y)$ is Lipschitz, we can compute $df(\widetilde{S}_{\theta,h}(Y))$ with the divergence formula in Corollary \ref{corcor1} and obtain $df(\widetilde{S}_{\theta}(Y))$ by letting $h\downarrow 0$. Since the calculations are very similar to the ones in the proof of Corollary \ref{cor3add}, we do not repeat here. Adding up the \df formulas of $\widetilde{S}_{\theta}(Y)$ and $\bar{S}_{\theta}(Y)$ finishes the proof.

\subsection{Proof of Corollary \ref{cor3}} \label{pf:rank}
 We consider the non-trivial case when $K<n$. The case $K=n$ can be directly verified. Before we go to the the main proof, we prove two useful lemmas that will be applied in the proof. 
\begin{lemma}\label{loc:lip}
For any given $Y_1, Y_2 \in \mathbb{R}^{m \times n}$, denote 
\[
\mathcal{L} \triangleq \max \Big(\frac{\sigma_K(Y_1)}{\sigma_K(Y_1)-\sigma_{K+1}(Y_1)}, \frac{\sigma_K(Y_2)}{\sigma_K(Y_2)-\sigma_{K+1}(Y_2)} \Big). 
\]
We then have
\begin{eqnarray}
\|C_K(Y_1)-C_K(Y_2)\|_F \leq \mathcal{L}\cdot \|Y_1-Y_2\|_F . \label{lip}
\end{eqnarray}
\end{lemma}

\begin{proof}
Let $f_1(\sigma)=\sigma \mathbbm{1}(\sigma\geq \sigma_K(Y_1))$ and $f_2(\sigma)=\sigma \mathbbm{1}(\sigma\geq \sigma_K(Y_2))$. Then
\begin{eqnarray*}
&& \mathcal{L}^2\|Y_1-Y_2\|^2_F-\|C_K(Y_1)-C_K(Y_2)\|^2_F \\
&=&\sum_i [\mathcal{L}^2(\sigma^2_i(Y_1)+\sigma^2_i(Y_2)) -f^2_1(\sigma_i(Y_1))-f^2_2(\sigma_i(Y_2))] \\
&&-2\mathcal{L}^2\mbox{tr}(Y_1'Y_2)+2\mbox{tr}(C'_K(Y_1)C_K(Y_2)) \\
&=&\sum_i [\mathcal{L}^2(\sigma^2_i(Y_1)+\sigma^2_i(Y_2)) -f^2_1(\sigma_i(Y_1))-f^2_2(\sigma_i(Y_2))] \\
&& -2\mbox{tr}[(\mathcal{L}Y_1-C_K(Y_1))'(\mathcal{L}Y_2-C_K(Y_2))] \\
&&-2\mbox{tr}[(\mathcal{L} Y_1-C_K(Y_1))'C_K(Y_2)]-2\mbox{tr}[C_K(Y_1)'(\mathcal{L} Y_2-C_K(Y_2))] \\
&\overset{(a)}{\geq}&\sum_i [\mathcal{L}^2(\sigma^2_i(Y_1)+\sigma^2_i(Y_2)) -f^2_1(\sigma_i(Y_1))-f^2_2(\sigma_i(Y_2))]\\
&&-2\sum_i [\mathcal{L} \sigma_i(Y_1)-f_1(\sigma_i(Y_1))]\cdot [ \mathcal{L} \sigma_i(Y_2)-f_2(\sigma_i(Y_2))] \\
&&-2\sum_i [\mathcal{L} \sigma_i(Y_1)-f_1(\sigma_i(Y_1))]\cdot f_2(\sigma_i(Y_2)) \\
&& -2\sum_i f_1(\sigma_i(Y_1))\cdot [ \mathcal{L} \sigma_i(Y_2)-f_2(\sigma_i(Y_2))]  \\
&\geq& \sum_i [\mathcal{L}^2(\sigma_i(Y_1)-\sigma_i(Y_2))^2-(f_1(\sigma_i(Y_1))-f_2(\sigma_i(Y_2)))^2] \\
&\geq& 0 ,
\end{eqnarray*}
where $(a)$ holds by applying von Newmann's trace inequality \citep{von1937some}. Note that by the way we define $\mathcal{L}$, the sequence $\{\mathcal{L}\sigma_i(Y_1)-f_1(\sigma_i(Y_1))\}_{i=1}^n$ preserves the descending order, so does $\{ \mathcal{L} \sigma_i(Y_2)-f_2(\sigma_i(Y_2))\}_{i=1}^n$. This is the key to derive Inequality $(a)$. 
\end{proof}
\begin{lemma}\label{divergence:cmp}
Given any $Y \in \mathbb{R}^{m\times n}$, if $\sigma_K(Y)>\sigma_{K+1}(Y)$, then $C_K(Y)$ is directionally differentiable at $Y$ and 
\begin{eqnarray}\label{lem5:one}
\mathbb{E}_Z \left( \lim_{h \rightarrow 0+}\frac{([C_K(hZ+Y)]_{ij}-[C_K(Y)]_{ij})Z_{ij}}{h} \right)=\frac{\partial [C_K(Y)]_{ij}}{\partial Y_{ij}},
\end{eqnarray}
where the entries of $Z$ follow i.i.d $N(0,1)$. Moreover, 
\begin{eqnarray}\label{lem5:two}
\sum_{ij}\frac{\partial [C_K(Y)]_{ij}}{\partial Y_{ij}}=(m+n-K)K+2\sum_{i=1}^K\sum_{K+1}^n\frac{\sigma^2_j}{\sigma_i^2-\sigma_j^2}.
\end{eqnarray}
\end{lemma}

\begin{proof}
Construct a function $v:\mathbb{R}^+\rightarrow \mathbb{R}^+$ as
\begin{eqnarray*}
v(x)=
\begin{cases}
0 & \text{if } x\leq \sigma_{K+1}(Y)+\Delta, \\
\frac{(\sigma_K(Y)-\Delta)(x-\sigma_{K+1}(Y)-\Delta)}{\sigma_K(Y)-\sigma_{K+1}(Y)-2\Delta}& \text{if }  \sigma_{K+1}(Y)+\Delta<x\leq \sigma_K(Y)-\Delta ,\\
x & \text{otherwise},
\end{cases}
\end{eqnarray*}
where $\Delta$ is a positive constant smaller than $(\sigma_K(Y)-\sigma_{K+1}(Y)) / 2$. It is straightforward to confirm that $v(0)=0$ and $v(\cdot)$ is differentiable at $\sigma_i(Y), 1\leq i \leq n$. Hence applying Theorem 1 gives
\begin{eqnarray*}
\sum_{ij}\frac{\partial [S(Y;w)]_{ij}}{\partial Y_{ij}}=(m+n-K)K+2\sum_{i=1}^K\sum_{K+1}^n\frac{\sigma^2_j}{\sigma_i^2-\sigma_j^2}.
\end{eqnarray*}
Note that since the singular values $\sigma_i(Y)$ are continuous, we know $C_K(\tilde{Y})=S(\tilde{Y};w)$ for $\tilde{Y}$ in a small neighborhood of $Y$. This fact combined with the last equality proves \eqref{lem5:two}. Regarding \eqref{lem5:one}, since $v(\cdot)$ is differentiable at $\sigma_i(Y), 1\leq i \leq n$, we can combine Lemmas 1 and 2 (as we did in the proof of Theorem 1) to conclude that the directional differential $\lim_{h \rightarrow 0+}\frac{([C_K(hZ+Y)]_{ij}-[C_K(Y)]_{ij})}{h}$ is linear in $Z$. Denote it by $D(Y) \in \mathbb{R}^{m\times n}$. Then we have
\begin{eqnarray*}
\mathbb{E}_Z \lim_{h \rightarrow 0+}\frac{([C_K(hZ+Y)]_{ij}-[C_K(Y)]_{ij})Z_{ij}}{h}= \mathbb{E}_Z [{\rm tr}(Z'D(Y))Z_{ij}]=[D(Y)]_{ij}
\end{eqnarray*}
By the definition of $D(Y)$, we already know ($e_{ij}$ below is the matrix with only its $(i,j)$th entry being non-zero and equal $1$)
\[
[D(Y)]_{ij}=\lim_{h \rightarrow 0+}\frac{([C_K(he_{ij}+Y)]_{ij}-[C_K(Y)]_{ij})}{h}=\frac{\partial [C_K(Y)]_{ij}}{\partial Y_{ij}}.
\]
This completes the proof of \eqref{lem5:one}.
\end{proof}

We now consider a smoothed version of $C_K(Y)$, defined below
\[
g_{h}(Y)\triangleq \mathbb{E}_Z[C_K(Y+hZ)],
\]
where the elements of $Z$ are i.i.d from $N(0,1)$, independent of $Y$; the expectation $\mathbb{E}_Z$ is taken only with respect to $Z$; and $h$ is a positive constant. We would like to show that $g_h(Y)$ is a good approximation to $C_K(Y)$, in terms of calculating degrees of freedom, i.e.,
\begin{eqnarray}\label{justify:one}
\lim_{h \rightarrow 0+} df(g_h(Y)) = df(C_K(Y)).
\end{eqnarray}
To prove \eqref{justify:one}, by using the original definition of \emph{df}, it is sufficient to show for $1\leq i \leq m, 1 \leq j \leq n$,
\begin{eqnarray*}
\lim_{h \rightarrow 0+} \mathbb{E}( [g_h(Y)]_{ij}Y_{ij})=\mathbb{E}( [C_K(Y)]_{ij}Y_{ij}), \quad \lim_{h \rightarrow 0+} \mathbb{E}( [g_h(Y)]_{ij})=\mathbb{E}( [C_K(Y)]_{ij}).
\end{eqnarray*}
We now prove the first equality above and the second one follows the same route of proof. First note that
\begin{eqnarray*}
\mathbb{E}( [g_h(Y)]_{ij}Y_{ij})-\mathbb{E}( [C_K(Y)]_{ij}Y_{ij})=\mathbb{E} (([C_K(Y+hZ)]_{ij}-[C_K(Y)]_{ij})Y_{ij})
\end{eqnarray*}
Since $\|C_K(Y)\|_F \leq \|Y\|_F$ for any $Y \in \mathbb{R}^{m\times n}$, we have for small $h$
\[
|([C_K(Y+hZ)]_{ij}-[C_K(Y)]_{ij})Y_{ij}| \leq (\|Z\|_F+2\|Y\|_F)\|Y\|_F.
\]
Hence we can use Dominated Convergence Theorem (DCT) to conclude
\begin{eqnarray*}
&&\lim_{h \rightarrow 0+} \mathbb{E} (([C_K(Y+hZ)]_{ij}-[C_K(Y)]_{ij})Y_{ij}) \\
&&=\mathbb{E} \lim_{h \rightarrow 0+} (([C_K(Y+hZ)]_{ij}-[C_K(Y)]_{ij})Y_{ij})\overset{(b)}{=}0.
\end{eqnarray*}
To derive $(b)$ we have used the fact that $C_K(Y)$ is directionally differentiable from Lemma \ref{divergence:cmp}. Based on \eqref{justify:one}, we can compute $df(C_K(Y))$ by first calculating $df(g_h(Y))$ and then letting $h$ goes to zero. Since $g_h(Y)$ is differentiable, it is straightforward to get
\begin{eqnarray}\label{smooth:df}
df(g_h(Y))&=&  \sum_{ij} \mathbb{E} \Big(\frac{\partial [g_h(Y)]_{ij}}{\partial Y_{ij}} \Big)= \sum_{ij} \mathbb{E}\Big (  \frac{[C_K(hZ+Y)]_{ij} Z_{ij }}{h} \Big)  \nonumber \\
&\overset{(c)}{=}& \sum_{ij} \mathbb{E}\Big (  \frac{([C_K(hZ+Y)]_{ij} -[C_K(Y)]_{ij})Z_{ij }}{h} \Big),
\end{eqnarray}
where $(c)$ holds because $Z$ is independent of $Y$ and has zero mean. We aim to calculate the following limits:
\begin{eqnarray}\label{limits:one}
&& \lim_{h \rightarrow 0+} \mathbb{E}\Big (  \frac{([C_K(hZ+Y)]_{ij} -[C_K(Y)]_{ij})Z_{ij }}{h} \Big)=  \nonumber \\
&& \hspace{-0.9cm} \lim_{h \rightarrow 0+} \mathbb{E}_Z  \underbrace{\int  \frac{([C_K(hZ+Y)]_{ij} -[C_K(Y)]_{ij})Z_{ij }}{h} \frac{1}{(\sqrt{2\pi}\tau)^{mn}}{\rm exp}\Big(\frac{\|Y-M^*\|^2_F}{-2\tau^2} \Big)dY}_{\triangleq J(Z, h)} \nonumber \\
\end{eqnarray}
According to Lemma \ref{loc:lip}, we can obtain
\begin{eqnarray}\label{firstdct:one}
\hspace{-1cm}
 &&|J(Z,h)|  \leq   \|Z \|_F|Z_{ij}| \cdot  \mathbb{E}_Y \frac{\sigma_K(Y)}{\sigma_K(Y)-\sigma_{K+1}(Y)} + \nonumber \\
 && \|Z \|_F|Z_{ij}|  \cdot \mathbb{E}_Y \frac{\sigma_K(hZ+Y)}{\sigma_K(hZ+Y)-\sigma_{K+1}(hZ+Y)}  
\end{eqnarray}
Moreover, a simple change of variable gives us
\begin{align}
& \hspace{0.3cm}\mathbb{E}_Y \frac{\sigma_K(hZ+Y)}{\sigma_K(hZ+Y)-\sigma_{K+1}(hZ+Y)} =  \nonumber \\
& \int \frac{\sigma_K(Y)}{\sigma_K(Y)-\sigma_{K+1}(Y)} \frac{1}{(\sqrt{2\pi}\tau)^{mn}}{\rm exp}\Big(\frac{\|Y-hZ-M^*\|^2_F}{-2\tau^2} \Big)dY  \nonumber \\
&\leq  \frac{1}{(\sqrt{2\pi}\tau)^{mn}}{\rm exp}\Big(\frac{\|hZ+M^*\|^2_F}{2\tau^2}\Big)\int \frac{\sigma_K(Y)}{\sigma_K(Y)-\sigma_{K+1}(Y)}  {\rm exp}\Big(\frac{\|Y\|^2_F}{-4\tau^2} \Big)dY \nonumber \\
&\leq \frac{1}{(\sqrt{2\pi}\tau)^{mn}}{\rm exp}\Big(\frac{\|M^*\|^2_F}{\tau^2}\Big)\int \frac{\sigma_K(Y)}{\sigma_K(Y)-\sigma_{K+1}(Y)}  {\rm exp}\Big(\frac{\|Y\|^2_F}{-4\tau^2} \Big)dY \cdot {\rm exp}\Big(\frac{h^2\|Z\|^2_F}{\tau^2}\Big)  \nonumber \\
\label{firstdct:two}
\end{align}
Combining Lemma \ref{key_key} part (1) with \eqref{firstdct:one} and \eqref{firstdct:two}, we can conclude that for sufficiently small $h$, there exists an upper bound on $J(Z,h)$ that is independent of $h$ and is integrable. We thus can employ DCT to get
\begin{eqnarray}\label{limits:two}
\lim_{h \rightarrow 0+} \mathbb{E}_Z [J(Z,h)] =\mathbb{E}_Z \lim_{h \rightarrow 0+}  [J(Z,h)].
\end{eqnarray}
We next focus on calculating $\lim_{h \rightarrow 0+}  [J(Z,h)]$. We decompose $J(Z,h)$ into two parts:
\begin{align}\label{two:parts}
& J(Z,h)= \nonumber \\
&\hspace{-0.2cm} \mathbb{E}_Y \underbrace{\Bigg[ \frac{([C_K(hZ+Y)]_{ij} -[C_K(Y)]_{ij})Z_{ij }}{h} \mathbbm{1}(\sigma_K(Y) \geq h^{2/3}, \sigma_K(Y)-\sigma_{K+1}(Y)\geq h^{2/3})\Bigg]}_{\triangleq H_1(Y, Z,h)}+ \nonumber \\
&\hspace{-0.2cm} \mathbb{E}_Y \underbrace{\Bigg[ \frac{([C_K(hZ+Y)]_{ij} -[C_K(Y)]_{ij})Z_{ij }}{h} \mathbbm{1}(\sigma_K(Y) \leq h^{2/3} \mbox{~or~} \sigma_K(Y)-\sigma_{K+1}(Y)\leq h^{2/3})\Bigg]}_{\triangleq H_2(Y, Z,h) },\nonumber \\
\end{align}
and analyze $H_1(Y,Z,h), H_2(Y, Z,h)$ separately. Regarding $H_1(Y, Z,h)$, first note that according to Weyl's inequality \citep{stewart1998perturbation}, we know
\[
|\sigma_i(hZ+Y)-\sigma_{i}(Y)| \leq h \|Z\|_F, \quad 1 \leq i \leq n
\]
Therefore, on the event $\{\sigma_K(Y) \geq h^{2/3}, \sigma_K(Y)-\sigma_{K+1}(Y)\geq h^{2/3}\}$, it is not hard to show that when $h$ is sufficiently small, 
\[
\frac{\sigma_K(hZ+Y)}{\sigma_K(hZ+Y)-\sigma_{K+1}(hZ+Y)} \leq \frac{4\sigma_K(Y)}{\sigma_K(Y)-\sigma_{K+1}(Y)},
\]
We can then employ Lemma \ref{loc:lip} to obtain, 
\begin{eqnarray*}
|H_1(Y, Z,h)| \leq \|Z\|_F|Z_{ij}|\frac{4\sigma_K(Y)}{\sigma_K(Y)-\sigma_{K+1}(Y)}.
\end{eqnarray*}
This enables us to apply DCT to derive
\begin{eqnarray}\label{H1:dct}
\mathbb{E}_Z \lim_{h \rightarrow 0+} \mathbb{E}_YH_1(Y,Z,h)= \mathbb{E} \lim_{h \rightarrow 0+} H_1(Y,Z,h)\overset{(d)}{=}\mathbb{E}_Y\frac{\partial [C_K(Y)]_{ij}}{\partial Y_{ij}},
\end{eqnarray}
where $(d)$ is due to Lemma \ref{divergence:cmp}. For $H_2(Y,Z,h)$, we have
\begin{align}\label{H2:decomp}
& |\mathbb{E}_YH_2(Y,Z,h)|  \nonumber \\
\leq &\mathbb{E}_Y[|Z_{ij}| \cdot |(\|Z\|_F+2\|Y\|_F/h) \cdot  \mathbbm{1}(\sigma_K(Y) \leq h^{2/3} \mbox{~or~} \sigma_K(Y)-\sigma_{K+1}(Y)\leq h^{2/3}) ] \nonumber \\
\overset{(e)}{\leq}&  |Z_{ij}|\cdot \|Z\|_F \cdot P(\sigma_K(Y)-\sigma_{K+1}(Y)\leq h^{2/3}) + \nonumber \\
&\hspace{0.4cm} 2|Z_{ij}| \cdot (\mathbb{E}\|Y\|_F^{7})^{1/7}\cdot \Big(\frac{P(\sigma_K(Y)-\sigma_{K+1}(Y)\leq h^{2/3})}{h^{7/6}} \Big)^{6/7}
\end{align}
We have used H$\ddot{o}$lder's inequality to derive $(e)$. Clearly the first term of the upper bound above vanishes as $h \rightarrow 0+$. We now show the second term goes to zero as well. For simplicity, we only show it for $M^*=0$. The general case $M^*\neq 0$ can be proved by the same arguments as presented in the proof of Lemma \ref{key_key}. We hence skip it here. Similar to the proof in Lemma \ref{key_key}, let $\lambda_i=\sigma^2_i(Y), 1 \leq i \leq n$ and denote the joint distribution of $(\lambda_1,\ldots, \lambda_n)$ by $f(\lambda_1,\ldots,\lambda_n)$. We can then rewrite 
\begin{align}\label{final:key}
& P(\sigma_K(Y)-\sigma_{K+1}(Y)\leq h^{2/3}) \nonumber \\
=&\idotsint \displaylimits_{\lambda_1\geq \ldots \geq \lambda_n\geq 0}f(\lambda_1,\ldots, \lambda_n)\cdot \mathbbm{1}(\sqrt{\lambda_K}-\sqrt{\lambda_{K+1}}\leq h^{2/3}) d\B \lambda \nonumber \\
\propto&  \idotsint \displaylimits_{\lambda_1\geq \ldots \geq \lambda_n \geq 0}  \mathbbm{1}(\sqrt{\lambda_K}-\sqrt{\lambda_{K+1}}\leq h^{2/3}) \prod_{a=1}^ne^{\frac{\lambda_a}{-2\tau}}\cdot \prod_{a=1}^n\lambda_a^{(m-n-1)/2}\cdot \prod_{a<b}(\lambda_a-\lambda_b)d\B \lambda \nonumber \\
 \overset{(f)}{\leq}& \idotsint \displaylimits_{\lambda_1\geq \ldots \geq \lambda_n \geq 0}  \mathbbm{1}(\sqrt{\lambda_K}-\sqrt{\lambda_{K+1}}\leq h^{2/3}) \prod_{a=1}^ne^{\frac{\lambda_a}{-2\tau}}\cdot \prod_{a \neq K}^n\lambda_a^{(m+n-1)/2-a} \cdot \lambda_K^{(m+n-3)/2-K} (\lambda_K-\lambda_{K+1})d\B \lambda \nonumber \\
\overset{(g)}{\leq}&\hspace{-0.5cm} \underbrace{ \iint \displaylimits_{0\leq  \lambda_{K+1}\leq \lambda_K} \mathbbm{1}(\sqrt{\lambda_K}-\sqrt{\lambda_{K+1}}\leq h^{2/3}) e^{\frac{\lambda_K+\lambda_{K+1}}{-2\tau}}(\lambda_K\lambda_{K+1})^{(m+n-3)/2-K}(\lambda_K-\lambda_{K+1})d\lambda_Kd\lambda_{K+1}}_{\triangleq Q(h)} \nonumber \\
& \cdot \Big[\prod_{a \neq K,K+1}^n \int_0^{\infty}e^{\frac{\lambda_a}{-2\tau}}\lambda_a^{(m-n-1)/2+n-a}d\lambda_a\Big],
\end{align}
where $(f)$ is obtained by using $\lambda_a-\lambda_b \leq \lambda_a$, for $a<b$; $(g)$ holds simply because we enlarge the set that is integrated over. We easily see that the second term on the right hand side of the last inequality is finite and independent of $h$. For the first term $Q(h)$, by using $\lambda_K-\lambda_{K+1}\leq 2\sqrt{\lambda_K}(\sqrt{\lambda_K}-\sqrt{\lambda_{K+1}})$. We have
\begin{eqnarray}
&&\hspace{-0.8cm} Q(h) \leq 2h^{2/3}\iint \displaylimits_{0\leq  \lambda_{K+1}\leq \lambda_K} \mathbbm{1}(\sqrt{\lambda_K}-\sqrt{\lambda_{K+1}}\leq h^{2/3}) e^{\frac{\lambda_K+\lambda_{K+1}}{-2\tau}}\lambda_K^{(m+n)/2-K-1} \cdot \nonumber \\
&&\hspace{5.8cm} \lambda_{K+1}^{(m+n-3)/2-K}d\lambda_Kd\lambda_{K+1} \nonumber \\
&&= 2h^{2/3} \int_0^{\infty}\Big[\int_{\lambda_{K+1}}^{(\sqrt{\lambda_{K+1}}+h^{2/3})^2}e^{\frac{\lambda_K}{-2\tau}}\lambda_K^{(m-n)/2+n-K-1}d\lambda_K \Big] \cdot \nonumber \\
&& \hspace{5cm} e^{\frac{\lambda_{K+1}}{-2\tau}}\lambda_{K+1}^{(m-n-1)/2+n-K-1}d\lambda_{K+1} \label{double:integ} \\
&&\overset{(h)}{=}O(h^{4/3}), \label{vanish:fin}
\end{eqnarray}
where $(h)$ can be derived by using mean value theorem for the inside integral in \eqref{double:integ}. Combining \eqref{H2:decomp}, \eqref{final:key} and \eqref{vanish:fin} together gives us
\begin{eqnarray}\label{H2:answer}
\lim_{h \rightarrow 0+} \mathbb{E}_YH_2(Y,Z,h)=0.
\end{eqnarray}
Collecting the results from \eqref{justify:one}, \eqref{smooth:df}, \eqref{limits:one}, \eqref{limits:two}, \eqref{two:parts}, \eqref{H1:dct} and \eqref{H2:answer}, we can finally conclude
\begin{eqnarray*}
df(C_K(Y)) &=&\mathbb{E}  \sum_{ij} \frac{\partial [C_K(Y)]_{ij}}{\partial Y_{ij}}.
\end{eqnarray*}
A direct application of Equation \eqref{lem5:two} from Lemma \ref{divergence:cmp} completes the proof.

\subsection{Proof of Corollary \ref{cor4}}
\label{proof:cor6}

Denote the compact SVD of $X$ by $X=U\Sigma V'$. We construct an ancillary matrix $Q=U'Y$, which is the response matrix of the following additive model,
\begin{eqnarray}
Q=\tilde{M}+\tilde{\mathcal{E}},   \label{new}
\end{eqnarray}
where $\tilde{M}=U'XM^{\ast}, \tilde{\mathcal{E}}=U'\mathcal{E}$. Due to the orthogonality of the columns of $U$, the entries of $\tilde{\mathcal{E}}$, i.e., $\tilde{\epsilon}_{ij}\overset{iid}{\sim}N(0,\tau^2)$. We now relate $df(M_K(Y))$ under the model \eqref{multimodel} to $df(C_K(Q))$ under the model \eqref{new}. A key observation is 
\[
XM_K(Y)=C_K(UU'Y)=UC_K(Q).
\]
It follows that
\begin{eqnarray*}
&&\hspace{-0.7cm}\mathbb{E} \left(\sum_{ i=1}^m\sum_{j=1}^n(XM_K(Y))_{ij} \epsilon_{ij}\right)=\mathbb{E}\left(\sum_{ i=1}^m\sum_{j=1}^n\sum_{k=1}^ru_{ik}(C_K(Q))_{kj}\epsilon_{ij}\right), \\
&&\hspace{-0.6cm} \mathbb{E}\left( \sum_{a=1}^r\sum_{b=1}^n(C_K(Q))_{ab}\tilde{\epsilon}_{ab}\right)=\mathbb{E}\left(\sum_{a=1}^r\sum_{b=1}^n\sum_{l=1}^m(C_K(Q))_{ab}u_{la}\epsilon_{lb}\right),
\end{eqnarray*}
where $u_{ik}$ is the $(i, k)$th entry of $U$ and $(C_K(Q))_{ab}$ is the $(a,b)$th element of $C_K(Q)$. Arranging the notation $a=k, b=j, l=i$, we thus obtain $df(M_K(Y))=df(C_K(Q))$. Given that $Q$ and $\hat{Y}$ share the same singular values, a direct use of Corollary \ref{cor3} for $C_K(Q)$ gives us the \df formula for $M_K(Y)$.

\subsection{Proof of Corollaries \ref{final:last0} and \ref{final:last}}
\label{proof;cor78}

Observe that 
\[
X RM_{\theta}(Y)=S_{\theta}(UU'Y)=US_{\theta}(U'Y). 
\]
Thus we can use the same arguments as in the proof of Corollary \ref{cor4} to obtain 
\[
df(RM_{\theta}(Y))=df(S_{\theta}(U'Y)). 
\]
Then applying Corollaries \ref{cor2} and \ref{cor3done} complete the proof of Corollaries \ref{final:last0} and \ref{final:last}, respectively.

\subsection{Stein's Unbiased Risk Estimate}
\label{SURE:condition}

\begin{proposition*}
\cite{stein1981estimation,efron2004least, johnstone2017gaussian} Suppose $\M{y}\sim N(\B{\mu}, \tau^2\M{I}_n)$, $\B{h}: \mathbb{R}^n\rightarrow \mathbb{R}^n$ is weakly differentiable, and $\mathbb{E}|y_ih_i(\M{y})|+\mathbb{E}|\frac{\partial h_i(\M{y})}{\partial y_i}|<\infty$ for $i=1,\ldots, n$. Then
\begin{align*}
df(\B{h}(\M{y}))&=\mathbb{E}\Big(\sum_{i=1}^n \partial h_i(\M{y})/ \partial y_i \Big), \\
\mathbb{E}\|\B{h}(\M{y})-\B{\mu}\|_2^2&=\mathbb{E}\Big[-\tau^2n+\|\B{h}(\M{y})-\M{y}\|_2^2+2\tau^2 \cdot \sum_{i=1}^n \partial h_i(\M{y})/ \partial y_i \Big].
\end{align*}
A function $g: \mathbb{R}^n\rightarrow \mathbb{R}$ is said to be weakly differentiable if there exist functions $f_i:\mathbb{R}^n\rightarrow \mathbb{R}, i=1,\ldots, n$, such that for all compactly supported and infinitely differentiable functions $\varphi$,
\begin{align*}
\int \varphi(\B{z})h(\B{z})d\B{z}=-\int \frac{\partial \varphi(\B{z})}{\partial z_i}g(\B{z})d\B{z}.
\end{align*}
\end{proposition*}

\bibliography{mc}
\bibliographystyle{ims}

\end{document}